\documentclass[11pt, letter]{article}
\usepackage[utf8]{inputenc}
\usepackage[linesnumbered,vlined,ruled]{algorithm2e}
\usepackage{amsfonts}
\usepackage{amsmath}
\usepackage{amssymb}
\usepackage{amsthm}
\usepackage{geometry}
\usepackage{graphicx}
\usepackage{mathrsfs}
\usepackage{multicol} 
\usepackage{multirow}
\usepackage{paralist}
\usepackage{pgfplots}
\usepackage{setspace}
\usepackage{soul}
\usepackage{subcaption}
\usepackage{tikz}
\usepackage{verbatim}

\definecolor{DRed}{rgb}{0,0,0}

\definecolor{Red}{rgb}{0,0,0}
\newcommand{\Red}{\color{Red}}
\definecolor{Blue}{rgb}{0,0,0}

\definecolor{Green}{rgb}{0,0,0}

\definecolor{Yellow}{rgb}{0,0,0}
\newcommand{\Yellow}{\color{Yellow}}

\allowdisplaybreaks[4]

\newtheorem{example}{Example}[section]
\newtheorem{theorem}{Theorem}[section]
\newtheorem{corollary}{Corollary}[section]

\newtheorem{proposition}{Proposition}[section]
\newtheorem{remark}{Remark}[section]
\newtheorem{lemma}{Lemma}[section]

\newtheorem{assumption}{Assumption}

\newtheorem{ntc}{Need To Correct}

\title{\textbf{\textsc{Optimal Kernel Estimation of Spot Volatility of Stochastic Differential Equations}}}

\author{
Jos\'{e} E. Figueroa-L\'{o}pez
\thanks{Department of Mathematics, Washington University in St. Louis, St. Louis, MO 63130, USA.
{\tt figueroa@math.wustl.edu}. Research supported in part by the NSF Grants: DMS-1149692 and DMS-1613016.}
\and
Cheng Li
\thanks{Department of Statistics, Purdue University, West Lafayette, IN, 47907, USA.
{\tt li1534@purdue.edu}.}
}

\begin{document}
\maketitle

\abstract{
	Kernel Estimation is one of the most widely used estimation methods in non-parametric Statistics, having a wide-range of applications, including spot volatility estimation of stochastic processes. The selection of bandwidth and kernel function is of great importance, especially for the finite sample settings commonly encountered in econometric applications. In the context of spot volatility estimation, most of the proposed selection methods are either largely heuristic (e.g., cross-validation methods) or just formally stated without any feasible implementation (e.g., plug-in methods). In this work, an objective method of bandwidth and kernel selection is proposed, under some {mild} conditions on the volatility, which not only cover classical Brownian {motion} driven dynamics but also some processes driven by long-memory fractional Brownian motions or other Gaussian processes.
	Under such {a unifying framework}, we characterize the leading order terms of the Mean Squared Error, which are also ratified by central limit theorems for the estimation error. As a byproduct, an  approximated optimal bandwidth is then obtained in closed form, which is shown to be asymptotically equivalent to the true optimal bandwidth. This result allows us to develop a feasible plug-in type bandwidth selection procedure, for which, as a sub-problem, we propose a new estimator of the volatility of volatility. The optimal selection of kernel function is also discussed. For Brownian Motion type volatilities, the optimal kernel function is proved to be the exponential kernel, which is also shown to have desirable computational properties. For fractional Brownian motion type volatilities, numerical results to compute the optimal kernel are devised and, for the deterministic volatility case, explicit optimal kernel functions of different orders are derived. Finally, simulation studies further confirm the good performance of the proposed methods.}
\newline


\section{Introduction}

{It is no mystery that mathematical finance is} greatly influenced by the Geometric Brownian Motion in {the} Black-Scholes-Merton's option pricing model, which assumes a constant volatility parameter. {The latter assumption has greatly been refuted by many empirical studies performed on both asset and option price data. Two nonexclusive general approaches, {namely local and stochastic volatility,} have been proposed in the literature to incorporate the stylized features of market volatilities.} As a more general setting, the log price of an asset is usually assumed to follow {the dynamics} $dX_t = \mu_t dt + \sigma_t dW_t$, where {the volatility $\sigma_t$ may vary through time, in either deterministic or stochastic ways.} As a result, the estimation of the spot volatility has become an important and attractive problem, especially with the availability of high frequency data (HFD). {Accurate estimation of the spot volatility not only helps market participants to better assess and characterize the behavior of the volatility through time but is also crucial in many problems of finance such as option pricing, portfolio selection, and risk management.}

In this work, we revisit the problem of spot volatility estimation by kernel methods. Kernel estimation has a long history, starting with Rosenblatt et al. (1956) for density estimation. {Extensive treatments of the method can be found in many textbooks, such as, Tsybakov (2008).} The basic idea is to take {a} weighted average of the data where the weights are given by a kernel function that is appropriately scaled by a bandwidth parameter. {The selection of the bandwidth and the kernel function are of great importance for the performance of the kernel estimator in a finite sample setting.} {Bandwidth} selection methods have been thoroughly studied for density estimation and kernel regression. {Broadly there are two general approaches: plug-in type and cross validation methods. We refer the reader to Hall (1983), Park and Marron (1990), Park and Turlach (1992), Cao et al. (1994), and Jones et al. (1996), for a more in depth introduction and comparison of these methods. The problem of kernel selection has also been considered by, for instance,} Epanechnikov (1969) {and van Eeden (1985)}.

{In the context of spot volatility estimation,} Foster and Nelson (1996) studied a rolling window estimator, which {can be seeing as a kernel estimator with a compactly supported kernel function.} They established the point-wise asymptotic normality of the estimator, and drew some conclusions about the optimal window length (i.e., bandwidth) and the optimal weight functions (kernel functions). However, in spite of the non-parametric model setting, {the volatility is constrained to have a specific degree} of smoothness (see Assumption A (vii) and (viii) therein). Also, the selection of bandwidth and kernel function is not studied systematically, {since it presumed a strict relationship between the window's length and the sample size (see Assumption D therein). Under such a relationship, they obtained the optimal kernel weights and separately established the optimal constant appearing in the formula for the window length, though only for the flat-weights case (i.e., a uniform kernel function)}. Fan and Wang (2008) also {shows} a point-wise asymptotic normality for a general kernel estimator under some conditions on the order of smoothness of the volatility processes (in the sense of convergence in probability) and a specific constraint on the rate of convergence of the bandwidth (Condition A4 therein), without any considerations on the optimal bandwidth selection problem. The latter assumption on the bandwidth allows them to neglect the error coming from approximating the spot volatility by a kernel weighted volatility (we refer the reader to Section \ref{Section:CLT} for details), but the achieved convergence rates of the kernel estimator are suboptimal. Also, the {optimal selection of the} kernel function was not considered\footnote{{Indeed, with a suboptimal convergence rate, the selection of kernel function is generally not well defined}}.  {A recent paper in the same vein is Mancini et al. (2015), where the asymptotic normality of a more general class of spot volatility estimators, which includes kernel estimators, is studied without considering the kernel and bandwidth selection problems. Besides asymptotic normality, Kristensen (2010) also studied optimal bandwidth selection method, but under a strong path-wise {smoothness condition}, which has several practical and theoretical drawbacks}. {Indeed, even for simple volatility processes, it is not possible to verify the H\"older continuity needed for a central limit theorem with optimal rate.} {Furthermore, even though an `optimal' bandwidth is deduced in closed form therein, this is not well-defined if we want to attain optimal convergence rates for the estimation error (see Remark \ref{CmpKristensen} below for further details)}.
Based on a heuristic argument, an alternative cross-validation method of bandwidth selection {was} {also proposed {by Kristensen (2010)}}, but this algorithm {has high} computational complexity and the asymptotic properties were not {studied}.

Having discussed some previous work on the kernel estimation of spot volatility, we {now} mention some {motivating factors and objectives} of the present research. To begin with, we wish to adopt easily verifiable and general enough conditions to cover a wide range of frameworks without imposing strong constraints on the degree of smoothness of the volatility process. {From a theoretical point of view}, {we also aim to provide a formal justification of the optimal convergence rate of the kernel estimator and to establish central limit theorems {(CLT)} and asymptotic estimates of the mean square errors with optimal rates}. From the practical {side}, the two factors that affects the performance of the estimator, bandwidth and kernel function, {ought to be optimized jointly}, not separately, and meanwhile, the proposed method should remain feasible {and sufficiently efficient to be implementable for HFD.}

In this present work, we introduce a natural and relatively mild assumption on the volatility processes, which allows us to {obtain} feasible solutions to the optimal bandwidth and kernel selection problems. The assumption imposes a {local} homogeneous or scaling property for the covariance structure of the volatility process. {This assumption covers a wide range of frameworks including deterministic differentiable volatility processes {and} volatilities driven by Brownian Motion, long-memory fractional Brownian Motion, and more generally, {functions} of suitable Gaussian processes.}
{Under the referred assumption,} we characterize the  leading order terms of the Mean Squared Error (MSE). {As a byproduct, we} are then able to derive an approximated optimal bandwidth in closed form, which is shown to be asymptotically equivalent to the true optimal bandwidth. From this, the theoretical optimal convergence rate for {the estimation error} is rigorously identified. We then proceed to show that our optimal bandwidth formulas are feasible by proposing an iterated plug-in type algorithm for their implementation. An important {intermediate} step is to find an estimate of the volatility of volatility (vol vol), for which we propose a new {estimator based} on the two-time scale realized variance of Zhang et al. (2005). Consistency {and convergence rate} of our vol vol estimator are also established.  

Equipped with an explicit formula for the asymptotically optimal MSE, we {proceed to} setup a well posed problem for {optimal} kernel selection. Concretely, for Brownian motion driven volatilities, we prove that the optimal kernel function is the exponential kernel: $K(x) = 2^{-1} \exp (-|x|)$. Such a result formalizes and extends a previous result of Foster and Nelson (1994), {where only kernels of bounded support were considered. We} also show that, due to the nature of {the} data we are analyzing (namely, HFD), exponential kernel function enjoys outstanding computational advantages, as it reduces the time complexity for estimating the whole path of the volatility {on all grid points} from $O(n^{2})$ to $O(n)$. We {also} consider the volatility processes driven by the long-memory fractional Brownian motion and, in such a case, we provide numerical schemes to compute the optimal kernel function. For sufficiently smooth volatilities, {we also consider  higher order} kernel functions and, by using calculus of variation with constraints, we obtain optimal kernel functions of different orders. The second order optimal kernel is exactly Epanechnikov (1969) kernel and, for higher order cases, we provide ways to calculate those optimal kernel functions.

{To complement our asymptotic results based on MSE, asymptotic} normality of the kernel estimators is also established for two broad types of volatility processes: {It\^{o}} processes and continuous function of {some} Gaussian processes. {In this way, our results cover volatility} processes with flexible degrees of smoothness. The results are consistent with the leading order approximation of the MSE, so {that} CLT's with the optimal convergence rate are obtained. By contrast, as mentioned above, the CLT's of Fan and Wang (2008) and Kristensen (2010) have suboptimal convergence rate, while the analogous result of Foster and Nelson (1994) is limited to a specific smoothness order and strong constraints on the kernel function and bandwidth.

In the big picture, our {results} can be connected to several related topics in Statistical estimation of stochastic processes. For example, {our approach can be combined} with the Threshold Realized Power Variation (TPV) (cf. Mancini (2001, 2004), Figueroa-L\'{o}pez and Nisen (2013)). Furthermore, market micro-structure noise can also be included and methods like two-time scale (cf. Zhang et al. (2005)), multi-time scale (cf. Zhang (2006)) and kernel methods (cf. Barndorff et al. (2004)) {could potentially be combined with our kernel-based spot volatility estimators, though this problem is out of the scope of the present work.}

The rest of the paper is organized as follows. In Section 2, we introduce the kernel estimator and our assumptions, and verify that common volatility processes satisfy our assumptions. In Section 3, we deduce the leading order approximation of the MSE of the kernel estimator and solve the optimal bandwidth selection problem. Then, in Section 4, we deal with the optimal kernel function selection problem for different types of volatility processes. A feasible implementation approach of the optimal bandwidth is discussed in Section 5, where we also introduce the TSRVV estimator of vol vol. Central Limit Theorems of the kernel estimator are discussed in section 6. Finally in Section 7, we perform Monte Carlo studies. Some technical proofs are deferred to appendices.

\section{{Kernel Estimators and Assumptions}}\label{Estimator_and_Framework_Section}

In this section, we first introduce {the classical kernel estimator for the spot volatility.} We then discuss some needed smoothness conditions on the volatility processes and verify that most common volatility processes used in {the literature} indeed satisfy our assumptions. {Finally,} we discuss some regularity conditions on the kernel function and state some needed technical lemmas.

\subsection{{ Framework and Estimators}}

Throughout the paper, we will consider the following dynamic for the log price of an asset:
\begin{equation}\label{Asset_Dynamic}
dX_t = \mu_t dt + \sigma_t dB_t ,
\end{equation}
where all stochastic processes ($\mu := \{\mu_t\}_{t\geq 0}, \sigma := \{\sigma_t\}_{t\geq 0}, X := \{X_t\}_{t\geq 0}$, etc.) are defined on a complete filtered probability space $(\Omega, \mathscr{F}, \mathbb{F}=\{\mathscr{F}_t\}_{t\geq 0},\mathbb{P}_\theta)$ {and} where $\{\mathbb{P}_\theta:\theta \in \Theta\}$ is a class of probability {measures}, defined on $(\Omega, \mathscr{F})$ {and} indexed by $\theta \in\Theta$. {We also assume that $\mu$ and  $\sigma$ are adapted to the filtration $\mathbb{F}$ and} $B := \{B_t\}_{t\geq 0}$ is a standard Brownian Motion (BM) adapted to $\mathbb{F}$. We suppose throughout the paper {that we observe the log price process $X$ at the times $t_{i}:=t_{i,n} := iT/n$, $0 \leq i \leq n$}. We will use $\Delta_i^n X := \Delta X_{t_{i-1}} := X_{t_i} - X_{t_{i-1}}$ to denote the increments of log prices and {$\Delta_n = T/n$} to denote the time increments.
From standard theory of stochastic analysis, the integrated volatility $IV = \int_0^T \sigma_t^2 dt$ has a {classical} estimator, the {Realized} Quadratic Variation or Variance, {which is} defined as:
\begin{equation}\label{Usual_RV_Estimator}
RV_n := \widehat{[X]}_{T,n} := \sum_{i=1}^n (\Delta_i^n X)^2 \xrightarrow{\mathbb{P}} [X]_T = \int_0^T \sigma_t^2 dt.
\end{equation}
{Above, $[X]_T = \int_0^T \sigma_t^2 dt$ is the quadratic variation or integrated variance process. {In some literature},  $\int_0^T \sigma_t^2 dt$ is also sometimes called the integrated volatility of the process.} 
A natural way of turning the integrated {variance estimator} into a {spot volatility estimator} is to take a weighted average of the squared increments. Throughout, we consider the estimation of $\sigma_{\tau}^2$, {for a fixed time $\tau\in(0,T)$,} and we use a kernel function as weights so that more weights are given to points closer to $\tau$. Concretely, the Kernel weighted {spot variance} (c.f. \cite{fan2008spot} and \cite{kristensen2010nonparametric}) is defined as
\begin{equation}\label{Kernel_Weighted_Volatility}
KV_h (\tau) = \int _0^T K_h(s-\tau) \sigma_s^2 ds = \int _0^T K_h(s-\tau) d[X]_s ,
\end{equation}
where $K$ is the kernel function such that $\int K(x) dx = 1$ and we denote $K_h(x) := K(x/h)/h$, where $h$ is the so-called bandwidth. Some basic analysis shows us that $KV_h(\tau) \rightarrow \sigma_{\tau}^2$ as $h \rightarrow 0$, {under some} mild regularity conditions such as continuity. {By replacing $[X]$ with $\widehat{[X]}$, we then elucidate the Kernel weighted realized volatility estimator:}
\begin{equation}\label{Kernel_Estimator}
\hat{\sigma}^2_{\tau,n,h} = \widehat{KV}_{n,h}(\tau) := \int _0^T K_h(s-\tau) d\widehat{[X]}_{s,n} = \sum _{i=1}^n K_h(t_{i-1} - \tau) (\Delta_i^n X)^2 .
\end{equation}
At the first glance, we may expect that, {similarly to (\ref{Kernel_Weighted_Volatility}), $\widehat{KV}_{n,h}(\tau) \rightarrow \sigma_\tau^2$,} as $h \rightarrow 0$. However, since we are facing a finite sample of log prices, if we simply set $h \rightarrow 0$, the behaviour of $\widehat{KV}_{n,h}(\tau)$ is irregular. Therefore, in order to construct a consistent kernel estimator of the spot volatility, the bandwidth $h$ has to be selected carefully. 

{{As discussed in the introduction, the} literature on bandwidth and kernel selection methods for the spot volatility estimator (\ref{Kernel_Estimator}) is rather scarce}. \cite{fan2008spot} does not shed any light on {this problem, while} the conditions proposed by \cite{kristensen2010nonparametric} to address this problem are hard to be verified {and do not covered most of the models proposed in the literature} {(see Remark \ref{CmpKristensen} below for further details)}. In this work, we go further with better crafted conditions that are satisfied by most common volatility processes while enabling us to obtain explicit expressions for the optimal bandwidth and the optimal kernel function.

{Let us close by introducing} some notations that will be used throughout this paper. We will mainly consider limits when $n\rightarrow \infty$ and $h\rightarrow 0$. Without ambiguity and for brevity, we will use the 	simplified notations: $t_i := t_{i,n}, \Delta := \Delta_n, \Delta_i X := \Delta_i^n X, \hat{\sigma}_\tau := \hat{\sigma}_{\tau,n,h}$, etc. However, when we encounter $K(\cdot)$, we will always use $K(\cdot)$ to denote the kernel function itself and never drop the subscript of $K_h(x) = K(x/h)/h$.

\subsection{{Assumptions on the Volatility Process}}
In this section, we give the required assumptions on the volatility process that allow us to examine the rate of convergence of the kernel estimator defined in (\ref{Kernel_Estimator}). Our first assumption is a non-leverage assumption. This simplifying assumption {will make the problem more tractable and {is} widely used in the literature (see, e.g., \cite{kristensen2010nonparametric})}.

\begin{assumption}\label{IndependentCondition}
$(\mu,\sigma)$ is independent of $B$.
\end{assumption}

Another assumption that we need later is the boundedness of {some moments of} $\mu$ and $\sigma$.

\begin{assumption}\label{Boundedness_Condition}
There exists $M_T > 1$ such that $\mathbb{E}[\mu^4_t + \sigma^4_t] < M_T$, for all $0 \leq t \leq T$.
\end{assumption}

\begin{remark}
Note that this assumption implies $\mathbb{E}[|\mu_t|] < M_T$, $\mathbb{E}[\mu_t^2] < M_T${,} and $\mathbb{E}[\sigma^2_t] < M_T$, for all $t \in [0,T]$. We will use the notation $M_T$ later.
\end{remark}

Since we aim to study the problem of minimizing the Mean Squared Error of the estimator, we should correspondingly assume some smoothness of the expectation of the squared increments. The following assumption is {of this type, and,} as it turns out, is satisfied by most {volatility processes driven by BM (see {Proposition} \ref{BMProcesses} below for details)}.

\begin{assumption}\label{VolatilityCondition}
Suppose that for {a locally bounded function $A: \mathbb{R}_+ \rightarrow \mathbb{R}_+$, a function} $L: \mathbb{R}_+ \rightarrow \mathbb{R}_+$, and real numbers $\alpha, \beta \geq 0$ such that $\alpha + \beta > 1$, the variance process $V := \{V_t = \sigma_t^2 : t\geq 0\}$ satisfies
\begin{equation}\label{SVDefinition}
\begin{split}
\mathbb{E} [ (V_{t+h} - V_t)^2 ] 
= &
L (t) h +  O(h^{\alpha + \beta}), \quad t>0, h\rightarrow 0, \\
|\mathbb{E} [ (V_{t+h} - V_{t})(V_{t} - V_{t-s}) ]| 
\leq &
A(t) h^{\alpha}s^{\beta}, \quad {h>0,\;t>s>0}.
\end{split}
\end{equation}
\end{assumption}

Although the assumption above is enough for BM type volatility processes, we are also interested in other types of volatilities, {such as those driven by a fractional Brownian motion, that do not satisfy this condition}. To this end, we also consider the following more general assumption.

\begin{assumption}\label{VolatilityConditionGeneral}
Suppose that for $\gamma > 0$ and certain functions $L: \mathbb{R}_+ \rightarrow \mathbb{R}_+$, $C_\gamma: \mathbb{R} \times \mathbb{R} \rightarrow \mathbb{R}$, such that $C_\gamma$ is not identically zero and
\begin{equation}\label{HomoFunction}
\begin{split}
C_\gamma(hr,hs) & = h^{\gamma}C_\gamma(r,s), \quad
\mbox{ for } r, s \in \mathbb{R}, h \in \mathbb{R}_+,
\end{split}
\end{equation}
the variance process $V := \{V_t = \sigma_t^2 : t\geq 0\}$ satisfies
\begin{equation}\label{SCDefinitionGeneral}
\mathbb{E} [ (V_{t+r} - V_{t})(V_{t+s} - V_{t}) ] 
= L(t)C_\gamma(r,s) + o ((r^2 + s^2)^{\gamma/2}) , \quad r,s \rightarrow 0.
\end{equation}
Hereafter, we will also denote $C(r,s;t) = L(t) C_\gamma (r,s)$.
\end{assumption}

{{As shown in the next section, Assumption \ref{VolatilityConditionGeneral} is satisfied by} most common volatility models.}
\begin{remark}\label{CmpKristensen}
We now draw some connections {with the assumptions and work in \cite{kristensen2010nonparametric}. Therein,} the variance process {$\{V_{t}\}_{t\geq{}0}$} is assumed to satisfy  the following {pathwise condition}
\begin{equation}\label{Cndkristen}
|{V_{t + \delta} - V_{t}}| \leq \tilde{L}(t, |\delta|) |\delta|^\gamma + o(|\delta|^\gamma), \quad \delta \to 0,
\end{equation}
where $\tilde{L}(t, \cdot)$ is a slowly varying random function. {To gain some intuition about the plausibility of this assumption, let us suppose that $\{V_t\}$} is a Brownian motion. {In that case, the above holds for all $\gamma < 1/2$, but such choices of $\gamma$ can only produce} suboptimal convergence rate of the kernel estimator. {On the other, in light of L\'evy's modulus of continuity, 
$$
\sup_{t \in (0, T)} \limsup_{\delta \to 0} \frac{|B(t + \delta) - B(t)|}{\sqrt{2\delta\log (1 / \delta)}} = 1, \quad a.s.
$$
the condition {(\ref{Cndkristen})} holds for $\gamma = 1/2$, but only} if $\tilde{L}(t, \delta) \to \infty$, as $\delta \to 0$. {But, in that case, the optimal bandwidth selection formulas obtained in \cite{kristensen2010nonparametric} are not well defined as they presume that $\lim_{\delta\to{}0}\tilde{L}(t, \delta)=:\tilde{L}(t,0)$ is finite.} 
\end{remark}

{A function $C_{\gamma}$ satisfying the condition (\ref{HomoFunction}) is said to be homogeneous of order $\gamma$}.
There are several {preliminary properties} we need to establish regarding the previous assumption and the function $C_{\gamma}$ therein. Firstly, it is easy to see that Assumption \ref{VolatilityCondition} is a special case of Assumption \ref{VolatilityConditionGeneral} with $\gamma = 1$ and $C_\gamma(r,s) = \min\{|r|, |s|\} 1_{\{rs \geq 0\}}$. Therefore, throughout the paper we will refer to Assumption \ref{VolatilityConditionGeneral} rather than Assumption \ref{VolatilityCondition}.

{The next result shows the non-negative definiteness of {the function} $C_{\gamma}$}.
\begin{proposition}\label{Proposition_Nonnegative_Definiteness_of_C_Gamma}
Under Assumption \ref{VolatilityConditionGeneral}, {both $C_{\gamma}(\cdot, \cdot ; t)$ and $C_{\gamma}(\cdot, \cdot)$} are integrally non-negative definite. {That is, 
\begin{equation}\label{Eq:NPDDfn}
	{\iint K(x)K(y) C(x, y) dxdy \geq 0},
\end{equation}
for all functions $K: \mathbb{R} \rightarrow \mathbb{R}$ for which the integral therein is well-defined.}
\end{proposition}

\begin{proof}
To prove the result, we write (\ref{SCDefinitionGeneral}) as
$
\mathbb{E} [ (V_{t+r} - V_{t})(V_{t+s} - V_{t}) ] 
= C_{\gamma}(r,s;t) + D(r,s;t),
$
where $D(r,s;t) = o ((r^2 + s^2)^{\gamma/2})$, as $r,s \rightarrow 0$. We first show that $C_\gamma$ is non-negative definite. Indeed, for $n\in \mathbb{N}$, $(x_1,...,x_n) \in \mathbb{R}^n$,  $(c_1,...,c_n) \in \mathbb{R}^n - \{0\}$ and $h \in \mathbb{R}_+$, we have
\begin{equation*}
\begin{split}
&\quad \sum_{i = 1}^n \sum_{j = 1}^n c_i c_j C_{\gamma} (x_i,x_j;t)
=
h^{-\gamma} \sum_{i = 1}^n \sum_{j = 1}^n c_i c_j C_{\gamma}(hx_i,hx_j;t) \\
& =
h^{-\gamma} \sum_{i = 1}^n \sum_{j = 1}^n c_i c_j \mathbb{E} [ (V_{t+hx_i} - V_{t})(V_{t+hx_j} - V_{t}) ] 
- h^{-\gamma} \sum_{i = 1}^n \sum_{j = 1}^n c_i c_j D(hx_i,hx_j;t) \\
& = h^{-\gamma} \mathbb{E} \left[ \left( \sum_{i = 1}^n c_i (V_{t+hx_i} - V_{t}) \right)^2 \right]
- h^{-\gamma} \sum_{i = 1}^n \sum_{j = 1}^n c_i c_j D(hx_i,hx_j;t).
\end{split}
\end{equation*}
{On the right-hand side of the previous equation,} we let $h \rightarrow 0_+$ and we have that the first term is always non-negative, {while} the second term converges to zero. This shows the non-negative definiteness of $C_{\gamma}$. The integral non-negative definiteness follows then, since the Riemann integration is defined to be the limit of finite sum, which is always non-negative.
\end{proof}

The next result establishes the uniqueness of $\gamma$ and $C_{\gamma}$ in (\ref{SCDefinitionGeneral}).
\begin{proposition}\label{Proposition_Uniqueness_of_C_Gamma}
Under Assumption \ref{VolatilityConditionGeneral}, the $\gamma$ and $C_{\gamma}(r,s;t)$ defined in (\ref{SCDefinitionGeneral}) are unique. This means that $C_{\gamma}(r,s)$ is unique up to a multiple of a positive constant for a given $t$.
\end{proposition}

\begin{proof}
First we prove the uniqueness of $\gamma$. Suppose there are $\gamma, \gamma'$ such that $\gamma' > \gamma > 0$, and correspondingly, $C_{\gamma}$ and $C'_{\gamma'}$, that satisfies (\ref{SCDefinitionGeneral}). Since $C_{\gamma}$ is non-zero, there exists $r,s \in \mathbb{R}$, such that $C_{\gamma} (r,s;t) \neq 0$. Then,
\begin{equation*}
\begin{split}
\mathbb{E} [ (V_{t+rh} - V_{t})(V_{t+sh} - V_{t}) ] 
& = 
h^\gamma C_{\gamma}(r,s;t) + o(h^{\gamma} (r^2 + s^2)^{\gamma / 2}) \\
& =
h^{\gamma'} C'_{\gamma'}(r,s;t) + o(h^{\gamma'} (r^2 + s^2)^{\gamma / 2}), \quad h\rightarrow 0 .
\end{split}
\end{equation*}
Note now that in the right two parts, all the terms are $o(h^{\gamma})$ except $h^{\gamma} C_{\gamma}(r,s;t)$. Since we have assumed that $C_{\gamma}(r,s;t) \neq 0$, this is impossible. Therefore, $\gamma = \gamma'$ and, thus, $\gamma$ must be unique. Now with the same $\gamma$, suppose at some $r,s$, we have $C_\gamma(r,s) \neq C_\gamma'(r,s)$. Then, a similar argument shows that this leads to a contradiction. This proves the uniqueness of $\gamma$ and $C_\gamma$.
\end{proof}

{It is worth} mentioning that we are assuming a fixed $C_{\gamma}(r,s)$, for any $t \in (0,T)$. Intuitively, this means that the covariance structure does not change over time. For example, we are not considering the case {where} the volatility is BM type in $[0,T_0]$ and is deterministic and smooth in $[T_0,T]$.
{We shall see in the next section that most} volatility processes that are studied in {the} Mathematical Finance literature satisfy Assumption \ref{VolatilityConditionGeneral} with {a function $C_\gamma$ of the form:}
\begin{equation}\label{fBM_Vol_Cov_Structure}
C_\gamma(r,s)
=
\frac{1}{2} (|r|^{\gamma} + |s|^{\gamma} - |r-s|^{\gamma}) ,
\end{equation}
{for some $\gamma \in [1,2]$. The case of $\gamma = 1$ covers volatility} processes driven by BM, {while $\gamma \in (1,2)$ corresponds to volatility processes driven by fractional Brownian Motions (fBM) with Hurst parameter $H>1/2$. Deterministic and sufficiently smooth volatility processes can also be incorporated by taking $\gamma = 2$}. 

{Let us note} that although {most of the volatility models considered in the literature are covered by} (\ref{fBM_Vol_Cov_Structure}), mathematically one can consider more general models as long as Assumption \ref{VolatilityConditionGeneral} is satisfied.  {For instance,} we will see in the next section that for a suitable Gaussian processes {$\{Z_t\}_{t\geq{}0}$} and a smooth function $f : \mathbb{R} \to \mathbb{R}$, $V_{t}:=f(Z_t)$ satisfies Assumption \ref{VolatilityConditionGeneral}. {Furthermore,} for any valid non-negative definite symmetric function $C_\gamma$ that is homogeneous to order $\gamma$, one can define a zero-mean continuous Gaussian process {$\{Z_t\}_{t\geq{}0}$} such that $\mathbb{E}[{Z_tZ_s}] = C_\gamma(t,s)$. In such a case, if we define $V_t = \sigma_t^2$ as a stochastic integral {with respect to} $\{{Z_t}\}_{t \geq 0}$, then generally $\{V_t\}_{t\geq{}0}$ {would satisfy} Assumption \ref{VolatilityConditionGeneral}.

To close this part, we briefly summarize the advantages of our key Assumption \ref{VolatilityConditionGeneral}:
\begin{itemize}
\item
The assumption is natural since the spirit of kernel estimator is to {focus} on data points closed to the estimated point. Therefore, the convergence of such an estimator should be determined by {a local} property of the volatility process near the estimated point.
\item
The assumption enables us to consider the randomness of log price process and volatility process simultaneously. Although we assume independence of $(\mu,\sigma)$ and $B$, the randomness of $\sigma$ does create some subtleties. Also this makes it possible for future work to incorporate leverage effect of the volatility process.
\item
The assumption provides us the possibility of obtaining an explicit asymptotic characterization of the Mean Squared Error (approximated to the first order) of the kernel estimator, so that {we will then be able to setup a well-posed optimal selection problem for the} bandwidth and kernel function.
\end{itemize}

\subsection{Common Volatility Processes}\label{Common_Volatility_Processes_Section}

In this section, we demonstrate that common volatility processes satisfy the Assumption \ref{VolatilityConditionGeneral}. There are {four} fundamental cases that we would like to investigate. The simplest case is when the volatility process is deterministic and is differentiable. The second case is the solutions of {a classical} {Stochastic Differential Equation (SDE)} driven by BM. {A prototypical example of this case is the Heston model \cite{Heston_Model}}. The third case is the solution of {a} SDE driven by fBM. As a fundamental example of this case, we prove that a fractional Ornstein Uhlenbeck (fOU) process satisfies Assumption \ref{VolatilityConditionGeneral}. {Finally, we consider a volatility that is a smooth function of a Gaussian process satisfying the Assumption \ref{VolatilityConditionGeneral}. This covers a wide range of different processes of fractional order and with different {distribution laws}.}

\subsubsection{Deterministic Volatility Process}\label{SubSect:Deterministic}
{This is the simplest type of volatility process}, but still worth mentioning since it demonstrates the generality of Assumption \ref{VolatilityConditionGeneral}. {The proof {of the following result} is standard and is omitted for the sake of brevity.}

\begin{proposition}\label{prop:DtmVol}
Suppose the squared volatility process is given by a deterministic function {$f(t) = \sigma_t^2$, $0\leq t \leq T$,} such that, {for some $m\geq{}1$, $f$ is $m^{th}$-times differentiable at $\tau\in(0,T)$}, $f^{(i)}(\tau) = 0$, for $1 \leq i < m$, and $f^{(m)}(\tau) \neq 0$. Then, {$f$ satisfies} Assumption \ref{VolatilityConditionGeneral} {with} $\gamma = 2m$ and {$C_{2m}(r,s) := r^ms^m$}.
\end{proposition}

\subsubsection{Brownian Motion Case}
We {next} consider the solutions {of a standard} SDE driven by BM. This is one of {the most popular approaches} to generalize the Black-Scholes-Merton model to non-constant volatility {and is widely used} in practice.
{The following is our main result, whose proof is deferred to the Appendix \ref{Sect:TecProof}.}

\begin{proposition}\label{BMProcesses}
Consider {a} complete filtered probability space $(\Omega, \mathscr{F}, \mathbb{F}=\{\mathscr{F}_t\}_{t\geq 0} , \mathbb{P})$ and an It\^{o} process $V_t = \sigma^2(t,\omega)$ that satisfies {the SDE}
\begin{equation}\label{eq:Ito_process}
dV_t = f(t,\omega) dt + g(t,\omega)dW_t, \quad t\in [0,T],
\end{equation}
where {$\{W_t\}_{t\geq{}0}$} is a standard Wiener process {adapted to} $\mathbb{F}$. Assume that {$f(t,\omega)$ and $g(t,\omega)$} are adapted and progressively measurable {with respect to} $\mathbb{F}$, 
$\mathbb{E} \left[ f^2(t,\omega) \right] < M$, for $t \in [0,T]${, and} $\mathbb{E} \left[ g^2(t,\omega) \right]$ is continuous for $t \in [0,T]$. Then, Assumption \ref{VolatilityConditionGeneral} is satisfied with $\gamma = 1$, {$C_1(r,s) = \min\{|r|, |s|\}1_{\{rs\geq 0\}}$, and $L(t) = \mathbb{E}[g^2(t,\omega)]$. Furthermore, {$C_{1}(r,s)$} is an integrable positive definite function; i.e., we have strict inequality in (\ref{Eq:NPDDfn}) for all $K: \mathbb{R} \rightarrow \mathbb{R}$ such that $\int |K(x)| dx > 0$.}
\end{proposition}

\begin{example}[\textbf{Heston Model}]\label{HestonProcess}
Consider the following Heston model, which was studied in \cite{Heston_Model}:
\begin{equation}\label{HestonTransformed}
\begin{split}
dX_t = & \mu_t dt + \sqrt{V_t} dB_t , \\
dV_t = & \kappa(\theta-V_t)dt + \xi \sqrt{V_t} dW_t ,
\end{split}
\end{equation}
where parameters are restricted to $2\kappa \theta > \xi^2$, {so that} $V_t$ is always positive. {This is one of the most widely used stochastic volatility models in Finance. The volatility process appearing above follows the so-called CIR model, which was introduced in \cite{CIR_Model}. As {an immediate consequence} of {Proposition} \ref{BMProcesses}, we deduce that the Heston model satisfies Assumption \ref{VolatilityConditionGeneral}} with $\gamma = 1$, a positive definite $C_1(r,s) = \min\{|r|, |s|\}1_{\{rs\geq 0\}}${,} and $L(t) = \mathbb{E}[g^2(t,\omega)]$.
\end{example}

\subsubsection{Fractional Brownian Motion Case}
We {now proceed to study a volatility {process} driven by a fBM with Hurst parameter $H>1/2$}. 
{Recall that a} stochastic process ${B^{(H)}}=\{B_t^{{(H)}}: t\in {\mathbb{R}}\}$ is called a {(two-sided)} fractional Brownian Motion with Hurst parameter $H \in (0,1)$ if {this is a zero-mean Gaussian process with covariance function}
$$
{\mathbb{E}\left[B_t^{{(H)}} B_s^{{(H)}}\right] = \frac{1}{2}\left({|s|^{2H} + |t|^{2H} - |s-t|^{2H}}\right), \quad {s,t\in\mathbb{R}}}.
$$
{In particular, when $H=\frac{1}{2}$,} we have $\mathbb{E}\left[B_t^{{(H)}} B_s^{{(H)}}\right] = \min \{ s,t \}$ for $s,t > 0$, {and, thus, $\{B_t^{(1/2)}\}_{t\geq{}0}$} is the standard BM. {We refer the reader to \cite{SamoTaqqu} for a {detailed} survey of fBM.}

{An important property of fBM, that is relevant to our problem, is self similarity. Concretely, $B^{(H)}$ is such that, {for any $r>0$,} the process $\{r^{H}B_t^{{(H)}}\}_{t\in\mathbb{R}}$ has the same finite-dimensional distributions as $\{B_{rt}^{{(H)}}\}_{t\in\mathbb{R}}$}, because the covariance function is homogeneous of order $2H$. This gives us some intuition {as to} why Assumption \ref{VolatilityConditionGeneral} holds. {The hurst parameter $H$} characterizes several important properties of fBM. For example, {for $H \in (\frac{1}{2},1)$, the process {exhibits} the so-called long-range dependence property, which broadly states that the {autocorrelation} of the increments of the process, $\{B_{k}^{(H)}-B_{k-1}^{(H)}\}_{k\geq{}1}$, {vanishes rather slowly so that the following holds:}
$$\sum_{n=1}^{\infty} |\mathbb{E}[(B_{k}^{{(H)}} - B_{k-1}^{{(H)}})(B_{k+n}^{{(H)}} - B_{k+n-1}^{{(H)}})]|= \infty .$$
Some empirical studies (see, e.g., \cite{Volatility_has_Long_Memory}) have suggested that {the volatility in real markets exhibits some type of long-memory and, due to this, we focus on the case $H \in (\frac{1}{2},1)$.}

{In what follows, we show that some processes defined as integrals with respect to fBM satisfy Assumption \ref{VolatilityConditionGeneral}. It is worth mentioning that, when} $H\neq1/2$, the fBM} is not a semimartingale {and the problem of defining  the stochastic integral with respect to fBM is more subtle. There are several approaches to this problem}. In our paper, we only focus on integrals of deterministic functions $f$ for which the integral can be defined on a path-wise sense under the following condition (cf. \cite{SamoTaqqu}):
\begin{align}\label{NCFEI}
 	\int_{-\infty}^{\infty} \int_{-\infty}^{\infty} {|f(u)f(v)|}|u-v|^{2H-2} dudv < \infty.
\end{align}
{Since there is no guarantees that the stochastic integral of $f$ with respect to fBM is nonnegative, which is a requirement of  a volatility process, we also consider the exponential of such a process. 
This is our result, whose proof is deferred to the Appendix \ref{Sect:TecProof}.} 
\begin{proposition}\label{stationaryfBM}
Consider a filtered probability space $(\Omega, \mathscr{F}, \mathbb{F}=\{\mathscr{F}_t\}_{t\geq 0} , P)$ and a process {$Y^{(H)}=\{Y_t^{{(H)}}\}_{t\geq{}0}$} that satisfies
$$Y_t^{{(H)}} = \int_{-\infty}^t f(u) dB_u^{{(H)}},\quad {t\geq{}0},$$
where $\{B_t^{{(H)}}\}_{{t\in\mathbb{R}}}$ is a {(two-sided)} fBM with Hurst parameter $H \in (\frac{1}{2},1)$ and $f(\cdot)$ is a continuous function that {satisfies (\ref{NCFEI}).
Then, the processes} $Y^{(H)}$ and $\{\exp(Y_t^{{(H)}})\}_{t\geq{}0}$ satisfy Assumption \ref{VolatilityConditionGeneral} with $\gamma = 2H \in (1,2)$ and $C_\gamma$ given by (\ref{fBM_Vol_Cov_Structure}).
\end{proposition}

As a {prototypical example, the fractional Ornstein-Uhlenbeck process (fOU)} (cf. \cite{fOU_Cheridito}), which is frequently used to model volatility processes, satisfies Assumption \ref{VolatilityConditionGeneral}. The fOU process, with Hurst parameter $H \in (\frac{1}{2},1)$, is defined {as the solution of the following SDE}, 
\begin{equation}\label{eq:fOU_SDE}
\begin{split}
dY_t^{{(H)}} = - \lambda Y_t dt + \sigma d B_t^{{(H)}}.
\end{split}
\end{equation}
{It is known that the previous SDE admits} the stationary solution:
\begin{equation}\label{eq:fOU_solution}
\begin{split}
Y_t^{{(H)}} = \sigma \int_{-\infty}^t e^{-\lambda(t-u)}dB_u^{{(H)}},\quad t\geq 0.
\end{split}
\end{equation}
{We have the following result (see Appendix \ref{Sect:TecProof} for a proof)}.
\begin{lemma}\label{stationaryfOU}
Let $\{Y_t^{{(H)}}\}_{t\geq{}0}$ be the fractional Ornstein-Uhlenbeck process defined by \eqref{eq:fOU_solution}, with Hurst parameter $H \in (1/2,1)$. Then, {the processes $\{Y_t^{{(H)}}\}_{t\geq{}0}$ and $\{\exp(Y_t^{{(H)}})\}_{t\geq{}0}$} satisfy Assumption \ref{VolatilityConditionGeneral} with $\gamma = 2H \in (1,2)$ and $C_\gamma$ given by (\ref{fBM_Vol_Cov_Structure}).
\end{lemma}

\subsubsection{Functions of Gaussian Processes}

We now proceed to define another class of processes satisfying Assumption \ref{VolatilityConditionGeneral}. The {following} proposition guarantees that if a Gaussian process satisfies Assumption \ref{VolatilityConditionGeneral}, so does a {suitable smooth} function of the process. {See Appendix \ref{Sect:TecProof} for a proof.}
\begin{proposition}\label{proposition:C_gamma_of_function_of_gaussian_process}
Assume that $(Z_t)_{t \geq 0}$ is a Gaussian process that satisfies Assumption \ref{VolatilityConditionGeneral} uniformly over $(0, T)$,\footnote{{The Assumption \ref{VolatilityConditionGeneral} is satisfied uniformly over $(0, T)$ if
$\sup_{\tau\in (0,T)}
(r^2 + s^2)^{ - \gamma/2} \left| 
\mathbb{E} [ (V_{\tau+r} - V_{\tau})(V_{\tau+s} - V_{\tau}) ]  - L(\tau)C_\gamma(r,s) 
\right| \rightarrow 0$, 
as $r, s \rightarrow 0$, and, also, $\sup_{\tau\in (0, T)} |L(\tau)| < \infty$. This implies the existence of a positive constant $C$ such that $\mathbb{E} [(Z_{t} - Z_{s})^2] \leq C |t - s|^\gamma $, for all $t,s\in(0,T)$.}}
with $\gamma^{(Z)} \in [1, 2)$, $L(\cdot)$, and $C^{(Z)}_\gamma(\cdot, \cdot)$ defined as in \eqref{SCDefinitionGeneral}. For each fixed $\tau \in (0, T)$ and a function $f \in C^2(\mathbb{R})$, further assume the following:
\begin{enumerate}[(a)]
\item
$\mathbb{E}[(Z_{\tau + r} - Z_{\tau})Z_\tau] = O(|r|)$, $\mathbb{E}[Z_{\tau + r}] - \mathbb{E}[Z_{\tau}] = O(|r|)$, {as $r\to{}0$.}
\item
$\mathbb{E}[ (f^\prime(Z_\tau))^4] < \infty$, $\mathbb{E}[ \sup_{t\in (\tau - \epsilon, \tau + \epsilon)}(f^{\prime\prime}(Z_t))^4] < \infty$ for some $\epsilon > 0$.
\end{enumerate}
Then, the process $V_{t}:=f(Z_t)$, $t\geq{}0$, satisfies Assumption \ref{VolatilityConditionGeneral} with $\gamma^{(V)} = \gamma$ and $C^{(V)}_\gamma = \mathbb{E}[(f^\prime(Z_t))^2] C_\gamma^{(Z)}$.
\end{proposition}

\begin{remark}
Note that the condition (a) in Proposition \ref{proposition:C_gamma_of_function_of_gaussian_process} is not a consequence of Assumption \ref{VolatilityConditionGeneral}.  {This is satisfied by a large class of Gaussian processes,} such as a fBM with zero mean and covariance structure given by \eqref{fBM_Vol_Cov_Structure}. Intuitively, this condition states that, although $Z_\tau$ and $Z_{\tau + r} - Z_{\tau}$ may not be independent, its correlation coefficient vanishes, as $r \rightarrow 0$, fast enough as compared with standard deviation of $Z_{\tau + r} - Z_{\tau}$.
\end{remark}

\subsection{Conditions on the Kernel {and Preliminary Results}}\label{Kernel_Condition_Section}
In this part, we introduce the {assumptions} needed on the kernel function, together with some required lemmas.

\begin{assumption}\label{AdmissibleKernel}
Given $\gamma > 0$ and $C_{\gamma}$ as defined in Assumption \ref{VolatilityConditionGeneral}, {we assume that the kernel function $K:\mathbb{R} \rightarrow \mathbb{R}$ satisfies the following conditions:}
\begin{enumerate}[(1)]
\item
$\int K(x) dx = 1${;}
\item
$K$ is Lipschitz and piecewise $C^1$ on its support $(A,B)$, where $-\infty \leq A < 0 < B \leq \infty${;}
\item
{(i)} $\int |K(x)||x|^{\gamma} dx < \infty$; {(ii) $K(x)x^{\gamma + 1} \rightarrow 0$, as} $|x| \rightarrow \infty$; {(iii)} $\int |K^{\prime}(x)| dx < \infty$, {(iv)} $V_{-\infty}^{\infty} (|K^{\prime}|) < \infty$, where $V_{-\infty}^{\infty}(\cdot)$ is the total variation{;}
\item
$\iint K(x)K(y)C_{\gamma}(x,y)dxdy > 0$.
\end{enumerate}
\end{assumption}

\begin{remark}
{Note} that (4) above does not put substantial restriction on $K$ since, {in any case,} $C_\gamma$ is non-negative {definite (see} Proposition \ref{Proposition_Nonnegative_Definiteness_of_C_Gamma}) {and, furthermore, $C_\gamma$ is strictly positive definite in some important cases such as BM type volatilities}. In the case of deterministic volatility, it is possible to find $K$ such that $\iint K(x)K(y)C_{\gamma}(x,y)dxdy = 0$, {which actually will lead to even a faster rate of convergence of the estimation mean-squared error.} This will be discussed further in Section \ref{section_optimal_kernel_deterministic}.
\end{remark}

The following two technical lemmas will be used throughout the paper, and the proofs are deferred to the Appendices.

\begin{lemma}\label{D1Lemma}
For $\gamma > 0$, assume the following for a function $f:\mathbb{R}^m \rightarrow \mathbb{R}$ {and {functions} $K_i : \mathbb{R} \rightarrow \mathbb{R}$, $1\leq i \leq m$}:
\begin{enumerate}[(i)]
\item
$f(\tau+s_1,...,\tau+s_m) - f(\tau,...,\tau) = C_{\gamma} (s_1,...,s_m;\tau) + o ((s_1^2 + .. + s_m^2)^{\gamma / 2})$, as $(s_1,...,s_m)\rightarrow 0$ for any given $ \tau \in (0,T)$, where $C_\gamma :\mathbb{R}^m \times [0,T] \rightarrow \mathbb{R}$ is a function such that
$$C_{\gamma}(hs_1,...,hs_m;\tau) = h^{\gamma} C_{\gamma}(s_1,...,s_m;\tau),\quad s_1,...,s_m \in \mathbb{R}, h > 0, \tau \in (0, T).$$
\item
$f \in C([0,T]^m)$.
\item
For $1\leq i \leq m$, $K_i$ satisfies {Conditions} (2) and (3) of Assumption \ref{AdmissibleKernel} with a support $(A_i, B_i)$.
\end{enumerate}	
Let
\begin{equation*}
\begin{split}
D_1 (f)
:= & 
\sum _{i_1,...,i_m=1}^n K_{1h}(t_{i_1-1} - \tau)...K_{mh}(t_{i_m-1} - \tau) \int _{t_{i_1-1}}^{t_{i_1}} ... \int _{t_{i_m-1}}^{t_{i_m}} f(s_1,...,s_m) ds_1...ds_m \\
& \quad - 
\int _{[0,T]^m} K_{1h}(s_1 - \tau) ... K_{mh}(s_m - \tau) f(s_1,...,s_m) ds_1 ... ds_m ,
\end{split}
\end{equation*}
where {$K_{ih}(x) := K_i(x/h)/h$}. Then, for {each} $\tau \in (0,T)$, we have the following:
\begin{equation*}
\begin{split}
D_1 (f)
=
\frac{1}{2} f(\tau,...,\tau) \left[ \prod_{i=1}^m \int K_i(x)dx \right] \sum_{i=1}^m \frac{ K_i(A_{i}^+)- K_i(B_{i}^-)} {\int K_i(x)dx} \frac{\Delta}{h} + o \left( \frac{\Delta}{h} \right) .
\end{split}
\end{equation*}
as {$h\to{}0$ and $\Delta/h \rightarrow 0$}. {If, furthermore, the condition (i) above} is satisfied uniformly over $\tau \in (0, T)$, {then} the approximation above is also uniform over $\tau \in (0, T)$.
\end{lemma}

\begin{remark}
{It is worth mentioning that $C_\gamma$ here has similar meaning as the one appeared in Assumption \ref{VolatilityConditionGeneral}, so we use the same notation $C_\gamma$.}
It is also worth noticing that {if} $h \rightarrow h_0 > 0$ {but still $\Delta\to{}0$}, {then} we {again have} {$D_1(f) \sim \Delta/h$}, but the constant before {$\Delta/h$} depends on $f$ not only through $f(\tau, ..., \tau)$. 
\end{remark}

\begin{lemma}\label{D2Lemma}
For $\gamma > 0$, assume the following for a function $f:\mathbb{R}^m \rightarrow \mathbb{R}$ and a {function} $K : \mathbb{R} \rightarrow \mathbb{R}$:
\begin{enumerate}[(i)]
\item
$f$ satisfies {the conditions} (i) and (ii) of Lemma \ref{D1Lemma},
\item
$K$ satisfies {the conditions} (2) and (3) of Assumption \ref{AdmissibleKernel} with a support $(A, B)$.
\end{enumerate}
Let
$$
D_2 (f) := \int_{[0,T]^m} K_h(t_1 - \tau)... K_h(t_m - \tau) f(t_1,...,t_m)dt_1...dt_m - f(\tau,...,\tau) \left( \int K(x) dx \right)^m	.
$$
Then, for all $\tau \in (0,T)$, we have:
$$
D_2 (f)=
h^{\gamma} \int K(t_1)...K(t_m) C_\gamma (t_1,...,t_m;\tau) {dt_{1}\dots dt_{m}} + o(h^{\gamma}),\qquad {(h\to{}0).}
$$
{The} result remains the same if the integration domain of the first term in $D_2(f)$ is $\mathbb{R}^m$, instead of $[0,T]^m$. {Furthermore, if the condition} (i) of Lemma \ref{D1Lemma} is satisfied uniformly over $\tau \in (0, T)$, the approximation above {holds true uniformly} over $\tau \in (0, T)$.
\end{lemma}

\section{Approximation of MSE and {Optimal} Bandwidth Selection}\label{section_bandwidth_selection}
In this section, we assume that the processes $\mu$, $\sigma$, and, $B$ satisfy Assumptions \ref{IndependentCondition}, \ref{Boundedness_Condition}{,} and \ref{VolatilityConditionGeneral}{,} and we consider a kernel function $K$ that satisfies Assumption \ref{AdmissibleKernel}. 
{In what follows}, we first deduce an explicit leading order approximation (up to $O(\frac{\Delta}{h})$ and $O(h^{\gamma})$ {terms}) of the $\mbox{MSE} = \mbox{MSE}_{n,h} = \mathbb{E}[(\hat{\sigma}^2_{\tau,n,h} - \sigma^2_{\tau} )^2]$. After this, we proceed to study the approximated optimal bandwidth $h${, which is defined as the bandwidth that} minimizes {the leading order approximation of the} MSE. {Finally, we} prove that our approximated optimal bandwidth is {asymptotically equivalent} to the true optimal bandwidth that minimizes the true MSE.

\subsection{Approximation of the Mean Squared Error}
Let us start by writing the MSE as
\begin{align*}
\mbox{MSE}
& =
\mathbb{E}\left[\left(\sum_{i=1}^n K_h(t_{i-1} - \tau) (\Delta_i X)^2 - \sigma^2_{\tau} \right)^2\right] \\
& = 
\mathbb{E}\left[\left(\sum_{i=1}^n K_h(t_{i-1} - \tau) ((\Delta_i X)^2 - \Delta \sigma^2_{\tau} ) + \left( \sum_{i=1}^n K_h(t_{i-1} - \tau) \Delta - 1 \right)  \sigma^2_{\tau} \right)^2 \right] .
\end{align*}
By {Lemmas} \ref{D1Lemma} and \ref{D2Lemma} with $f(t) \equiv 1$, we have $\sum_{i=1}^n K_h(t_{i-1} - \tau) \Delta - 1 = O \left( \frac{\Delta}{h} \right) + o(h^\gamma)$ and, {thus,
\begin{equation}\label{MSE_Deduction0}
\begin{split}
\mbox{MSE}
& =
\mathbb{E}\left[\left(\sum_{i=1}^n K_h(t_{i-1} - \tau) ((\Delta_i X)^2 - \Delta \sigma^2_{\tau} ) \right)^2\right] + o \left( \frac{\Delta}{h} \right) + o(h^\gamma)\\
& =
\sum_{i=1}^n \sum_{j=1}^n K_h(t_{i-1} - \tau) K_h(t_{j-1} - \tau) \mathbb{E}[ ((\Delta_i X)^2 - \Delta \sigma^2_{\tau} ) ((\Delta_j X)^2 - \Delta \sigma^2_{\tau} ) ] + o \left( \frac{\Delta}{h} \right) + o(h^\gamma),
\end{split}
\end{equation}
which, applying {Lemmas} \ref{D1Lemma} and \ref{D2Lemma} together with Assumption \ref{IndependentCondition} and \ref{Boundedness_Condition} {(we} refer to the Appendix \ref{More_Technical_Details} for more details), can further be written as 
\begin{equation}\label{MSE_Deduction1}
\begin{split}
MSE
& =
\sum_{i=1}^n \sum_{j=1}^n K_h(t_{i-1} - \tau) K_h(t_{j-1} - \tau) \\& \quad
\quad\quad\quad {\times}\,\mathbb{E} \left[ \left( \left( \int_{t_{i-1}}^{t_i} \sigma_t dB_t \right)^2 - \Delta \sigma^2_{\tau} \right) \left( \left( \int_{t_{j-1}}^{t_j} \sigma_t dB_t \right)^2 - \Delta \sigma^2_{\tau} \right) \right] + o \left( \frac{\Delta}{h} \right) + o(h^\gamma).
\end{split}
\end{equation}
Next, by Assumption \ref{IndependentCondition}, it readily follows that
\begin{equation}\label{MSE_Deduction}
\begin{split}
MSE& =
2 \sum_{i=1}^n K_h^2(t_{i-1} - \tau) \mathbb{E} \left[ \left(\int_{t_{i-1}}^{t_i} \sigma_t^2  dt \right)^2 \right] \\
& \quad +
\sum_{i=1}^n \sum_{j=1}^n K_h(t_{i-1} - \tau) K_h(t_{j-1} - \tau)  \int_{t_{i-1}}^{t_i} \int_{t_{j-1}}^{t_j} \mathbb{E} [ (\sigma_t^2 - \sigma_{\tau}^2) (\sigma_s^2 - \sigma_{\tau}^2) ] dt ds
+ o \left( \frac{\Delta}{h} \right) + o(h^\gamma)\\
& =:
2 V_1 + V_2 + o \left( \frac{\Delta}{h} \right) + o(h^\gamma).
\end{split}
\end{equation}
We} now proceed to {analyze} $V_1$ and $V_2$. Firstly, for $V_1$, {note that}
\begin{align*}
\mathbb{E} \left(\int_{t_{i-1}}^{t_i} \sigma_t^2  dt \right)^2
& =
\Delta ^2 \mathbb{E}[\sigma_{\tau}^4] + 2 \Delta \int_{t_{i-1}}^{t_i}  \mathbb{E} [ (\sigma_t^2 - \sigma_{\tau}^2) \sigma_{\tau}^2 ] dt + \mathbb{E} \left( \int_{t_{i-1}}^{t_i} ( \sigma_t^2 - \sigma_{\tau}^2 ) dt \right) ^2 \\&
=:
\Delta ^2 \mathbb{E}[\sigma_{\tau}^4] + B_i + C_i .
\end{align*}
To {analyze} the contribution of {each of the three terms above} to $V_1$, we use Lemma \ref{D1Lemma} and \ref{D2Lemma} with kernel function $K^2$ and the following three different {functions $f$}:
$$f(t) = 1, \quad f(t) = \sqrt{ \mathbb{E} [ ( \sigma_t^2 - \sigma_{\tau}^2 )^2 ] \mathbb{E} [ \sigma_{\tau}^4 ] }, \quad f(t) = \mathbb{E} [ ( \sigma_t^2 - \sigma_\tau^2 )^2 ],$$
respectively. {It then follows that}
\begin{align*}
\mathbb{E}[\sigma_{\tau}^4] \Delta ^2 \sum_{i=1}^n K^2_h(t_{i-1} - \tau)
& =
\mathbb{E}[\sigma_{\tau}^4] \frac{\Delta}{h} \sum_{i=1}^n K^2(\frac{t_{i-1} - \tau}{h}) \frac{\Delta}{h} 
= 
\frac{\Delta}{h} \mathbb{E}[\sigma_{\tau}^4] \int K^2(x)dx + o \left( \frac{\Delta}{h} \right) + o(h^\gamma), \\
\sum_{i=1}^n K^2_h(t_{i-1} - \tau) B_i
& \leq
2 \frac{\Delta}{h} \sum_{i=1}^n K^2 (\frac{t_{i-1} - \tau}{h}) \frac{1}{h} \int_{t_{i-1}}^{t_i} 
\sqrt{ \mathbb{E} [ ( \sigma_t^2 - \sigma_{\tau}^2 )^2 ] \mathbb{E} [ \sigma_{\tau}^4 ] } dt 
= o \left( \frac{\Delta}{h} \right) + o(h^\gamma), \\
\sum_{i=1}^n K^2_h(t_{i-1} - \tau) C_i
& \leq
\frac{\Delta}{h} \sum_{i=1}^n K^2 (\frac{t_{i-1} - \tau}{h})\frac{1}{h} \int_{t_{i-1}}^{t_i} \mathbb{E} [ ( \sigma_t^2 - \sigma_\tau^2 )^2 ] dt 
= 
o \left( \frac{\Delta}{h} \right) + o(h^\gamma),
\end{align*}
{where the second line above follows from the fact that $\mathbb{E} [ ( \sigma_t^2 - \sigma_{\tau}^2 )^2 ] = O(|t - \tau|^\gamma)$.} 
{Putting together the previous relationships}, we conclude that
$$
V_1 
=
\sum_{i=1}^n K_h^2(t_{i-1} - \tau) {\mathbb{E}\left[ \left(\int_{t_{i-1}}^{t_i} \sigma_t^2  dt\right)^2 \right]}
=
\frac{\Delta}{h} \mathbb{E}[\sigma_{\tau}^4] \int K^2(x)dx + o \left( \frac{\Delta}{h} \right) + o(h^{\gamma}).
$$
{Next,} applying directly Lemmas \ref{D1Lemma} and \ref{D2Lemma} and Assumption \ref{VolatilityConditionGeneral}, $V_2$ can be written as
$$
V_2 = h^{\gamma} \int\int K(x)K(y)C_\gamma(x,y;\tau) dxdy + o \left( \frac{\Delta}{h} \right) + o(h^{\gamma}).
$$
Finally, we conclude the following explicit asymptotic expansion for the MSE of our kernel estimator.

\begin{theorem}\label{Approximated_MSE_Theorem}
For {the} model (\ref{Asset_Dynamic}) with $\mu$ and $\sigma$ satisfying Assumptions \ref{IndependentCondition}, \ref{Boundedness_Condition},  and \ref{VolatilityConditionGeneral}{,} and a kernel function {$K$} satisfying {Assumption} \ref{AdmissibleKernel}, we have, for any {$\tau \in (0,T)$},
\begin{equation}\label{Approximated_MSE_Formula}
\begin{split}
\mbox{MSE}_{\tau, n, h}
& =
\mathbb{E}[(\hat{\sigma}^2_{\tau} - \sigma^2_{\tau} )^2] \\&
= 
2\frac{\Delta}{h} \mathbb{E}[\sigma_{\tau}^4] \int K^2(x)dx
 +
h^{\gamma} L(\tau) \iint K(x)K(y) C_\gamma (x,y) dxdy 
+
{o\left(\frac{\Delta}{h}\right)} + {o \left(h^{\gamma}\right)}.
\end{split}
\end{equation}
\end{theorem}

Theorem \ref{Approximated_MSE_Theorem} {will be} the main tool to obtain the approximated optimal bandwidth and kernel function. As a direct consequence, we also have the following consistency result for the kernel estimator.
\begin{corollary}\label{Corollary_Convergence_of_Estimator}
With the same assumptions as those in Theorem \ref{Approximated_MSE_Theorem}, $\|\hat{\sigma}^2_{\tau} - \sigma^2_{\tau}\| _{L_2} \rightarrow 0$ when {$h\to{}0$ and $\Delta/h\rightarrow 0$}.
\end{corollary}

It is not very hard to see from the previous proof that all $o(\cdot)$ terms are uniform for $\tau \in (0, T)$ {if the condition given by \eqref{SCDefinitionGeneral} is satisfied uniformly in $t$}, {and,} therefore, the following explicit asymptotic expansion for the {integrated mean-squared error (IMSE)} holds.
\begin{corollary}\label{Approximated_IMSE_Corollary}
For {the} model (\ref{Asset_Dynamic}) with $\mu$ and $\sigma$ satisfying Assumptions \ref{IndependentCondition}, \ref{Boundedness_Condition}, and \ref{VolatilityConditionGeneral}, {so that the term  $o ((r^2 + s^2)^{\gamma/2})$ in Eq.~(\ref{SCDefinitionGeneral}) is uniform in $t$,} and a kernel function {$K$ satisfying Assumption} \ref{AdmissibleKernel}, we have, for any $0 < a < b < T$,
\begin{equation}\label{Approximated_IMSE_Formula}
\begin{split}
\mbox{IMSE}_{n, h}
& :=
\int_a^b \mathbb{E}[(\hat{\sigma}^2_{t} - \sigma^2_{t} )^2] dt \\&
= 
2\frac{\Delta}{h} \int_a^b \mathbb{E}[\sigma_{t}^4] dt \int K^2(x)dx
 +
h^{\gamma} \int_a^b L(t) dt \iint K(x)K(y) C_\gamma (x,y) dxdy 
+
o \left(\frac{\Delta}{h}\right) + o (h^{\gamma}).
\end{split}
\end{equation}
\end{corollary}

\subsection{Approximated Optimal Bandwidth}\label{Subsection_Approximated_Optimal_Bandwidth}
Based on the {approximations above}, it is natural to {analyze the behavior of the} approximated MSE of the kernel estimator:
\begin{equation}\label{Approximated_MSE}
\begin{split}
\mbox{MSE}^{a}_{\tau, n, h}
:=
2\frac{\Delta}{h} \mathbb{E}[\sigma_{\tau}^4] \int K^2(x)dx
 +
h^{\gamma} L(\tau) \iint K(x)K(y) C_\gamma (x,y) dxdy  .
\end{split}
\end{equation}
Correspondingly, the approximated IMSE of the kernel estimator is defined {as}
\begin{equation}\label{Approximated_IMSE}
\begin{split}
\mbox{MSE}^{a}_{n, h} (a, b)
:=
2\frac{\Delta}{h} \int_a^b \mathbb{E}[\sigma_{t}^4] dt \int K^2(x)dx
 +
h^{\gamma} \int_a^b L(t)dt \iint K(x)K(y) C_\gamma (x,y) dxdy  .
\end{split}
\end{equation}

{Obviously, $2\mathbb{E}[\sigma_{\tau}^4] \int K^2(x)dx > 0$, while, by Assumption \ref{AdmissibleKernel}, we also have} that $L(\tau) \iint K(x)K(y) C_\gamma (x,y) dxdy > 0$.
{We then obtain} the following approximated optimal bandwidth:
\begin{proposition}\label{OptimalBandwidthProposition}
With the same assumptions as Theorem \ref{Approximated_MSE_Theorem}, the approximated optimal bandwidth, denoted by $h^{a, opt}_n$, which is defined to minimize the approximated MSE {defined in Eq.~}(\ref{Approximated_MSE}){,} is given by
\begin{equation}\label{OptimalBandwidth}
\begin{split}
h^{a, opt}_n = &
n^{-1/(\gamma + 1)} \left[ \frac{2 T \mathbb{E}[\sigma_{\tau}^4] \int K^2(x)dx} {\gamma L(\tau) \iint K(x)K(y) C_\gamma (x,y) dxdy} \right] ^{1/(\gamma + 1)},
\end{split}
\end{equation}
while the {attained global minimum of the} approximated MSE is given by
\begin{equation}\label{OptimalMSE}
\begin{split}
\mbox{MSE}^{a, opt}_n
& =
n^{-\gamma/(1+\gamma)} \left( 1 + \frac{1}{\gamma} \right)
\left( 2 T \mathbb{E}[\sigma_{\tau}^4] \int K^2(x)dx \right) ^{\gamma/(1+\gamma)} 
\\&\quad
\times \left( \gamma L(\tau) \iint K(x)K(y) C_\gamma (x,y) dxdy \right) ^{1/(1+\gamma)}.
\end{split}
\end{equation}
\end{proposition}

A direct yet considerably important consequence of Theorem \ref{Approximated_MSE_Theorem} and Proposition \ref{OptimalBandwidthProposition} is the following proposition about the optimal convergence rate. {This provides a} rigorous justification of the optimal convergence rate of the kernel estimator. It is worth mentioning that (4) of Assumption \ref{AdmissibleKernel} is necessary for this proposition.
\begin{proposition}\label{prop:optimal_convergence_rate_spot}
With the same assumptions as those in Theorem \ref{Approximated_MSE_Theorem}, the optimal convergence rate of the kernel estimator is given by $n^{-\gamma/(1 + \gamma)}$. This is attainable if the bandwidth is selected to be $h = O(n^{-1/(\gamma + 1)})$.
\end{proposition}

Corresponding to Corollary \ref{Approximated_IMSE_Corollary}, we have the following proposition for {the approximated ``uniform"} optimal bandwidth that minimizes {the} approximated IMSE.

\begin{proposition}\label{OptimalBandwidthGlobalProposition}
With the same assumptions as Corollary \ref{Approximated_IMSE_Corollary}, the approximated {optimal} homogeneous bandwidth, denoted by $\bar{h}^{a, opt}_n$, which is defined to minimize the approximated IMSE given by (\ref{Approximated_IMSE}){,} is given by
\begin{equation}\label{OptimalBandwidthGlobal}
\begin{split}
\bar{h}^{a, opt}_n = &
n^{-1/(\gamma + 1)} \left[ \frac{2 T \int_a^b \mathbb{E}[\sigma_{t}^4] dt \int K^2(x)dx} {\gamma \int_a^b L(t) dt \iint K(x)K(y) C_\gamma (x,y) dxdy} \right] ^{1/(\gamma + 1)},
\end{split}
\end{equation}
while the {attained} minimum of the approximated IMSE is given by
\begin{equation}\label{OptimalMSEGlobal}
\begin{split}
\mbox{IMSE}^{a, opt}_n (a, b)
& 
=
n^{-\gamma/(1+\gamma)} \left( 1 + \frac{1}{\gamma} \right)
\left( 2 T \int_a^b \mathbb{E}[\sigma_{t}^4] dt \int K^2(x)dx \right) ^{\gamma/(1+\gamma)} 
\\&\quad
\times \left( \gamma \int_a^b L(t) dt \iint K(x)K(y) C_\gamma (x,y) dxdy \right) ^{1/(1+\gamma)}.
\end{split}
\end{equation}
\end{proposition}

{It is worthwhile to draw some} connections with \cite{kristensen2010nonparametric} by considering the case of $\gamma = 2$, {which corresponds to a deterministic variance function $\sigma_t^2 = f(t)$ that is continuously differentiable at $\tau$ and such that} $f^\prime(\tau) \neq 0$. 
{In that case, the} approximated MSE (\ref{Approximated_MSE}) is given by
\begin{equation}\label{Approximated_MSE_gamma_two}
\begin{split}
\mbox{MSE}^{a}_{\tau, n, h}
=
2\frac{\Delta}{h} f^2(\tau) \int K^2(x)dx
 +
\left( hf^\prime (\tau) \int K(x)x dx \right)^2,
\end{split}
\end{equation}
{which coincides with the approximation obtained} in \cite{kristensen2010nonparametric}. 
{However, it is evident that, in the case} that the volatility is stochastic and non-smooth, our results are different from those in \cite{kristensen2010nonparametric}.
In Section \ref{section_optimal_kernel_deterministic}, we will see that in the case of deterministic and smooth volatility, we are able to use {``higher order" kernels} to improve the rate of convergence of the kernel estimator, but in other situations, for example BM type volatility, this is not possible. This is one of the major difference between our work and \cite{kristensen2010nonparametric}. Intuitively, this is due to the assumption of a stochastic volatility model, which in reality is more reasonable.

An important problem is to {formalize} the connection between the approximate optimal bandwidth  $h^{a,opt}_n$ and {the ``true"} optimal bandwidth, whenever it exists, which is denoted by $h^*_n$  and is {defined as a value of the bandwidth that minimizes} the actual MSE of the kernel estimator, $MSE_n(h) = \mathbb{E}[(\hat{\sigma}^2_{\tau,n,h} - \sigma^2_{\tau} )^2]$. In Appendix \ref{Sec:Equivalence}, we show that, under a mild additional condition, they are equivalent in the sense that
\begin{equation*}
\begin{split}
{ h^*_n} & = h^{a,opt}_n + o({h^{a,opt}_n}),\\
{MSE_n}(h^{a,opt}_n) & = \inf_h {MSE_n(h)} + o(\inf_h {MSE_n(h)}).
\end{split}
\end{equation*}

\section{Kernel Function Selection}\label{sec:kernel_function_selection}

{As an important application of the well-posed optimal bandwidth selection problem defined in Section \ref{section_bandwidth_selection}, we now proceed to consider the problem of selecting an optimal kernel function. Although the theoretical optimal convergence rate can be attained with {a bandwidth of the form} $h_n = C n^{-1/(\gamma + 1)}$, and we indeed obtained the coefficient $C$ that optimizes the first order approximation of the MSE of the kernel estimator for a given kernel function $K$, we can achieve {further} variance reduction by choosing an appropriate kernel function. This is particularly important for finite sample settings encountered in practice.

As shown by (\ref{OptimalMSE}), 
the optimal kernel function {only depends} on the covariance structure, $C_\gamma(\cdot, \cdot)$. 
There are two possible situations}. The first one is when $C_\gamma$ is positive definite. In such a case, we cannot improve the rate of convergence of the MSE, but we can minimize the constant {appearing} before the asymptotic MSE and IMSE. Another situation is when $C_\gamma$ is {simply non-negative definite}. In such a case, if we relax (4) of Assumption \ref{AdmissibleKernel}, it is possible to improve the rate of convergence {of the MSE} by choosing a so-called ``higher order" kernel function. 

{More concretely, in} this section, we consider three different cases. The first case, which is of fundamental importance in {finance}, is when $\gamma = 1$ and $C_\gamma(r,s) = 1_{\{rs>0\}} \min (|r|,|s|)$ {(Brownian driven volatilities)}. In such a case, an explicit form of the optimal kernel function can be obtained and an efficient algorithm is available for its implementation. The second case is when the covariance structure is given by (\ref{fBM_Vol_Cov_Structure}) with $\gamma \in (1, 2)$, {which can be obtained, for instance, when the volatility is driven by long-memory fBm's}. The final case is when $\gamma = 2$ and $C_\gamma(r,s) = rs$ {(e.g., deterministic smooth volatilities)}. Such a covariance structure is not positive definite, so it {will be} possible to use ``higher order" kernels to improve the rate of convergence.

\subsection{Optimal Kernel Selection {for a BM driven Volatility}}\label{Section_Optimal_Kernel_BM_Vol}
In this part, we consider the first case, i.e. the BM type volatility with $\gamma = 1$ and $C_1(r,s) = 1_{\{rs>0\}} \min (|r|,|s|)$. We will show that the exponential kernel function is the optimal kernel function. 

Exponential kernel function has been shown to be optimal for different problems in previous literature. For example, van Eeden (1985) showed that it is the optimal kernel function for {the} density estimation problem under some {conditions}. Foster and Nelson (1994) {argued} that the exponential kernel is the optimal kernel function to estimate spot volatility, under similar but different assumptions as we have. Their result is more general in the sense that they allow the leverage. However, their proof {lacks rigor}, due to their bounded support assumption on the kernel function. Also they did not draw any connection between optimal bandwidth and optimal kernel, while our results show that the optimal bandwidth and kernel function are jointly optimal.

To start with, from (\ref{OptimalMSE}), the objective function that we need to minimize is the following:
\begin{equation}\label{Object_Function_For_BM_Vol}
\begin{split}
I(K)
= 
\int K^2(x)dx
\int_0^{\infty}\int_0^{\infty} [K(x)K(y) + K(-x)K(-y)] \min(x,y) dxdy ,
\end{split}
\end{equation} 
with the restriction $\int K(x)dx = 1$. Here we notice that $I(K)$ is always positive, as shown by the proof of {Proposition} \ref{BMProcesses}. We divide the problem of minimizing (\ref{Object_Function_For_BM_Vol}) in three steps.

\subsubsection*{Step 1. Symmetric kernel}
Firstly, we claim that we only need to consider symmetric kernel functions. To this end, we prove that $I(K)\geq I(K_{s})$, where {$K_{s}(x) := \left(K(x) + K(-x)\right)/2$}. Indeed, we have $\int K_{s}(x)dx = 1$ and for the first factor of $I(K)$, 
\begin{equation}\label{Optimal_Kernel_Deduction_IK_First_Part}
\begin{split}
\int K_{s}^2(x)dx
\leq
\int \frac{1}{2}(K^2(x) + K^2(-x))dx
=
\int K^2(x)dx .
\end{split}
\end{equation} 
where the equality holds if and only if $K(x) = K(-x)$ for all $x \in \mathbb{R}$, i.e., $K$ is symmetric.

For the second factor of $I(K)$, let $J(K) = \int_0^{\infty}\int_0^{\infty} [K(x)K(y) + K(-x)K(-y)] \min(x,y) dxdy$ and note that
\begin{equation*}
\begin{split}
\quad J(K) - J(K_s)
& = 
\frac{1}{2}
\int_0^{\infty}\int_0^{\infty} [K(x)-K(-x)][K(y)-K(-y)] \min(x,y) dxdy \\
& = 
\frac{1}{2}
\int_0^{\infty}\int_0^{\infty} [K(x)-K(-x)][K(y)-K(-y)] \int_0^{\infty} 1_{\{t\leq x\}}1_{\{t\leq y\}} dt dxdy \\
& = 
\frac{1}{2}
\int_0^{\infty} \int_0^{\infty}\int_0^{\infty} [K(x)-K(-x)][K(y)-K(-y)] 1_{\{t\leq x\}}1_{\{t\leq y\}}  dxdy dt \\
& = 
\frac{1}{2}
\int_0^{\infty} \left[ \int_0^{\infty} [K(x)-K(-x)] 1_{\{t\leq x\}}  dx \right]^2 dt \geq 0.
\end{split}
\end{equation*} 
where the equality holds if and only if $K(x) = K(-x)$ for all $x \in \mathbb{R}$, i.e., $K$ is symmetric.
From here, we see that we only need to consider symmetric kernel functions, and the problem is changed to minimize
$$
\frac{1}{4} I(K)
=
\int_0^{\infty} K^2(x)dx
\int_0^{\infty}\int_0^{\infty} K(x)K(y) \min(x, y) dxdy  .
$$

\subsubsection*{Step 2. Changing to an equivalent optimization problem}

By writing $\min (x, y)$ as $\int_0^{\infty} 1_{\{ u \leq x \}}1_{\{ u \leq y\}} du$, we have
\begin{equation*}
\begin{split}
& \int_0^{\infty}\int_0^{\infty} K(x)K(y) \min(x,y) dxdy 
=
\int_0^{\infty} \left[ \int_{u}^{\infty} K(x) dx \right]^2 du .
\end{split}
\end{equation*} 
We define $L(u) = \int_{u}^{\infty} K(x) dx$ and note that, by definition of $K$, there are only finite many points where $L^{\prime}$ does not exist (note that $K$ is not assumed to be continuous, and at those points, {where} $K$ is not continuous, $L$ is not differentiable, though left and right derivatives exist).
Then, $\frac{1}{4}I(K)$ can be written as
\begin{equation*}
\begin{split}
\frac{1}{4} I(K)
= &
\int_0^{\infty} [L^{\prime}(x)]^2 dx 
\int_0^{\infty} [L(x)]^2 dx
=: I^*(L) ,
\end{split}
\end{equation*} 
and the problem is changed to minimize $I^*(L)$ for functions $L$ with the following restrictions:
\begin{enumerate}[(1)]
\item
$L$ is continuous and piece-wise twice differentiable on $\mathbb{R}_+$.
\item
$L(0) = \frac{1}{2}$ and $\lim_{x\rightarrow +\infty} L(x) = 0$.
\end{enumerate}
Note that the restrictions above are equivalent to the conditions we put on $K$. Since we are not assuming a non-negative kernel function, it is not necessary that $L$ is non-increasing.

\subsubsection*{Step 3. {Derivation} of the exponential kernel}

Using Cauchy-Schwartz inequality, we get
\begin{equation*}
\begin{split}
I^*(L)
\geq &
\left( \int_0^{\infty} L^{\prime}(x)L(x) dx \right)^2
=
\left( \int_{0}^{\infty} L(x) dL(x) \right)^2
=
\left( \int_{1/2}^{0} u du \right)^2
=
\frac{1}{64} ,
\end{split}
\end{equation*} 
where the first inequality becomes equality if and only if there exist non-zero constants $C_1$ and $C_2$ such that
$$
C_1 L'(x) + C_2 L(x) \equiv 0, \quad \mbox{for all } x \in \mathbb{R}_+.
$$
Notice that $L$ is continuous on $\mathbb{R}_+$ and $L(0) = \frac{1}{2} > 0$. So we have two possible cases: (1) there exists $x_0 > 0$, such that $L(x) > 0$, for all $x \in [0, x_0)$ and $L(x_0) = 0$; (2) $L(x) > 0$, for all $x \in \mathbb{R}_+$.

For the first case, we have that for $x \in (0, x_0)$,
$$
\frac{L'(x)}{L(x)} = \frac{-C_2}{C_1},
$$
whose solution is $L(x)=\frac{1}{2}e^{Bx}$. It is then impossible that $L(x_0) = 0$. Therefore, only the second case is possible. By solving the same differential equation, and together with $L(0) = \frac{1}{2}$ and $\lim_{x\rightarrow +\infty} L(x) = 0$, we have 
$$
L(x)=\frac{1}{2}e^{Bx}, \quad B<0.
$$
Therefore, the optimal kernel function is $K(x) = -\frac{B}{2}e^{Bx}$. Here, different values of $B$ do not change the value of $I(K)$, so we can choose $B = 1$ for simplicity.

As a summary, we have the following theorem for the optimal kernel function.
\begin{theorem}\label{OptimalKernelFunctionTheorem}
For {the} model (\ref{Asset_Dynamic}) with $\mu$ and $\sigma$ satisfying Assumptions \ref{IndependentCondition}, \ref{Boundedness_Condition} and \ref{VolatilityConditionGeneral}, where $C_\gamma$ is given by {(\ref{fBM_Vol_Cov_Structure})} with $\gamma = 1$, and for a kernel function $K$ satisfying {Assumption} \ref{AdmissibleKernel}, we have that the optimal kernel function that minimizes the first order approximation of the MSE of the kernel estimator is the exponential kernel function: 
$$K^{opt}(x) = \frac{1}{2} e^{-|x|}, \quad x \in \mathbb{R} .$$
\end{theorem}

We now do some calculations and demonstrate to what extent the exponential kernel decreases the MSE.
\begin{example}\label{Example_Different_Kernels}
As we can see from (\ref{OptimalMSE}), $MSE^{a,opt}_n = C \sqrt{I(K)}$, where the constant $C$ does not depend on the kernel function $K$. 
{Below, we show the value of $I(K)$ for the exponential, uniform, triangular, and the Epanechnikov kernels:
\begin{align*}
&I(.5\, e^{-|x|})= \frac{1}{36} \approx 0.0625,\quad 
I(.5\, 1_{\{|x| < 1\}}) = \frac{1}{12} \approx 0.08334,\\
&I( |1 - x| 1_{\{|x| < 1\}}) = \frac{1}{15} \approx 0.066,\quad 
I(.75\, (1 - x^2) 1_{\{|x| < 1\}}) = \frac{297}{4120} \approx 0.072.
\end{align*}
Interestingly} enough, the triangle kernel performs better than Epanechnikov kernel and Epanechnikov kernel performs better than the uniform kernel. The intuition behind this is that a kernel function with a shape more similar to the exponential kernel generally performs better.
\end{example}

\subsection{Efficient Implementation of the {Uniform} and Exponential Kernel}\label{Section_Exponential_Kernel_Algorithm}
In this subsection, we demonstrate that the exponential kernel function not only minimizes the MSE of the kernel estimator, {as shown in the previous subsection,} but also {enables us} to substantially reduce the {computational} complexity of the {volatility estimation}.

{Let us recall that $n$ denotes} the number of observations we have. {In general, the evaluation of $\hat{\sigma}_{\tau}$ for a fixed time $\tau\in(0,T)$ requires $O(n)$ (respectively,  $O(nh)$) computations for a kernel function with unbounded (respectively, bounded) support,} as long as the bandwidth {$h$} has already been fixed. However, if we hope to get an estimation of {the whole discrete skeleton $\{{\sigma}_{t_{i}}\}_{i=0,\dots,n}$ of the volatility, one would then require a time of $O(n^2)$ or $O(n^2h)$ for a general kernel function with unbounded or bounded support, respectively. In particular, if, in addition, we were to use the approximated optimal bandwidth given by (\ref{OptimalBandwidth})}, the best {possible complexity}, which is achieved by kernels with bounded supports, is $O(n^{2-1/(\gamma+1)})$. {This computational time might be substantially long considering our goal to use high frequency data}.

{We now show that, for both the uniform and exponential kernels}, we can do {substantially} better. Indeed, for the uniform kernel and any $\tau \in (0,T - \Delta)$, we can use the idea of moving average as the following:
$$
\hat{\sigma}_{\tau + \Delta ,unif}^2
=
\hat{\sigma}_{\tau,unif}^2
+ \frac{1}{2h} (\Delta_k X)^2
- \frac{1}{2h} (\Delta_j X)^2 ,
$$
where {$
k = \min \{ l : t_{l-1} \geq \tau + h \}$ and 
$j = \min \{ l : t_{l-1} \geq \tau - h \}$.}

For the exponential kernel, we write $K^{exp}_h(x) = \frac{1}{2h} e^{-|x|/h}$ and we {introduce} the following notations:
\begin{equation*}
\begin{split}
\hat{\sigma}_{\tau,-}^2 =  \sum_{i < i_0} K^{exp}_h(t_{i-1} - \tau)(\Delta_i X)^2 ,\quad
\hat{\sigma}_{\tau,*}^2 =  K^{exp}_h(t_{i_0-1} - \tau) (\Delta_{i_0} X)^2 ,\quad
\hat{\sigma}_{\tau,+}^2 =  \sum_{i > i_0} K^{exp}_h(t_{i-1} - \tau)(\Delta_i X)^2 , 
\end{split}
\end{equation*} 
where $t_{i_0 - 1}\leq \tau < t_{i_0}$. {Note that $\hat{\sigma}_{\tau,exp}^2 =  \hat{\sigma}_{\tau,+}^2 + \hat{\sigma}_{\tau,*}^2 + \hat{\sigma}_{\tau,-}^2 $.} Then, by the idea of geometric moving average, we can use the following recurrent algorithm:
\begin{equation}\label{Exponential_Iterate}
\begin{split}
\hat{\sigma}_{\tau + \Delta,-}^2 = & e^{-\Delta/h} \left[ \hat{\sigma}_{\tau,-}^2 + \hat{\sigma}_{\tau,*}^2 \right] , \\
\hat{\sigma}_{\tau + \Delta,*}^2 = & K^{exp}_h(t_{i_0} - (\tau + \Delta)) (\Delta_{i_0 + 1} X)^2 , \\
\hat{\sigma}_{\tau + \Delta,+ }^2  = & e^{\Delta/h} \left[ \hat{\sigma}_{\tau,+}^2 - K^{exp}_h(t_{i_0} - \tau ) (\Delta_{i_0 + 1} X)^2 \right] .
\end{split}
\end{equation}

It is then {clear that,} in order to estimate the ``whole path" of $\sigma_\tau^2$ {using an exponential kernel}, we need a time of $O(n)$, instead of $O(n^2)$ or $O(n^2h)$.
{The} difference between the two time complexities mentioned above is quite significant, since we are considering high frequency data. For example, {suppose that we want to compute the whole discrete skeleton of the volatility for a trading day with data every second so that $n = 23400$}. Let us also assume that we consider volatility process generated by Brownian motion, i.e. $\gamma = 1$. In such case, the recurrent algorithm {described above requires} about $10^4$ {computations, while the standard} algorithm requires about $10^6$ or $10^8$, for kernels with bounded or unbounded supports. Actually the recurrent algorithm is at least about $500 (\approx \sqrt{23400})$ times faster than the naive algorithm.

{In practice, kernel estimators suffer of biases at times closer to the boundary}. As proposed in \cite{kristensen2010nonparametric}, we can correct such boundary {effects} by using the following estimator:
\begin{equation}\label{Kernel_Estimator_with_Boundary_Effect}
\hat{\sigma}^b_{\tau,n,h}
=
\frac{\sum _{i=1}^n K_h(t_{i-1} - \tau) (\Delta_i^n X)^2}
{\Delta \sum _{i=1}^n K_h(t_{i-1} - \tau)} .
\end{equation}
where the superscript denotes boundary effect. We can still implement our fast estimation algorithm to calculate this estimator {since we only need to} calculate $\sum _{i=1}^n K_h(t_{i-1} - \tau)$, which can be calculated similarly as (\ref{Exponential_Iterate}) except that all $(\Delta_i X)^2$ are replaced by $1$.

{Another important problem, usually encountered in high frequency trading, is} to calculate the current spot volatility as quickly as possible, {based on the previous volatility estimate and the newly observed return}. Under such {a} setting, the exponential kernel can perform pretty well in both time and space complexity. Indeed, an ``online" type algorithm can be implemented by setting $\hat{\sigma}_{t_{i}}^2 = \exp(-\Delta /h) \left[ \hat{\sigma}_{t_{i-1}}^2 + (\Delta_i X)^2/(2h) \right]$.
Because of this, in order to update the estimation of the current volatility, we only need $O(1)$ time and space, instead of $O(n)$ or $O(nh)$ for other kernel functions with unbounded or bounded supports.

\subsection{Optimal Kernel Function {for a fBM driven Volatility}}
In this section, {we now consider a general} fBM covariance structure, i.e. $\gamma \in (1, 2)$ and $C_\gamma$ given by (\ref{fBM_Vol_Cov_Structure}). From (\ref{OptimalMSE}), our goal is to minimize
$$
I(K)
= 
\left( \int K^2(x)dx \right)^\gamma
\int \int K(x)K(y) C_\gamma(x,y) dxdy .
$$

Our first step is still to prove that we only need to consider symmetric kernel {functions. In particular, we have that $I(K_{s})\leq{}I(K)$, where $K_{s}(x):=\frac{1}{2}\left(K(x)+K(-x)\right)$. To this end, it is useful to note we can write $C_{\gamma}$ as follows 
$$
C_\gamma (x,y)
=
\int F_\gamma (x,u) F_\gamma (y,u) du .
$$
with $F_\gamma(x, y) := C \left( |x - y|^{\frac{\gamma - 1}{2}} \mbox{sgn}(x - y) + |y|^{\frac{\gamma - 1}{2}} \mbox{sgn}(y) \right)$, for a} certain constant $C$ {(see \cite{positive_definite_kernel} for details)}.
The first {factor} of $I(K)$ {can be treated as in} (\ref{Optimal_Kernel_Deduction_IK_First_Part}){, while, for the second factor,} since $C_\gamma(x,y) = C_\gamma(-x,-y)$, we have
\begin{equation}\label{eq:fBM_Symmetric_Kernel_has_Lower}
\begin{split}
& \iint K(x)K(y) C_\gamma(x,y) dxdy 
-
\iint K_{s}(x)K_{s}(y) C_\gamma(x,y) dxdy \\
&\quad =
\frac{1}{4}
\int \int [K(x)-K(-x)][K(y)-K(-y)] \int F_\gamma(x,u)F_\gamma(y,u) du dxdy \\
 &\quad =
\frac{1}{4}
\int \left[ \int [K(x)-K(-x)] F_\gamma(x,u) dx \right]^2 du \geq 0 .
\end{split}
\end{equation}
Therefore, we only need to consider symmetric kernel functions.

Unfortunately, the problem of finding an explicit form for the optimal kernel function is much more challenging in this case. Therefore, we instead seek for {a numerical method} to find the optimal kernel function. We notice that $C_{\gamma}(x,y) + C_{\gamma}(x, -y) = |x|^\gamma + |y|^\gamma - \frac{1}{2}|x+y|^{\gamma} - \frac{1}{2} |x-y|^{\gamma}$, {for any $x,y > 0$, and, thus,} our goal is then changed to minimize
\begin{equation}\label{Objective_Function_General}
I^*(K)
= 
\left( \int_0^{\infty} K^2(x)dx \right)^\gamma
\int_0^{\infty} \int_0^{\infty} K(x)K(y) D(x, y) dxdy .
\end{equation}
where $D(x, y) = |x|^\gamma + |y|^\gamma - \frac{1}{2}|x+y|^{\gamma} - \frac{1}{2} |x-y|^{\gamma}$.
Since our problem is unchanged with $K(x)$ scaled by {a small} bandwidth, we will limit the support of $K(x)$ to be $[0,1]$. Note that this excludes those kernels with unbounded support. However, since all unbounded support kernels can be approximated, to an arbitrary precision, by a kernel with a bounded support, we will limit our consideration to bounded support kernels at this point.

A way to solve such an optimization problem numerically is to approximate the kernel function $K$ by step functions and then use gradient descent to directly optimize (\ref{Objective_Function_General}). Indeed, the kernel function can be approximated by
$$
K_m (x) = {\frac{1}{\sum_{i=1}^{n}a_{i}}\sum_{i=1}^m a_i 1_{[\frac{i-1}{m},\frac{i}{m})}(x),\quad x\in[0,1],\; a_{i}\in\mathbb{R},\; i=1,\dots,n.}
$$
With such an approximation, {it is then natural to consider the following optimization problem:}
\begin{equation}\label{Numerical_Objective_Function}
\mbox{Minimize }
f(a) :=
\frac{
m^\gamma 
\left( \sum_{i=1}^m a_i^2 \right)^\gamma
\left( \sum_{i=1}^m \sum_{j=1}^m a_i a_j A_{ij} \right)
}
{ \left( \sum_{i=1}^m a_i \right)^{2\gamma + 2} }
\mbox{ for } a = (a_1,...,a_m), a_i \in \mathbb{R},1\leq i \leq m {,}
\end{equation}
{where} $A_{ij} = |x_i|^\gamma + |x_j|^\gamma - \frac{1}{2}|x_i+x_j|^\gamma - \frac{1}{2}|x_i-x_j|^\gamma$, {with} $x_i = {(i-0.5)/m}$, $1\leq i \leq m$.

We notice that this is an optimization problem with high dimensional independent variables. In fact, in order {for the resulting approximated optimal kernel to converge} to the true optimal kernel function, we need {$m\rightarrow \infty$}. However, the numerical optimization problem is still tractable, since the gradient can be calculated explicitly, {which allows us to use a gradient descent method to calculate the optimal kernel function numerically}. 
Of course, 
{there are some practical issues to consider when dealing with gradient descent.} The first problem is that {the method} may yield a local minima, but not the global minima. {In order to alleviate the latter issue, we choose} several initial values randomly {and select} the best final result. Another problem is how to determine the step size for each iteration. There are several standard ways to solve this problem. In our implementation, we first choose a step size that is large enough. If the objective function decreases when walking through the gradient direction {with the selected} step size, we update the point and go on to the next iteration. Otherwise, we shrink the step size to a half.

Figure \ref{Plot_Optimal_Kernel_Functions_With_Different_Gamma} shows the {resulting optimal} kernel functions for $\gamma = 1.0,1.3,1.6,1.9$. {Note that the resulting approximated optimal kernel for $\gamma = 1$ is consistent with true optimal kernel that was proved to be exponential in Section \ref{Section_Optimal_Kernel_BM_Vol}. 
We also} observe from Figure \ref{Plot_Optimal_Kernel_Functions_With_Different_Gamma} that, as $\gamma$ increases, the optimal kernel function becomes flatter and less convex. This indeed makes sense, since a higher $\gamma$ indicates less chaos of the volatility, and thus more weights should be given to data farther from the estimated point.

\begin{figure}[h]
\caption{Optimal Kernel Functions for Different $\gamma$}
\label{Plot_Optimal_Kernel_Functions_With_Different_Gamma}
\centering
   \includegraphics[scale=0.5]{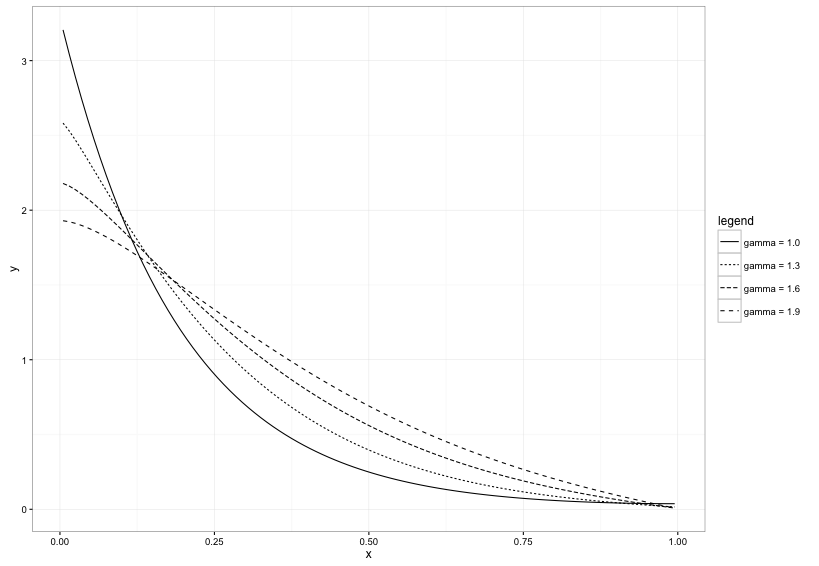}
\end{figure}

\subsection{Optimal Kernel for a Deterministic Volatility Function}\label{section_optimal_kernel_deterministic}
{Lastly, we consider the case $\sigma_{t}^{2}=V_{t}=f(t)$, for a deterministic function $f$. As seen in Subsection 
\ref{SubSect:Deterministic}, we have that $\gamma = 2$} and $C_2(r,s) = rs$. Obviously, such a $C_2$ is not strictly positive definite, so theoretically we can consider ``higher order" kernels to improve the convergence rate {of the estimation MSE. Specifically, we generalize the condition (4) of Assumption \ref{AdmissibleKernel} as follows:
\[
	\int K(x) x^i dx = 0, \quad i = 1, 2, ..., p - 1,\text{ and }\int K(x) x^p dx \neq 0.
\]
Such a kernel is said to be of order $p$. We also extend Assumption \ref{VolatilityConditionGeneral} as follows:}
$$
\mathbb{E} [ (V_{t+r} - V_{t})(V_{t+s} - V_{t}) ] 
= \sum_{i = 1}^{2p - 1} L_i(t)C_{i + 1}(r,s) + o ((r^2 + s^2)^{p}), \quad r,s \rightarrow 0,
$$
{where the function $C_{i}$ is such that} $C_i(hr,hs) = h^{i}C_\gamma(r,s)$, for {any $r, s \in \mathbb{R}$, $h \in \mathbb{R}_+$, and $i\geq{}1$}. {In that case, with a similar procedure as that of} Section \ref{section_bandwidth_selection}, we can prove that the approximated MSE (\ref{Approximated_MSE}) is given by
\begin{equation*}
\begin{split}
\mbox{MSE}^{a}_{\tau, n, h}
=
2\frac{\Delta}{h} \mathbb{E}[\sigma_{\tau}^4] \int K^2(x)dx
 +
h^{2p} \left( f^{(p)} (\tau) {\int} K(x) x^p dx \right)^2  .
\end{split}
\end{equation*}
We can then select {an optimal bandwidth. The optimal bandwidth and corresponding optimal convergence rate are the same as those obtained in \cite{kristensen2010nonparametric}.}

\begin{remark}
For the construction of higher order kernel functions, we refer to Section 1.2.2 in \cite{tsybakov2008nonparametric}. However, as was {already} pointed out by \cite{kristensen2010nonparametric}, ``higher order" kernels cannot be non-negative\footnote{{Indeed, it is not possible to have both $\int x^2 K(x) dx = 0$ and $K(x) \geq 0$, for all $x \in \mathbb{R}$.}} and, thus, in principle, may yield non-positive {estimates of the volatility.  
Therefore, there is some tradeoff} between accuracy and positivity.
\end{remark}

{An interesting application, which was not considered in \cite{kristensen2010nonparametric}, is to find the kernel that minimizes the resulting optimal approximated MSE. Concretely, if we limit ourselves to symmetric kernels of order $p$, the goal is to minimize}
\begin{equation}\label{objective_function_order_p}
\begin{split}
I_p(K) =
\left( \int_0^{\infty} K^2(x)dx \right) ^p \int_0^{\infty} K(x) x^p dx  .
\end{split}
\end{equation}
For such a problem, we {further} limit ourself to kernel function with support $[0, 1]$ and use calculus of variation to derive the optimal kernel function.
For any continuous function $\eta: [0, 1] \rightarrow \mathbb{R}$ {such that $\int_{0}^{1}\eta(x)dx=0$} and a real number $\epsilon$, we consider
\begin{equation*}
\begin{split}
I_p(K + \epsilon \eta) =
\left( \int_0^{\infty} (K(x) + \epsilon \eta(x) )^2 dx \right) ^p \int_0^{\infty} (K(x) + \epsilon \eta(x) ) x^p dx  .
\end{split}
\end{equation*}
{Next,} in order to find a local minimum point of $I(K)$, we take {the} derivative of $I$ with respect to $\epsilon$ {to get}
\begin{equation*}
\begin{split}
\left. \frac{\partial I_p}{\partial \epsilon} \right|_{\epsilon = 0} =
2p \left( \int_0^{\infty} K^2(x) dx \right)^{p - 1} \int_0^{\infty} K(x)\eta(x) dx \int_0^{\infty} K(x)x^p dx
+
\left( \int_0^{\infty} K^2(x) dx \right)^p \int_0^{\infty} \eta(x)x^p dx.
\end{split}
\end{equation*}
Then, we solve $\left. \frac{\partial I_p}{\partial \epsilon} \right|_{\epsilon = 0} = 0$, which is equivalent to solve
\begin{equation*}
\begin{split}
2p \int_0^{\infty} K(x)\eta(x) dx \int_0^{\infty} K(x)x^p dx
+
\int_0^{\infty} K^2(x) dx \int_0^{\infty} \eta(x)x^p dx = 0.
\end{split}
\end{equation*}
In order {for the above to hold} for any $\eta$ {satisfying the stated properties,} $K$ needs to have the form $K(x) = a(1 - bx^p)$ for $a,b > 0$. By plugging {such a $K$ in}, we get
\begin{equation*}
\begin{split}
a \int_0^{\infty} \eta(x)x^p dx \left(-2pb \times \left(\frac{1}{p + 1} - \frac{b}{2p + 1} \right) + 1 - \frac{2b}{p + 1} + \frac{b^2}{2p + 1} \right) = 0.
\end{split}
\end{equation*}
Solving such an equation yields a unique solution $b = 1$ and, by solving {$\int_0^1 K(x)dx = 1/2$,} we can get $a = \frac{2p}{p + 1}$. Therefore, we get a local minimum point of (\ref{objective_function_order_p}) as the following:
$$
K_p(x) = \frac{p + 1}{2p} (1 - |x|^p) 1_{[-1, 1]}.
$$
It is worth mentioning that when $p = 2$, we have $K_2(x) = \frac{3}{4} (1 - |x|^p) 1_{[-1, 1]}$, which is exactly the Epanechnikov kernel.

There is still a problem for such {a kernel}. Take $p = 4$ as an example. Although $K_4$ is a local minimum point of (\ref{objective_function_order_p}), it does not satisfy {$\int_{-1}^1 K(x)x^2dx = 0$}. 
Therefore, {we propose to consider instead the following} optimization problem with {constraints}:
\begin{equation}\label{objective_function_order_p_constraint}
\begin{split}
&\mbox{minimize }I_{2q}(K) =
\left( \int_0^1 K^2(x)dx \right) ^{2q} \int_0^1 K(x) x^{2q} dx, \\
&\mbox{subject to }\int_0^1 K(x) x^{2r} dx = 0, \quad \mbox{for }0 < r < q,
\quad \mbox{and }\int_0^1 K(x) dx = \frac{1}{2}.
\end{split}
\end{equation}
To solve such a problem, we consider {the Lagrangian}
\begin{equation}\label{objective_function_order_p_constraint_lagrangian}
\begin{split}
I_{2q}^c(K) &=
\left( \int_0^1 K^2(x)dx \right) ^{2q} \int_0^1 K(x) x^{2q} dx + \sum_{i = 1}^{q - 1} \lambda_i \left( \int_0^1 K^2(x)dx \right) ^{2q} \int_0^1 K(x) x^{2i} dx\\
&\quad + \lambda_0 \left( \int_0^1 K^2(x)dx \right) ^{2q} \left(\int_0^1 K(x) dx - \frac{1}{2} \right),
\end{split}
\end{equation}
where $\lambda_i$ are Lagrangian multipliers. In order to solve such an optimization problem, we set $\left. \frac{\partial I^c_{2q}(K + \epsilon \eta)}{{\partial\epsilon}}\right|_{\epsilon = 0} = 0$ and $\frac{\partial I^c_{2q}(K + \epsilon \eta)}{{\partial\lambda_i}}= 0$. {After some simplifications, these yield the system of equations:}
\begin{equation*}
\begin{split}
& 4q  \int_0^1 K(x)\eta(x) dx \int_0^1 K(x)x^{2q} dx
+
\int_0^1 K^2(x) dx \int_0^1 \eta(x) \left(x^{2q} + \sum_{i = 0}^{q - 1} \lambda_i x^{2i} \right) dx = 0,
\\
& \int_0^1 K(x) x^{2r} dx = 0, \quad \mbox{for }0 < r < q,\qquad \int_0^1 K(x) dx = \frac{1}{2}.
\end{split}
\end{equation*}
Therefore, {$K$ needs} to take the form $K(x) = a \left(x^{2q} + \sum_{i = 0}^{q - 1} \lambda_i x^{2i}\right)$. Then, by plugging in such a $K$, we get $q + 1$ {equations:
\begin{equation*}
\begin{split}
0&=(4q + 1) a \left( \frac{1}{4q + 1} + \sum_{i = 0}^{q - 1} \lambda_i \frac{1}{2(q + i) + 1} \right)+ \frac{\lambda_0}{2},
\\
0&= \frac{1}{2(q + r) + 1} + \sum_{i = 0}^{q - 1} \lambda_i \frac{1}{2(i + r) + 1}, \quad 0 < r < q,\\
\frac{1}{2}&= a\left( \frac{1}{2q + 1} + \sum_{i = 0}^{q - 1} \lambda_i \frac{1}{2i + 1} \right).
\end{split}
\end{equation*}
Solving} such a system of equations {for $a$ and $\lambda_{0},\dots, \lambda_{q-1}$ provide} us a candidate of global optimal kernel function.

\section{Plug-In Bandwidth Selection Methods}\label{sec:Plug-inBandwidthSelection}

{In this section we propose a feasible plug-in type bandwidth selection algorithm, {for which,} as a sub-problem, we {also develop a {new}} estimator of {the} volatility of volatility based on the kernel estimator of {the} spot volatility and {a type of} two-time scale realized variance estimator.}

{Let start by giving a brief overview of the different methods for bandwidth selections. There are generally two types of methods:} cross-validation and plug-in type methods. In \cite{kristensen2010nonparametric}, a leave-one-out cross validation method for determining the optimal bandwidth is proposed, which does provide good results, as shown by simulations (see Section \ref{Section_SimulationResults} below for further details). The advantage of the cross validation method is its generality and portability across different settings (e.g., different $\gamma$'s and $C_\gamma$'s). However, this method has two main drawbacks. On one hand, the method generally suffers from loss of accuracy. One reason is that the cross-validation method yields a single bandwidth for the whole time period, in spite of the fact that, as seen from {\eqref{OptimalBandwidth}}, the optimal bandwidth varies from time to time. Also, even if we {restrict ourselves to homogeneous bandwidths}, cross validation method is still not as accurate as {a properly implemented} plug-in type method. On the other hand, the cross validation method is computationally expensive since it involves to carry out a numerical optimization scheme to find the bandwidth. {More specifically, as mentioned in Section \ref{Section_Exponential_Kernel_Algorithm},} for each proposed bandwidth $h$, we need $O(n^2)$ steps to calculate the objective function for a general kernel function with unbounded support {(even though, as it was mentioned before, such a complexity can be reduced to $O(n)$ when using an exponential kernel)}. Therefore, even with a good initial guess, it would generally take quite a long time to find a satisfactory bandwidth by cross validation.

{As previously mentioned, we proceed to propose a feasible plug-in} bandwidth selection method based on the explicit formula of the 
{global bandwidth} {\eqref{OptimalBandwidthGlobal}}. The proposed method slightly sacrifices generality for the advantage of higher accuracy and faster speed. We shall focus on the case of {a} BM type volatility, while similar methods can be developed for other types of volatility structures. {For easy reference, let} us recall that the global approximated optimal bandwidth 
takes the form:
\begin{equation}\label{OptimalBandwidthGlobal2}
\begin{split}
\bar{h}^{a, opt}_n = &
\left[ \frac{2 T \mathbb{E}\left[\int_{0}^{T}\sigma_{\tau}^4d\tau\right] \int K^2(x)dx} { n \mathbb{E}\left[\int_{0}^{T}g^{2}(\tau)d\tau\right] \iint K(x)K(y) C_1 (x,y) dxdy} \right] ^{1/2}.
\end{split}
\end{equation}
{To implement (\ref{OptimalBandwidthGlobal2}), it is natural to first} use the integrated quarticity of $X$, $IQ(X) = \int_{0}^{T}\sigma^{4}_{\tau}d\tau$, and the quadratic variation of $\sigma^2$, $IV(\sigma^2) = \int_{0}^{T}g^{2}(\tau)d\tau$, instead of their expected values. {Intuitively, this approach makes the estimator more data adaptive and, in the absence of parametric} constraints, these are the only estimable quantities with only one path. As it is well known, {a popular} estimate for $\int_{0}^{T}\sigma^{4}_{\tau}d\tau$ is the so-called realized quarticity, which is defined by $\widehat{IQ} =(3\Delta)^{-1} \sum_{i = 1}^n (\Delta_i X)^4$. The estimation of $\int_{0}^{T}g^{2}(\tau)d\tau$ is a more subtle problem and, {below, we} propose an estimator, which we call Two-time Scale Realized Volatility of Volatility (TSRVV) and is hereafter {denoted by 
$\widehat{IV(\sigma^2)}^{tsrvv}$.}
With these estimators, the final bandwidth can {then} be written as
\begin{equation}\label{General_Ito_Process_Bannwidth_Final}
h^{a,opt}_n = \left[ \frac{2 T \widehat{IQ(X)} \int K^2(x)dx} { n \widehat{IV(\sigma^2)}^{tsrvv} \iint K(x)K(y) C_1 (x,y) dxdy} \right] ^{1/2}.
\end{equation}

\subsubsection*{Iterative Algorithm}
{The previous bandwidth estimator} involves the spot volatility itself, {through $\widehat{IV(\sigma^2)}^{tsrvv}$,} which, of course, we do not know in advance. To deal with this problem, we propose to use an iterative algorithm in the same spirit of a fixed-point type of procedure.
Concretely, we start with an initial ``guess" for the bandwidth. For example, we can take (\ref{General_Ito_Process_Bannwidth_Final}) and simply set $\widehat{IQ(X)} / \widehat{IV(\sigma^2)}^{tsrvv} = 1$, which gives:
\begin{equation}\label{Optimal_Bandwidth_Initial_Guess}
h^{init}_n = \left[ \frac{2 T \int K^2(x)dx} { n \iint K(x)K(y) C_1 (x,y) dxdy} \right] ^{1/2}.
\end{equation}
With such a bandwidth, we {can} obtain initial estimates of the spot volatility at all the grid points. Such an initial spot volatility estimation {can then be applied to compute $\widehat{IV(\sigma^2)}^{tsrvv}$, which, in turn, can be used to} obtain another estimation of the optimal bandwidth. This procedure {is} continued iteratively until {a} predetermined stopping criteria is met.
In reality, {our simulations in Section \ref{Section_SimulationResults} show that} one or two iterations are enough to {obtain a} satisfactory result and more {iterations do not generally produce any improvement}. Algorithm \ref{Algorithm_Iterative} below outlines the proposed procedure for a global bandwidth $h$ given by (\ref{General_Ito_Process_Bannwidth_Final}).

\begin{center}
\begin{algorithm}
 	\KwData{$\Delta^n_1 X = X_1 - X_0,...,\Delta^n_{n} X = X_n - X_{n-1}$}
 	Set initial value of $h$ according to (\ref{Optimal_Bandwidth_Initial_Guess}) \;
 \While{Stopping criteria not met}{
	Get the estimation of spot volatility at all grid points, $ \hat{\sigma}_{t_i}^2$, by using the current bandwidth $h$ and (\ref{Kernel_Estimator}) \;
	Update the bandwidth $h$ by plugging in the new estimation of spot volatility to (\ref{General_Ito_Process_Bannwidth_Final});
 }
 \caption{Iterative Algorithm of Plug-In Type Bandwidth Selection Method}\label{Algorithm_Iterative}
\end{algorithm}
\end{center}

\subsection{A Two-time Scale Estimator of Integrated Volatility of Volatility}\label{sec:VolVolEst}

In this {subsection}, we propose an estimator of the quadratic variation of {$\sigma^{2}$}, $IV(\sigma^2) = \int_{0}^{T}g^{2}(\tau)d\tau$, which is often referred to as the Integrated Volatility of Volatility (IVV) of $X$. {This is based on the natural idea of using the realized quadratic variation of some estimated spot volatility $\hat{\sigma}$.
However,} the estimated spot volatilities have errors and a direct construction of {the} realized quadratic variation will {be highly biased} due to the errors of the estimation. This is similar to the case of estimating the integrated volatility of an It\^o  semimartingale based on high-frequency observations subject to micro-structure noise{, which} has received a great deal of attention in the literature. In this part, we adapt the so-called Two-time Scale Realized Volatility (TSRV) estimator of \cite{Two_Time_Scale} to estimate the IVV. For completeness, we first introduce the idea of TSRV and, then {proceed to defined} our estimator of IVV.

The {theory} of microstructure noise, put forward by \cite{Two_Time_Scale} and others, postulates that the true log prices of the asset $\{X_t\}_{t \geq 0}$ cannot directly be observed from the market and, instead, the log prices with {an additive} noise term, denoted by $Y_{t} = X_t + \varepsilon_t$, are observed. In the simplest case, $\{\varepsilon_t\}_{t\geq{}0}$ is white-noise (i.e., independent identically distributed with mean $0$ and constant variance $\sigma_{\varepsilon}^{2}$), which is independent from the true price {process $\{X_t\}_{t\geq{}0}$}. 
In that case, the Realized Volatility estimator of the integrated volatility has a bias and variance of order $O(n \mathbb{E}[\varepsilon^2])$ and $O(n \mathbb{E}[\varepsilon^4])$, respectively. Indeed, the following result is proved by \cite{Two_Time_Scale}:
\begin{equation}
\begin{split}
  \mathcal{L}\left(\left.[ Y, Y]_T^{all}\right|\mu,\sigma\right)
\,\stackrel{\mathcal{D}}{\longrightarrow}\, \int_{0}^{T}\sigma_{t}^{2}dt
   + 2n \mathbb{E}[\varepsilon^2]
    +
   \left[ 4n\mathbb{E}[\varepsilon^4] + \frac{2T}{n}\int_0^T\sigma_t^4dt\right]^{\frac{1}{2}}Z,
\end{split}
\end{equation}
where $Z \sim N(0,1)$ and $\mathcal{L}(\cdot|\mu,\sigma)$ represents the conditional law given the whole path of $(\mu, \sigma)$. Here, we follow the notation used in \cite{Two_Time_Scale} to denote $[\cdot , \cdot ]^{all}_T$ the realized quadratic variation based on all the observations $0=t_{0}<t_{1}<\dots<t_{n}=T$.
As proposed in \cite{Two_Time_Scale}, one way to alleviate the problem is to average the realized variations obtained at coarser time scale as the following:
\begin{equation*}
[ Y, Y]_T^{avg}
:=
\frac{1}{k}\sum_{\ell=1}^{k}[Y,Y]_{T}^{(\ell)}
=
\frac{1}{k} \sum_{i= 0}^{n - k} (X_{t_{i + k}} - X_{t_{i}})^2,
\end{equation*}
where $[Y,Y]^{(\ell)}:=\sum_{1\leq{}j\leq(n-\ell+1)/k}(X_{t_{\ell-1 + jk}} - X_{t_{\ell-1+(j-1)k}})^2$.
Then, {the biased corrected Two-time Scale Realized Volatility (TSRV) estimator is then} defined as
\begin{equation}
 [ Y, Y]_T^{\text{(tsrv)}}:=[Y,Y]_T^{(avg)}-\frac{n - k + 1}{nk}[Y,Y]_T^{(all)}.
\end{equation}
Such a TSRV estimator can be proved to converge to the true Integrated Volatility and a convergence rate of $O_{p}(n^{-1/6})$ can be achieved with $k = O(n^{2/3})$ (see \cite[Theorem 4]{Two_Time_Scale}).

Back to our problem, {let us first note that, at each observation time $t_{i}$, the estimated spot volatility can be written as} $\hat{\sigma}_{t_i}^2 = \sigma_{t_i}^2 + e_{t_i}$, where $e_{t_i}$ is the estimation error. The difference of our problem and the problem in \cite{Two_Time_Scale} is that our estimation errors are not independent. In fact, they are correlated. Such a correlation becomes more significant when we take the difference $\Delta_i \hat{\sigma}^2 = \hat{\sigma}_{t_{i + 1}}^2  - \hat{\sigma}_{t_{i}}^2 $. 
To {alleviate} such a problem, we propose to use {one-sided} kernel estimators and take the difference between the right {and left side estimators to find $\Delta_i \hat{\sigma}^2$. Concretely, define} $\hat{\sigma}_{l, t_i}^2$ and $\hat{\sigma}_{r, t_i}^2$ to be the left and right side estimator of $\sigma_{t_i}^2$, {respectively, defined} as the following:
\begin{equation}\label{eq:LeftRightKernelEstimator}
\hat{\sigma}_{l, t_i}^2
=
\frac{\sum _{j > i} K_h(t_{j-1} - t_i) (\Delta_j^n X)^2}
{\Delta \sum _{j > i} K_h(t_{j-1} - \tau)},
\quad
\hat{\sigma}_{r, t_i}^2
=
\frac{\sum _{j \leq i} K_h(t_{j-1} - t_i) (\Delta_j^n X)^2}
{\Delta \sum _{j \leq i} K_h(t_{j-1} - \tau)}.
\end{equation}
{Next,} we define the following two difference terms: $\Delta_i \hat{\sigma}^2 = \hat{\sigma}_{r, t_{i + 1}}^2 - \hat{\sigma}_{l, t_i}^2$, $\Delta_i^{(k)} \hat{\sigma}^2 = \hat{\sigma}_{r, t_{i + k}}^2 - \hat{\sigma}_{l, t_i}^2$. {Finally,} we can construct the following Two-time Scale Realized Volatility of Volatility (TSRVV):
\begin{equation}\label{eq:TSRVV}
   \widehat{IVV}_T^{\text{(tsrvv)}}= \frac{1}{k} \sum_{i = b}^{n - k - b} (\Delta_i^{(k)} \hat{\sigma}^2)^2  - \frac{n - k + 1}{nk} \sum_{i = b + k - 1}^{n - k - b} (\Delta_i \hat{\sigma}^2)^2.
\end{equation}
Here, $b$ {is a small enough integer, when} compared to $n$. The purpose of introducing such a number $b$ is to alleviate the boundary effect of the {one sided} estimators, since, for instance, it is expected that $\hat{\sigma}_{l, t_i}^2$ will be more inaccurate as $i$ {gets} smaller. 
It is worth to {noting} that
$\widehat{IVV}_T^{\text{(tsrvv)}}$
might become negative. In this case, a possible solution is to take simple
$\widehat{IVV}_T^{\text{(tsrvv)}}= \frac{1}{k} \sum_{i = b}^{n - k - b} (\Delta_i^{(k)} \hat{\sigma}^2)^2$.
Similar {to} \cite{Two_Time_Scale}, we can simply take $k = n^{2/3}$ in our case. There is some work to do if one wants to optimize such a TSRVV estimator, but this is outside the scope of the present work. Nevertheless, as our simulations in Section \ref{Section_SimulationResults} show, the accuracy of spot volatility is good enough even without refining such a TSRVV estimator.

{The following result shows the consistency of our estimator and shed some light on the rate of convergence. Its proof is deferred to the Appendix section.}
\begin{theorem}\label{CstTSVV}
For the model (\ref{Asset_Dynamic}) with $\mu$ and $\sigma$ satisfying Assumptions \ref{IndependentCondition}, \ref{Boundedness_Condition}, and $\sigma$ being an squared integrable It\^{o} process {as in Eq.~}\eqref{eq:Ito_process} (thus satisfying Assumption \ref{VolatilityConditionGeneral}), and a kernel function $K$ satisfying Assumption \ref{AdmissibleKernel}, for any fixed $t_b \in (0, T/2)$, \eqref{eq:TSRVV} is a consistent estimator of $\int_{t_b}^{T - t_b} g_t^2 dt$ with $b = t_b / \Delta$. The convergence rate is given by $O_p(\frac{n^{1/4}}{k^{1/2}}) + O_p(\sqrt{\frac{k}{n}})$ and, {thus, $k$ can be chosen to be of the form $Cn^{3/4}$ so that to attain the `optimal' convergence rate $n^{-1/8}$}.
\end{theorem}

\begin{remark}
(1) The proof {of Theorem \ref{CstTSVV}} actually holds even if we use {two-sided} kernel estimation. {The main reason for using one-side kernels is to correct the  estimation bias.}
{Also, according} to \cite{Two_Time_Scale}, \eqref{eq:TSRVV_second_part_error} actually converge to a normal distribution. Therefore, although we have not {investigated in detail the} asymptotic distribution of \eqref{eq:TSRVV_first_part_error}, {it is expected} that for any $\epsilon > 0$, with $k = Cn^{3/4 + 2\epsilon}$, $\mbox{TSRVV} - \int_0^T \Lambda_t^2 dt$ converges to normal distribution with a rate $n^{-1/8 + \epsilon}$.

(3) The estimation of integrated volatility of volatility {has also been} studied in other literature, for example, \cite{vetter2015volvolestimation}, {using a different method}. In the simulation study, we compare these two methods and found that our method works better under some widely used settings. Intuitively, the reason is that our estimation of volatility of volatility is based on more accurate estimation of spot volatility and our iterative methods further help enhance the accuracy.
\end{remark}

\section{Central Limit Theorems}\label{Section:CLT}

In this section, we seek to characterize the limiting distribution of the estimation error of the kernel estimator and prove feasible Central Limit Theorems (CLT) for the estimation error of kernel estimators. The {starting point} is to decompose the error into the following two parts:
\begin{equation}\label{eq:CLT_two_errors}
\begin{split}
\sum_{i=1}^n K_h(t_{i-1} - \tau) (\Delta_i X)^2 - \sigma_\tau^2
& =
\sum_{i=1}^n K_h(t_{i-1} - \tau) (\Delta_i X)^2 - \int_0^T K_h(t - \tau) \sigma_t^2 {dt}\\
&\quad +
\int_0^T K_h(t - \tau) (\sigma_t^2 - \sigma_\tau^2) dt
+
o_p \left(\left(\frac{\Delta}{h}\right)^{1/2}\right) + o_p \left(h^{\gamma / 2}\right).
\end{split}
\end{equation}
In order to obtain a CLT {for} the kernel estimator, we need to deal with the two error terms {above}. The first error term is easier to handle and has already been studied in the literature of kernel estimation of spot volatility (see, e.g., \cite{kristensen2010nonparametric}). By contrast, the limiting distribution of the second error term of \eqref{eq:CLT_two_errors} is more involved. {We can find two general type of results in the literature}:
\begin{enumerate}[(1)]
\item 
In the case that $\sigma_t^2$ follows an It\^{o} process, the limiting distribution of the second error term can be determined by Martingale Central Limit Theorems (cf. \cite{foster1994continuous_optKernel} Theorem 1 and 2).
\item A second approach consists of using a {`suboptimal'} bandwidth so that only the first error term in \eqref{eq:CLT_two_errors} is significant. This would be the case if, for instance, we choose $h = o(\Delta ^{1/(\gamma + 1)})$, in which case {the order of the second term in (\ref{eq:CLT_two_errors}) becomes {$O(h^{\gamma})$ and is} negligible compared to the order of the first term, {which is $o(\Delta/h)$}.} Instances of this {type of results} can be found in \cite{fan2008spot} (see Assumption A4 and Theorem 1 therein), and \cite{kristensen2010nonparametric} (Theorem 3 therein). 
\end{enumerate}
The {two previous} approaches have some obvious {limitations}. The results obtained using the first approach only deal with one level of smoothness in the volatility process, while the results obtained using the second approach can only yield suboptimal convergence rates.
In this work, we obtain {a CLT} for two {relatively broad frameworks} that are closely related to Assumption \ref{VolatilityConditionGeneral} and cover all the examples mentioned in Section \ref{Common_Volatility_Processes_Section}. On both cases, the CLT {has the} optimal convergence rate, and the second case covers a wide range of models of different smoothness order.

We begin with {an analysis of} the first error term, {which, as mentioned above, has already been studied in the literature (see, e.g., \cite{kristensen2010nonparametric}). Concretely, we have the following CLT for this term. A sketch of the proof is also provided for the sake of completeness.}
\begin{theorem}\label{thm:CLT_Error1}
For the model (\ref{Asset_Dynamic}) with $\mu$ and $\sigma$ satisfying Assumptions \ref{IndependentCondition}, \ref{Boundedness_Condition}, and, \ref{VolatilityConditionGeneral}, and a kernel function $K$ satisfying Assumption \ref{AdmissibleKernel}, we have, for any $\tau \in (0,T)$,
\begin{equation}\label{eq:CLT_Error1_version1}
\left( \frac{\Delta}{h} \right)^{-1/2} \sum_{i=1}^n K_h(t_{i-1} - \tau) \left( (\Delta_i X)^2 - \int_{t_{i - 1}}^{t_i} \sigma_t^2 dt \right) 
\rightarrow _D
\delta_1 N(0,1){,}
\end{equation}
where $\delta_1^2 = 2 \sigma_\tau^4 \int K^2(x)dx$. The following version also holds:
\begin{equation}\label{eq:CLT_Error1_version2}
\left( \frac{\Delta}{h} \right)^{-1/2} \left[ \sum_{i=1}^n K_h(t_{i-1} - \tau) (\Delta_i X)^2 
-
\int_0^T K_h(t - \tau) \sigma_t^2 dt
\right]
\rightarrow _D
\delta_1 N(0,1).
\end{equation}
\end{theorem}

\begin{proof}
In the proof, we condition on the whole path of $(\mu, \sigma)$ so that we can assume that these processes are deterministic. 
{Let us start by noting the following relationship, which can be justified by Lemma \ref{D1Lemma}{:}
\begin{equation}\label{eq:CLT_prep}
\sum_{i=1}^n K_h(t_{i-1} - \tau) \int_{t_{i - 1}}^{t_i} \sigma_t^2 dt - \int_0^T K_h(t - \tau) \sigma_t^2 dt
=
o_p \left(\left(\frac{\Delta}{h}\right)^{1/2}\right).
\end{equation}
Next, define} 
$$
Y_{n, i} = \left( \frac{\Delta}{h} \right)^{-1/2} K_h(t_{i-1} - \tau) \left( (\Delta_i X)^2 - \int_{t_{i - 1}}^{t_i} \sigma_t^2 dt \right).
$$ 
Then, for each $n$, $\{Y_{n, i} : 1 \leq i \leq n\}$ are independent. By applying Lemma \ref{D1Lemma} and with some similar calculation as {in} the proof of Theorem \ref{Approximated_MSE_Theorem}, we have the following as $n \to \infty$:
\begin{equation*}
\begin{split}
\mathbb{E} (Y_{n, i}) = 0, \quad
\sum_{i = 1}^n \mathbb{E} (Y_{n, i}^2)
=
2 \sigma_{\tau}^4 \int K^2(x)dx + o (1), \quad
\sum_{i = 1}^n \mathbb{E} (Y_{n, i}^2 1_{\{Y_{n, i} > \epsilon\}} ) = o(1).
\end{split}
\end{equation*}
Therefore, \eqref{eq:CLT_Error1_version1} follows by Lindeberg-Feller Theorem. By using \eqref{eq:CLT_prep} {together with (\ref{eq:CLT_Error1_version1})}, \eqref{eq:CLT_Error1_version2} follows.
\end{proof}

Next, we consider the second error term in (\ref{eq:CLT_two_errors}).
\begin{theorem}\label{thm:CLT_Error2}
Suppose that the coefficient processes $\mu$ and $\sigma$ in the model (\ref{Asset_Dynamic}) satisfy Assumptions \ref{IndependentCondition}, \ref{Boundedness_Condition}, and \ref{VolatilityConditionGeneral}, and the kernel function $K$ satisfies Assumption \ref{AdmissibleKernel}. Furthermore, suppose that either one of the following conditions holds:
\begin{enumerate}[(1)]
\item
$\{\sigma^2_t\}_{t\geq{}0}$ is an It\^{o} process given by
$\sigma_t^2 = \sigma_{0}^{2}+\int_0^t {f_{s}} ds + \int_0^t {g_{s}} dW_s$, {where we {further assume that} $\sup_{t \in [0, T]}\mathbb{E}[|f_t|] < \infty$, $\sup_{t \in [0, T] } \mathbb{E} [ g_{t}^{2} ] < \infty$, and $\mathbb{E} [ ( g_{\tau + h} - g_\tau )^2 ] \to 0$ as $h \to 0$.}
\item 
$\sigma^2_t = f(Z_t)$, for a deterministic function $f:\mathbb{R}\to\mathbb{R}$ and a Gaussian process $\{Z_t\}_{t\geq{}0}$ satisfying all requirements of Proposition \ref{proposition:C_gamma_of_function_of_gaussian_process}.
\end{enumerate}
{Then, on an extension $(\bar{\Omega}, \bar{\mathscr{F}}, \bar{\mathbb{P}})$ of the probability space $(\Omega, \mathscr{F},\mathbb{P})$, equipped with a standard normal variable $\xi$ independent of $\{Z_{t}\}_{t\geq{}0}$, we have, for each $\tau \in (0,T)$,}
\begin{equation}\label{eq:CLT_error2}
h^{-\gamma/2} \left( \int_0^T K_h(t - \tau) (\sigma_t^2 - \sigma_\tau^2) dt \right) \rightarrow_{D}\, \delta_2 \xi,
\end{equation}
where, under the condition (1) above, $\delta_2^2 = g(\tau, \omega)^2 \iint K(x) K(y) C(x, y) dxdy$, while, under the condition (2), $\delta_2^2 = [f^\prime(Z_\tau)]^2 L^{(Z)}(\tau) \iint K(x)K(y) C^{(Z)}_\gamma (x,y) dxdy$. 
\end{theorem}

\begin{proof}
\textbf{(1)}
{The proof in the It\^o process setting of (1) is inspired by that of Theorems 1 and 2 in \cite{foster1994continuous_optKernel}, but in our case, we do not assume a bounded support for the kernel function, and work directly with the continuous model, which makes the assumptions and proof clearer.} For simplicity, we will use the following notations:
$$
V_t = \sigma_t^2 = \sigma_{0}^{2}+\int_0^t f_s ds + \int_0^t g_s dW_s , \quad
v_t = \sigma_{0}^{2}+\int_0^t g_s dW_s.
$$
It is easy to see from Lemma \ref{BMProcesses} that $V$ and $v$ both satisfy Assumption \ref{VolatilityConditionGeneral} with $\gamma^V = \gamma^v = 1$ and $C_\gamma^V = C_\gamma^v$. Now, {since
$$
h^{-1/2} \mathbb{E} \left| \int_0^T K_h(t - \tau) \int_\tau^t f_s ds dt \right|
\leq
\sup_{s \in [0, T]} \mathbb{E} [|f_s|] h^{-1/2} \int_0^T | K_h(t - \tau)| |t - \tau| dt 
=
O(h^{1 / 2}) = o(1),
$$
we} can conclude that the drift term of $V$ has {a} negligible contribution to the final error, i.e. 
\begin{equation}\label{Eq:SS1}
h^{-1/2} \left( \int_0^T K_h(t - \tau) (V_t - V_\tau) dt \right) -
h^{-1/2} \left( \int_0^T K_h(t - \tau) (v_t - v_\tau) dt \right) 
=
o_P(1).
\end{equation}
Therefore, {it suffices to work with the process $v$ and only to consider the weak convergence of the following:
	\[
		I:=h^{-1/2} \left( \int_{\tau}^{T} K_h(t - \tau) (v_t - v_\tau) dt \right).
	\]
For the sake of clarity}, we will first assume a one-sided kernel function, i.e. $K(x) = 0$ for all $x < 0$. Applying the integration by parts formula, we have that
	\begin{align*}
		I&=h^{-1/2}\left(-L\left(\frac{t-\tau}{h}\right)(v_{t}-v_{\tau})\Big|_{t=\tau}^{t=T}+\int_{\tau}^{T}L\left(\frac{t-\tau}{h}\right)dv_{t}\right)\\
		&=-h^{-1/2}L\left(\frac{T-\tau}{h}\right)(v_{T}-v_{\tau})+h^{-1/2}\int_{\tau}^{T}L\left(\frac{t-\tau}{h}\right)g_{t}dW_{t}
		=:R+S,
	\end{align*}
	where $L(t)=\int_{t}^{\infty}K(u)du$ so that $\frac{d}{dt}(L((t-\tau)/h))=-K_{h}(t-\tau)$.
Since our assumptions on $K$ imply that $x^{1/2}L(x) \rightarrow 0$, as $x \rightarrow 0$
\footnote{{
Indeed, we have assumed that $x^{\gamma + 1}K(x) \rightarrow 0$, as $x \to \infty$, where in this case $\gamma = 1$. Then, by L'Hopital, it is easy to check that we have that $\lim_{x \to +\infty}L^2(x)/x^{-2} =0$.}},
we can conclude that $R = o_P(1)$. For the other term $S$, let us consider the following approximation
		\[
			\widetilde{S}:=h^{-1/2}g_{\tau}\int_{\tau}^{T}L\left(\frac{t-\tau}{h}\right)dW_{t},
		\]
We first observe that $S-\widetilde{S}=o_{P}(1)$. Indeed, Assumption \ref{AdmissibleKernel} implies that $\int_0^\infty L^2(x) dx < \infty$, so we have
		\begin{align*}
\mathbb{E}\left[\left(S-\widetilde{S}\right)^{2}\right]
&=
\frac{1}{h} \left( \int_{\tau}^{\tau + \sqrt{h}} + \int_{\tau + \sqrt{h}}^{T} \right)
L^{2}\left(\frac{t-\tau}{h}\right)\mathbb{E}\left[\left(g_{t}-g_{\tau}\right)^{2}\right]dt \\
&\leq
\sup_{t \in [\tau, \tau + \sqrt{h}]} \mathbb{E}\left[\left(g_{t}-g_{\tau}\right)^{2}\right] \int_{0}^{\infty} L^2(s) ds 
+
4 \sup_{t \in [\tau, T] } \mathbb{E} ( g_{t}^{2} ) \int_{1/\sqrt{h}}^{\infty} L^2(s) ds
=
o(1), \quad h \rightarrow 0.
\end{align*}
We also observe that conditional on $\mathcal{F}_{\tau}$, $\tilde{S}$ is Gaussian with mean $0$ and the following variance:
$$
g_{\tau}^{2}h^{-1}\int_{\tau}^{T}L^{2}\left(\frac{t-\tau}{h}\right)dt
=
g_{\tau}^{2}\int_{0}^{\frac{t-\tau}{h}}L^{2}(s)ds
\to
g_{\tau}^{2} \iint K(x) K(y) C_\gamma (x, y) dx dy,\quad h \to 0_+,
$$
where $C(x, y) = \min(|x|, |y|)1_{\{ xy > 0 \}}$.
Therefore, $\tilde{S} | \mathcal{F}_{\tau} \rightarrow _D \mathcal{N}(0, \delta_2^2)$ where $\delta_2^2 = g_{\tau}^{2}\iint K(x) K(y) C_\gamma (x, y) dx dy$. 

We now {consider the} general two-sided kernel case. {Let 
\[   
\bar{L}(t) =
	\left\{
	\begin{array}{ll}
      	\int_{-\infty}^{t}K(u)du, & t \leq 0, \\
     	\int_t^{\infty}K(u)du, & t > 0, \\
	\end{array} 
\right. 
\]
and note that, by the integration by parts formula,} 
	\begin{align*}
		\bar{I}
		&:=
		h^{-1/2} {\int_{0}^{T}} K_h(t - \tau) (v_t - v_\tau) dt\\
		&=
		h^{-1/2}\left(\bar{L}\left(\frac{t-\tau}{h}\right)(v_{t}-v_{\tau})\Big|_{t=0}^{t=\tau}
		-
		\int_{0}^{\tau}\bar{L}\left(\frac{t-\tau}{h}\right)dv_{t}\right) 
		+
		\tilde{S} + o_P(1)\\
		&=
		h^{-1/2}\bar{L}\left(\frac{T-\tau}{h}\right)(v_{T}-v_{\tau})
		-
		h^{-1/2}\int_{0}^{\tau}\bar{L}\left(\frac{t-\tau}{h}\right)g_{t}dW_{t}
		+
		\tilde{S} + o_P(1)\\
		&=:
		\bar{R} - \bar{S} + \tilde{S} + o_P(1).
	\end{align*}
Same as {in} the one-sided kernel case, our assumptions imply $\bar{R} = o_P(1)$. For $\bar{S}$, we still consider the following approximation
		\[
			\widetilde{\bar{S}}
			:=
			h^{-1/2}g_{\tau}\int_{0}^{\tau}\bar{L}\left(\frac{t-\tau}{h}\right)dW_{t}
			=
			h^{-1/2} g_\tau \left( \int_{0}^{\tau - \sqrt{h}} + \int_{\tau - \sqrt{h}}^\tau \right) 
			\bar{L}\left(\frac{t-\tau}{h}\right)dW_{t}
			=: 
			\widetilde{\bar{S}}_1 + \widetilde{\bar{S}}_2,
		\]
and we still have $\bar{S} - \widetilde{\bar{S}} = o_P(1)$. It is also true that $\widetilde{\bar{S}}_1$ vanishes as $h \to 0$. This can be justified by considering its second moment as the following:
		\begin{align*}
		\mathbb{E}\left[{\widetilde{\bar{S}}^{2}_1}\right]
		=
		h^{-1} \mathbb{E} [g_{\tau}^2] \int_{0}^{\tau - \sqrt{h}} 
		\bar{L}^2 \left(\frac{t-\tau}{h}\right) dt
		\leq
		\mathbb{E} [g_{\tau}^2] \int_{-\infty}^{-h^{-1/2}} \bar{L}^2(x) dx
		=
		o(1), \quad h \to 0,
	\end{align*}
Therefore,  we have
\begin{align*}
\bar{I} 
= 
h^{-1/2} g_{\tau} \int_{\tau - \sqrt{h}}^{T}\bar{L}\left(\frac{t-\tau}{h}\right)dW_{t} + o_P(1)
=
h^{-1/2} g_{t - \sqrt{h}} \int_{\tau - \sqrt{h}}^{T}\bar{L}\left(\frac{t-\tau}{h}\right)dW_{t} + o_P(1)
=:
\widetilde{\bar{I}}_{h} + o_P(1).
\end{align*}
where the second equality holds since
$$
{
\mathbb{E} \left[ \left| 
h^{-1/2} ( g_{t - \sqrt{h}} - g_{\tau} ) \int_{\tau - \sqrt{h}}^{T}\bar{L}\left(\frac{t-\tau}{h}\right)dW_{t}
\right| \right]^{2}
\leq 
\mathbb{E} \left[ ( g_{t - \sqrt{h}} - g_{\tau} )^2 \right]
\mathbb{E} \left[ h^{-1} \int_{\tau - \sqrt{h}}^{T}\bar{L}^2 \left(\frac{t-\tau}{h}\right)dt \right]
= 
o(1), \quad h\to 0}.
$$
{Note that $\mathcal{L}(\widetilde{\bar{I}}_h | g_{\tau} )
\to
\mathcal{N}\left(0,g_{\tau}^{2}\iint K(x)K(y) C(x, y) dxdy \right)$} since
\begin{align*}
	\mathbb{E}\left[\exp\left( iu\widetilde{\bar{I}}_{h} \right) \right]
	&=
	\mathbb{E}\left[\mathbb{E}\left[ \exp\left( iu\widetilde{\bar{I}}_{h} \right) \Big|\mathcal{F}_{\tau-\sqrt{h}}\right]\right]
	=
	\mathbb{E}\left[ \exp\left( -\frac{u^2 g_{ \tau-\sqrt{h} }^{2}}{2h}\int_{\tau-\sqrt{h}}^{T}\bar{L}^2\left(\frac{t-\tau}{h}\right)dt \right) \right]\\
	&=
	\mathbb{E}\left[ \exp\left( -\frac{u^2 g_{\tau-\sqrt{h} }^{2}}{2}\int_{-h^{-1/2}}^{\frac{T-\tau}{h}}\bar{L}^{2}\left(s\right)ds \right) \right]
	\stackrel{h\to{}0}{\longrightarrow}
	\mathbb{E}\left[ \exp\left( -\frac{u^2 g_{\tau}^{2}}{2}\int_{-\infty}^{\infty}\bar{L}^{2}\left(s\right)ds \right) \right] \\
	&= 
	\mathbb{E}\left[ \exp\left( -\frac{u^2 g_{\tau}^{2}}{2} \iint K(x)K(y) C(x, y) dxdy \right) 
	\right],
\end{align*}		
and, finally, we conclude that 
$
\mathcal{L}(\bar{I}| g_{\tau})\to
\mathcal{N}\left(0,g_{\tau}^{2}\iint K(x)K(y) C(x, y) dxdy \right)
$, as $h\to 0$. 

\textbf{(2)}
We then move on to consider case (2).
In the whole proof, the superscript $(Z)$ means that the quantity corresponds to process $Z$, while quantities without such a superscript corresponds to {the} process $\sigma^2$. {Let us start by noting that,} since $Z$ is a Gaussian process, we have
$$
h^{-\gamma/2} \left( \int_0^T K_h(t - \tau) (Z_t - Z_\tau) dt \right) 
\rightarrow_D 
\left( L^{(Z)}(\tau) \iint K(x)K(y) C^{(Z)}_\gamma (x,y) dxdy \right)^{1/2} N(0, 1).
$$
Now, for any $\epsilon \in (0, \min (\tau, T - \tau))$, and for any $t \in (\tau - \epsilon, \tau + \epsilon)$, there exists $s_t \in (\min(t, \tau), \max(t, \tau))$, such that $\sigma_t^2 - \sigma_\tau^2 = f^\prime(Z_\tau)(Z_t - Z_\tau) + \frac{1}{2} f^{\prime \prime} (Z_{s_t})(Z_t - Z_\tau)^2$. Then, we have
\begin{equation*}
\begin{split}
\int_0^T K_h(t - \tau) (\sigma_t^2 - \sigma_\tau^2) dt
=
\int_{\tau - \epsilon}^{\tau + \epsilon} K_h(t - \tau) [f^\prime(Z_\tau)(Z_t - Z_\tau) + \frac{1}{2} f^{\prime \prime} (Z_{s_t})(Z_t - Z_\tau)^2] dt + o(h^{\gamma / 2}).
\end{split}
\end{equation*}
For the second term, once we select $\epsilon$ small enough such that $\mathbb{E}[(f^{\prime \prime} (Z_{t}))^2] < M^2$ and $\mathbb{E}[{(Z_t - Z_\tau)^4}] \leq M |t - \tau|^\gamma$ for all $t \in (\tau - \epsilon, \tau + \epsilon)$, we have that
\begin{equation*}
\begin{split}
\mathbb{E} \left| \int_{\tau - \epsilon}^{\tau + \epsilon} K_h(t - \tau) f^{\prime \prime} (Z_{s_t})(Z_t - Z_\tau)^2 dt \right|
& \leq
\int_{\tau - \epsilon}^{\tau + \epsilon} |K_h(t - \tau)| \sqrt{\mathbb{E}[(f^{\prime \prime} (Z_{s_t}))^2] \mathbb{E}[(Z_t - Z_\tau)^4}] dt \\ 
& \leq
3 M^2 \int_{\tau - \epsilon}^{\tau + \epsilon} |K_h(t - \tau)| |t - \tau|^\gamma
=
O(h^\gamma) = o(h^{\gamma / 2}).
\end{split}
\end{equation*}
Now for the first term, we have
\begin{equation*}
\begin{split}
h^{-\gamma/2} \int_{\tau - \epsilon}^{\tau + \epsilon} K_h(t - \tau) [f^\prime(Z_\tau)(Z_t - Z_\tau)] dt
& =
f^\prime(Z_\tau) h^{-\gamma/2} \int_{\tau - \epsilon}^{\tau + \epsilon} K_h(t - \tau) (Z_t - Z_\tau) dt + o(1) \\
& \rightarrow _D
f^\prime(Z_\tau) \left( L^{(Z)}(\tau) \iint K(x)K(y) C^{(Z)}_\gamma (x,y) dxdy \right)^{1/2} N(0, 1).
\end{split}
\end{equation*}
where the standard normal $N(0, 1)$ appearing above is independent from $Z_\tau$. {The latter} convergence in distribution can be justified similar as Proposition \ref{proposition:C_gamma_of_function_of_gaussian_process}. Write 
$$
X = Z_\tau, \quad Y(h) = h^{-\gamma/2} \int_{\tau - \epsilon}^{\tau + \epsilon} K_h(t - \tau) (Z_t - Z_\tau) dt.
$$
We have that $(X, Y(h))$ is bi-variate normal for all $h > 0$ and, thus, {whenever the limit $(X, Y(h)) \rightarrow (X, Y)$ exists, $(X, Y)$ is bivariate normal  variable. There exist} $\alpha(h)$ and $\beta(h)$ such that $Y(h) = \alpha(h) X + \beta(h) Z(h)$, such that $X$ is independent with $Z(h)$ and $Z(h) \sim N(0, 1)$, for some $\alpha : \mathbb{R}_+ \rightarrow \mathbb{R}$ and $\beta : \mathbb{R}_+ \rightarrow \mathbb{R}_+$, as $h \rightarrow 0$. Note that $\alpha(h)$ and $\beta(h)$ are given by
$$
\alpha(h) = \frac{\mathbb{E}[XY(h)]}{\mathbb{E}[X^2]}, \quad \beta^2(h) = \mathbb{E}[Y^2(h)] - \alpha^2(h) \mathbb{E}[X^2].
$$
By our assumption, we have $\mathbb{E}[XY(h)] = o(1)$, which {implies that}
$$
\alpha(h) = o(1), \quad 
\beta^2(h) = \mathbb{E}[Y^2(h)] + o(1) = L^{(Z)}(\tau) \iint K(x)K(y) C^{(Z)}_\gamma (x,y) dxdy + o(1).
$$
With such representations, we are able to obtain the desired result:
$$
f^{\prime}(X)Y(h) 
= 
\alpha(h) f^{\prime}(X) X + \beta(h) f^{\prime}(X) Z(h)
=
o_p(1) + {\beta(h)} f^{\prime}(X) Z(h)
\rightarrow_D 
\beta f^{\prime}(X) Z.
$$
Here, $\beta^2 = \lim_{h \rightarrow 0} \beta^2(h) = L^{(Z)}(\tau) \iint K(x)K(y) C^{(Z)}_\gamma (x,y) dxdy$ and $Z$ is independent from $X$.
\end{proof}

\begin{remark}
It is {worth noting that 
in both of the cases covered in Theorem \ref{thm:CLT_Error2}, we have 
$$
\mathbb{E} [\delta_2^2] 
=
L^{(\sigma^2)}(\tau) \iint K(x)K(y) C^{(\sigma^2)}_\gamma (x,y) dxdy,
$$
which is exactly the coefficient of the second term appearing in \eqref{Approximated_MSE_Formula}. Similarly, the mean of the asymptotic conditional variance in the CLT of Theorem \ref{thm:CLT_Error1}, $E(\delta_{1}^{2})$, coincides with the coefficient of the first term in \eqref{Approximated_MSE_Formula}. Therefore, the CLT obtained from Theorems \ref{thm:CLT_Error1} and \ref{thm:CLT_Error2} are consistent with the asymptotic behavior of the MSE derived in Theorem \ref{Approximated_MSE_Theorem}. Note, however, that the framework of Theorem \ref{Approximated_MSE_Theorem} is more general.}
\end{remark}

\section{Simulation Results}\label{Section_SimulationResults}

In this section, we perform some simulation studies to further investigate the results that we developed before and compare our method with methods proposed in previous literature. Throughout, we will consider the Heston model (\ref{HestonTransformed}). As to the parameters values, we adopt the following widely used setting (cf. \cite{Two_Time_Scale}), unless otherwise specified: 
\[
	\kappa = 5,\quad  \theta = 0.04, \quad \xi = 0.5, \quad \mu_t = 0.05 - V_t / 2.
\] 
The initial values are set to be $X_0 = 1, \sigma_0^2 = 0.04$. We also assume both a non-leverage setting ($\rho=0$), as required by our Assumption \ref{IndependentCondition}, and a negative leverage situation ($\rho=-0.5$) to investigate the robustness of our method against non-zero $\rho$ values.

We will consider several different scenarios of observed data. First of all, we consider 5 days and 21 days data, which correspond to 1 week and 1 month data. For each trading day, we consider 6.5 trading hours and as {for} the frequency of the observations, we consider 5 minutes data and 1 minute data. Take 21 trading days and 5 minutes data, for example, totally we have $n = 21 \times 6.5 \times 12 + 1 = 1639$ observations and $T = 21/252$.

In order to alleviate boundary effect, we use estimator (\ref{Kernel_Estimator_with_Boundary_Effect}) throughout all the simulation. For each simulated discrete skeleton $\{ X_{t_i} : 0 \leq i \leq n, t_i = iT/n \}$, we estimate the corresponding discrete-skeleton of the variance process $\{ \sigma^2_{t_i} : 0 \leq i \leq n \}$, and calculate the average of the squared errors, $ASE = \frac{1}{n - 2 l + 1} \sum_{i = l}^{n - l} (\hat{\sigma}_{t_i}^2 - \sigma_{t_i}^2)^2$, for each simulation. (We incorporate an $l > 0$ to focus on evaluating the performance of the estimator without boundary effects. $l$ is generally taken to be $[0.1n]$.) Then, we take the sample average of such ASE's to estimate the mean ASE, defined as ${{\rm MASE}} = \mathbb{E}\left[ \frac{1}{n - 2 l + 1} \sum_{i = l}^{n - l} (\hat{\sigma}_{t_i}^2 - \sigma_{t_i}^2)^2 \right]$.

\subsection{Bandwidth Selection with Plug-In Method}\label{section_simulation_plug_in}

In this section, we investigate the plug-in method that we developed in {Sections \ref{section_bandwidth_selection} and \ref{sec:Plug-inBandwidthSelection}}. In the first part, we will see how the number of iterations, as described in Algorithm \ref{Algorithm_Iterative}, affects the MASE of the kernel estimator. In the second part, we compare {our plug-in method with} the leave-one-out cross-validation method proposed in \cite{kristensen2010nonparametric}. In the third part, we investigate the performance of the TSRVV estimator of volatility of volatility proposed in Section \ref{sec:VolVolEst}.

\subsubsection{Number of Iterations}\label{sec:nIterationSimulation}

{Let us start by investigating} how the number of iterations can affect the accuracy of the plug-in type kernel estimation of spot volatility. An exponential kernel is implemented. Table \ref{Table_nIterations} shows the MASE of the kernel estimator when we use 0 to 5 iterations for the plug-in method, where 0 iteration means we only use the initial value of the bandwidth given by \eqref{Optimal_Bandwidth_Initial_Guess} to estimate the volatility. 

From the results, we first observe that the initial guess has unstable performance, as one may expect. We also find out that after the initial guess, the MASE does not change a lot, and gradually move to a specific fix point value. It is interesting to notice that, after the first iteration, MASE does not always decrease as the number of iteration increases. This indeed makes sense, since our approximated optimal bandwidth and all the estimated parameters have errors, so the fix point that the estimator converges to, after several iterations, might be slightly different from the true optimal value, and it is possible that the initial guess leads to some bandwidth that performs better than the fix point. In conclusion, after one or at most two iterations, there is no significant performance enhancement with more iterations.

\begin {table}
\begin{center}

\center{\textbf{5 Days Data}}
\newline\newline
\begin{tabular}{ | l | l | l | l | l | l | l | l | }
\hline
nData/h	&	$\rho$	&	0	&	1	&	2	&	3	&	4	&	5
 \\
\hline
12	&	0	&	2.5664	&	2.5241	&	2.5482	&	2.5747	&	2.5804	&	2.5845
 \\
\hline
60	&	0	&	1.2180	&	1.0132	&	1.0100	&	1.0138	&	1.0150	&	1.0154
 \\
\hline
12	&	-0.5	&	2.5792	&	2.5177	&	2.5494	&	2.5742	&	2.5810	&	2.5842
 \\
\hline
60	&	-0.5	&	1.2336	&	1.0238	&	1.0206	&	1.0237	&	1.0248	&	1.0248
 \\
\hline
\end{tabular}
\\

\center{\textbf{21 Days Data}}
\newline\newline
\begin{tabular}{ | c | l | l | l | l | l | l | l | }
\hline
nData/h	&	$\rho$	&	0	&	1	&	2	&	3	&	4	&	5
 \\
\hline
12	&	0	&	2.8439	&	2.3712	&	2.3607	&	2.3620	&	2.3626	&	2.3625
 \\
\hline
60	&	0	&	1.3265	&	1.0454	&	1.0385	&	1.0379	&	1.0373	&	1.0375
 \\
\hline
12	&	-0.5	&	2.8923	&	2.4088	&	2.4006	&	2.4051	&	2.4055	&	2.4055
 \\
\hline
60	&	-0.5	&	1.3335	&	1.0459	&	1.0395	&	1.0391	&	1.0388	&	1.0388
 \\
\hline
\end{tabular}
\\

\caption {Comparison of Different Number of Iterations for the Plug-In Method (MASE$\times 10^{-5}$, 10000 sample paths)}
\label{Table_nIterations}
\end{center}
\end {table}

\subsubsection{Comparison Between Plug-In Method and Cross-Validation}

In this part, we compare the plug-in method with the cross-validation method under the three different sampling scenarios described above. In Table \ref{Table_CompareBandwidth}, we report the MASE for the kernel estimator obtained by the different methods. The first column is the plug-in method, where we use the approximated homogeneous optimal bandwidth (\ref{OptimalBandwidthGlobal}) with parameters estimated as proposed in Section \ref{sec:Plug-inBandwidthSelection}. Concretely, we use the formula 
\begin{equation}\label{General_Ito_Process_Bannwidth_Finalb}
h^{a,opt}_{n,j+1} = \left[ \frac{2 T \widehat{IQ(X)} \int K^2(x)dx} { n \widehat{IV\left(\widehat{\sigma}^{2}_{\cdot,j}\right)}^{tsrvv} \iint K(x)K(y) C_1 (x,y) dxdy} \right] ^{1/2},
\end{equation}
to find $\widehat{\sigma}_{\cdot,j+1}$, where $\widehat{\sigma}_{\cdot,j}$ is the estimated volatility at the $j^{th}$-iteration.
In the second column, we report the result for the leave-one-out cross validation. In the third column, we report the result for an oracle plug-in method, where the true $\sigma$ and $\xi$ are used to compute $\int_{0}^{T}\sigma_{\tau}^{4}d\tau$ and $\int_{0}^{T}g^{2}(\tau)d\tau=\xi^{2}\int_{0}^{T}\sigma_{\tau}^{2}d\tau$ in the formula (\ref{OptimalBandwidthGlobal}). Concretely, we use the formula
\begin{equation*}
\begin{split}
\bar{h}^{a, opt}_n = &
\left[ \frac{2 T \int_{0}^{T}\sigma_{\tau}^4d\tau \int K^2(x)dx} { n \xi^{2}\int_{0}^{T}\sigma^{2}_{\tau}d\tau \iint K(x)K(y) C_1 (x,y) dxdy} \right] ^{1/2}.
\end{split}
\end{equation*}
The  final column shows a ``semi-oracle" result, which assumes the knowledge of the volatility of volatility $\xi$ of the Heston model, but not $\sigma$. That is, the formula 
\begin{equation*}
\begin{split}
\bar{h}^{a, opt,semi}_{n+1} = &
\left[ \frac{2 T \int_{0}^{T}\widehat{\sigma}_{\tau,j}^4d\tau \int K^2(x)dx} { n \xi^{2}\int_{0}^{T}\widehat{\sigma}^{2}_{\tau,j}d\tau \iint K(x)K(y) C_1 (x,y) dxdy} \right] ^{1/2},
\end{split}
\end{equation*}
is used to compute the volatility $\widehat{\sigma}_{\cdot,j+1}$ at the $(j+1)^{th}$ iteration.

For this simulation, we only sample 2000 paths, since the cross-validation method is very time consuming. However, we do believe that the result is representative, since for each sample path, the ASE that we calculate already kills a lot of noises.

{\Red
As expected, the plug-in method runs significantly faster than cross validation. As to the accuracy of the kernel estimator, simulation results show that, in almost all sampling frequencies, the plug-in method outperforms the cross-validation method. However, we do observe that {for} 1 month and 1 minute data case, the cross validation is slightly better than the plug-in method. This is {due to the inaccuracy of the estimation of the vol vol for this sampling setting and the lack of optimal tuning of the estimation parameters for the vol vol estimator}. Indeed, when there are fewer data, the plug-in method outperforms cross validation significantly in accuracy. And when there are more data, the computational efficiency becomes a crucial issue. Although both methods tend to have similar performance in accuracy, plug-in method has superior advantage in speed.

It is worth to notice that, in all cases, there is still significant loss of accuracy for the plug-in method compared to the oracle ones. From the two oracle results, it can be easily observed that such a loss of accuracy is mainly due to the estimation error of the volatility of volatility. Further investigation of the estimation of the volatility of volatility is an interesting and important topic for future research.
}

\begin {table}
\begin{center}
\center{\textbf{5 Days Data}}
\newline\newline
\begin{tabular}{ | l | l | l | l | l | l | }
\hline
nData/h & $\rho$ & $MASE_{PI}$ & $MASE_{CV}$ & $MASE_{oracle}$ & $MASE_{semi-oracle}$ \\
\hline
12	&	0	&	1.0796E-07	&	1.3386E-07	&	9.1266E-08	&	9.0402E-08
\\
\hline
60	&	0	&	7.1439E-09	&	8.0542E-09	&	6.7286E-09	&	6.7074E-09
\\
\hline
12	&	-0.5	&	1.0296E-07	&	1.4180E-07	&	9.2620E-08	&	9.2009E-08
\\
\hline
60	&	-0.5	&	7.3872E-09	&	8.2567E-09	&	6.9356E-09	&	6.9060E-09
\\
\hline
\end{tabular}

\center{\textbf{21 Days Data}}
\newline\newline
\begin{tabular}{ | l | l | l | l | l | l | }
\hline
nData/h & $\rho$ & $MASE_{PI}$ & $MASE_{CV}$ & $MASE_{oracle}$ & $MASE_{semi-oracle}$ 
\\
\hline
12	&	0	&	1.9088E-08	&	2.1221E-08	&	1.8265E-08	&	1.8178E-08
\\
\hline
60	&	0	&	1.7064E-09	&	1.6868E-09	&	1.5984E-09	&	1.5961E-09
\\
\hline
12	&	-0.5	&	1.9039E-08	&	1.9495E-08	&	1.7587E-08	&	1.7506E-08
\\
\hline
60	&	-0.5	&	1.6652E-09	&	1.6011E-09	&	1.5509E-09	&	1.5505E-09
\\
\hline
\end{tabular}

\caption {Comparison of Different Bandwidth Selection Methods (MASE, 2000 sample paths)}
\label{Table_CompareBandwidth}
\end{center}
\end {table}

\subsubsection{Estimation of Volatility of Volatility}

In this section, we test the TSRVV estimator that we proposed in Section \ref{sec:VolVolEst}. We use one month data as demonstration, and, in order to see how the estimator performs with different sampling sequence, we consider 5 min and 1 min data. Since we are considering the Heston model, we will not report the integrated volatility of volatility, but instead, we report the following estimator of vol vol parameter $\xi$ of the Heston model: 
\[
	\widehat{\xi}:=\sqrt{\frac{\widehat{IVV}^{tsrvv}}{\widehat{IV}}}.
\]
 Generally, $\xi = 0.5$ is a rule of thumb value, but we will use $\xi = 0.2$ and $0.5$ to test the estimator.

The result is reported in Table \ref{Table_VolVolEst} and as we can see, the estimator performs better when when the sampling frequency increases or the value of $\xi$ get larger. However, it is also clear that estimation error is quite large, so further development of estimation of vol vol should be possible.

\begin {table}
\begin{center}
\begin{tabular}{ | l | l | l | l | l | l | }
\hline
nData/h	&	$\rho$	&	$\xi$	&	Bias	&	Std	&	$\sqrt{MSE}$
\\
\hline
12	&	0	&	0.2	&	-0.0006 	&	0.0990 	&	0.0990 
\\
\hline
12	&	0	&	0.5	&	-0.0584 	&	0.1979 	&	0.2063 
\\
\hline
60	&	0	&	0.2	&	-0.0122 	&	0.0772 	&	0.0782 
\\
\hline
60	&	0	&	0.5	&	-0.0411 	&	0.1549 	&	0.1603 
\\
\hline
12	&	-0.5	&	0.2	&	-0.0002 	&	0.0987 	&	0.0987 
\\
\hline
12	&	-0.5	&	0.5	&	-0.0571 	&	0.1984 	&	0.2065 
\\
\hline
60	&	-0.5	&	0.2	&	-0.0138 	&	0.0779 	&	0.0791 
\\
\hline
60	&	-0.5	&	0.5	&	-0.0443 	&	0.1551 	&	0.1613 
\\
\hline
\end{tabular}

\caption {Estimation of Volatility of Volatility by TSRVV (1 month data, 10000 sample paths)}
\label{Table_VolVolEst}
\end{center}
\end {table}

\subsection{Comparing Different Kernel Functions}
In this section, we compare the performance of different kernel functions. Specifically, we consider the following four different kernels:
\begin{equation*}
\begin{split}
K_1(x) = \frac{1}{2} e^{-|x|} , \quad
K_2(x) = \frac{1}{2} 1_{\{|x|<1\}} , 
\quad
K_3(x) = (1-|x|) 1_{\{|x|<1\}} , \quad
K_4(x) = \frac{3}{4} (1-x^2) 1_{\{|x|<1\}} .
\end{split}
\end{equation*}
The first kernel is the optimal kernel we obtained previously. The other three kernels are finite domain kernels with different order of polynomial. The fourth kernel is the so called Epanechnikov kernel, which is claimed to be the optimal kernel in \cite{kristensen2010nonparametric}.
In the formula for optimal bandwidth, (\ref{OptimalBandwidth}), we encounter some constants that depends on the kernel $K$. As a summary, we calculate them for all the four kernels in Table \ref{Table_Constant_For_Kernels}. The results of the simulation are shown by Table \ref{Table_Comparison_of_Different_Kernels}. Here we consider both the case of $\rho = 0, \rho = -0.5$, and plug-in method with uniform bandwidth. Note that the estimator becomes considerably slow for some kernels, we only simulate 2000 sample paths.

\begin {table}
\begin{center}
\begin{tabular}{ | l | l | l | l | }
\hline
Kernel & $\int_0^\infty K^2(x) dx$ & $L(x) = \int_x^\infty K(s) ds $ &$\int_0^\infty \int_0^\infty K(x) K(y) \min (x, y) dx dy$ \\
\hline
$K_1 = \frac{1}{2} e^{-|x|}$ & $1/8$ & $1/2e^{-x}$ & $1/8$ \\
\hline
$K_2 = \frac{1}{2} 1_{\{|x|<1\}}$ & $1/4$ & $1/2(1-x)$ & $1/12$ \\
\hline
$K_3 = (1-|x|) 1_{\{|x|<1\}}$ & $1/12$ & $1/2(1 - x)^2$ & $1/20$ \\
\hline
$K_4 = \frac{3}{4} (1-x^2) 1_{\{|x|<1\}}$ & $3/10$ & $1/4 (x - 1)^2(x + 2)$ & $33 / 560$ \\
\hline
\end{tabular}
\caption {Some Constants for Different Kernel Functions}
\label{Table_Constant_For_Kernels}
\end{center}
\end {table}

As shown from the result, the exponential kernel performs the best in all cases. As the calculation we had in Example \ref{Example_Different_Kernels}, we can see that the second best kernel is the triangle kernel, since its shape is more similar to exponential kernel. Similarly, the uniform kernel performs the worst, since it is the farthest to the optimal exponential kernel. 

\begin {table}
\begin{center}
\begin{tabular}{ | c | l | l | l | l | l | }
\hline
length	&	$\rho$	&	exponential	&	uniform	&	triangle	&	Epanechnikov
\\
\hline
5 days	&	0	&	2.5974E-05	&	2.8721E-05	&	2.6441E-05	&	2.7085E-05
\\
\hline
5 days	&	-0.5	&	2.5233E-05	&	2.8252E-05	&	2.5759E-05	&	2.6490E-05
\\
\hline
21 days	&	0	&	2.3406E-05	&	2.8047E-05	&	2.4988E-05	&	2.5914E-05
\\
\hline
21 days	&	-0.5	&	2.3692E-05	&	2.8603E-05	&	2.5248E-05	&	2.6173E-05
\\
\hline
\end{tabular}

\caption {Comparison of Different Kernel Functions (5 min data, 2000 sample paths)}
\label{Table_Comparison_of_Different_Kernels}
\end{center}
\end {table}

\appendix
\renewcommand{\appendixname}{Appendix~\Alph{section}}

\section{Equivalence of {the} Approximated Optimal Bandwidth}\label{Sec:Equivalence}

In {Section} \ref{Estimator_and_Framework_Section}, we proposed several assumptions on the volatility processes, which, as shown {in Section \ref{section_bandwidth_selection}}, are enough to construct a well posed optimal kernel estimation {problem}. In this subsection, {we compare the performance of the resulting approximated optimal bandwidth to that of the true optimal bandwidth, whenever it exists.}

{In what follows, $h^*_n$ stands for the ``the true" optimal bandwidth, which is defined to ``minimize" the actual MSE of the kernel estimator, $MSE_n(h) = \mathbb{E}[(\hat{\sigma}^2_{\tau,n,h} - \sigma^2_{\tau} )^2]$. However, since the mapping $h\to{}MSE_n(h)$ is not continuous, it is possible that such a global minimum might not exist or be unique. Hence, in what follows, $h^{*}_{n}$ is an extended nonnegative real number such that $h^{*}_{n}=\lim_{p\to{}\infty} h^{*}_{np}$ for a sequence $\{h^{*}_{np}\}_{p\geq{}1}$  satisfying that $MSE_n(h^{*}_{np})<\inf_{h \in \mathbb{R}_+} MSE_n(h)+\varepsilon_{p}$ and a sequence $\{\varepsilon_{p}\}_{p\geq{}1}$ of positive reals converging to $0$. Let us also recall that $h^{a,opt}_n$ denotes the approximated optimal bandwidth given by (\ref{OptimalBandwidth}).}
Our goal is to find the relationship between $h^*_n$ and $h^{a,opt}_n$, and between {$MSE_n(h^*_n)$ and $MSE_n(h^{a,opt}_n)$}.

Such a problem is in general {hard since} the estimator is not continuous with respect to the bandwidth $h$, when the kernel function $K(\cdot)$ is not continuous in $\mathbb{R}$, which is an important case since kernel functions with finite supports are frequently used in practice (e.g., the uniform kernel function $K_{unif}(x) =1_{[-1, 1]}(x)$). Indeed, when $h \rightarrow (t_{i - 1} - \tau)_-$, the summation appearing in (\ref{Kernel_Estimator}) does not include $K_h(t_{i - 1} - \tau) (\Delta_i X)^2$, while it does include this term when $h \rightarrow (t_{i - 1} - \tau)_+$.
Although it maybe hard to directly analyze the true MSE analytically, {its the first order approximation} is given by (\ref{Approximated_MSE}) and such an approximation is continuous with respect to $h$ for given $n$, which makes the problem still tractable. However, the approximated MSE is expected to be {close} to the true MSE only when $\frac{\Delta}{h}, h \rightarrow 0$, but not in other situations, i.e., $h \nrightarrow 0$ or $\frac{\Delta}{h} \nrightarrow 0$. {As we will show below, the latter situations are, however,} irrelevant when the model under consideration is complex enough. It is worth to remark that typical non-parametric statistical problems consider parameter spaces that are at least as complex as $C^1([0,T])$. However, when the parameter space shrinks to a more trivial case, non-parametric methods may not perform as good as other simpler methods. Hence, in order to rule out some trivial cases, we do need an additional assumption on the complexity of the the model. The following assumption turns out to be enough for our purpose:
\begin{assumption}\label{Complexity_of_Parameters}
Assume that for any $t \in (0,T)$, the mapping $(r,s) \mapsto \mathbb{E}[(\sigma_r^2 - \sigma_t^2)(\sigma_s^2 - \sigma_t^2)], r,s \in [0,T]$ is positive definite, for any fixed $t \in (0,T)$.
\end{assumption}

It is worth mentioning here that Assumption \ref{Complexity_of_Parameters} is not necessary for the kernel estimator to be a consistent estimator or {to possess} the convergence rate given by (\ref{OptimalMSE}) with the choice of approximated optimal bandwidth given by (\ref{OptimalBandwidth}). Such an assumption is solely for the purpose {of ruling} out trivial models so that we can compare the approximated optimal bandwidth with the true optimal bandwidth.

{We also need the following simple lemma}:
\begin{lemma}\label{Lemma_of_NonPredictable_And_Positive_Definiteness}
For the model (\ref{Asset_Dynamic}) satisfying Assumptions \ref{IndependentCondition}, it is not possible to have $t \in (0, T)$, $n \in \mathbb{N}_+$ and $(\alpha_1,...,\alpha_n) \in \mathbb{R}^n$, such that $\sum_{i = 1}^n \alpha_i (\Delta_i X)^2 = \sigma^2_t$ a.s. 
\end{lemma}

\begin{proof}
If we define $\mathscr{G} = \sigma (\sigma_t:t \in [0,T])$, then it is enough to notice that the conditional distribution {$\{ \Delta_i X\}_{1 \leq i \leq n}$, given $\mathscr{G}$} is a {collection of} independent non-trivial Gaussian {variables,} while $\sigma^2_t|\mathscr{G} = \sigma^2_t$ is a non-zero constant.
\end{proof}

We now give a simple example {in which} the Assumption \ref{Complexity_of_Parameters} is not satisfied.
\begin{example}
For a complete filtered probability space $(\Omega, \mathscr{F}, \mathbb{F}=\{\mathscr{F}_t\}_{t\geq 0},\mathbb{P})$, we define an $\mathscr{F}$ measurable random variable $\xi \sim \mbox{unif}(-c,c)$ and assume $\mathscr{F}_t = \sigma(\xi,B_s:s \leq t)$. Now we consider the following model for $t \in [0,T]$:
\begin{equation*}
\begin{split}
dX_t = \sigma_t dB_t , \quad
\sigma^2_t = \sigma^2_0 + \xi \sigma^2_0 \sin (\frac{2 \pi t}{T}),
\end{split}
\end{equation*}
where $\theta = (\sigma_0, c)$ is the parameter in the parameter space $\mathbb{R}_+ \times (0,1)$. Assumption \ref{VolatilityConditionGeneral} can be easily verified. Indeed, we have $\gamma = 2$, $C_\gamma(r,s) = rs$ and
$$
\mathbb{E} [ (\sigma^2_{t+r} - \sigma^2_{t})(\sigma^2_{t+s} - \sigma^2_{t}) ] = \frac{4\pi^2 \sigma_0^4 \mathbb{E}[\xi^2]}{T^2} rs + o(r^2 + s^2).
$$
We {now} consider the estimation of $\sigma^2_{T/2}$. For this model, we actually have
$
\sigma^2_{T/2} = \sigma_0^2 = \int_0^T \sigma_t^2 \cdot \frac{1}{T} dt
$.
We then consider the estimator 
$$
\hat{\sigma}_{T/2}^2 = \frac{1}{T}\sum_{i = 1}^n (\Delta _i X)^2.
$$
The bias of such estimator is zero and the variance is given by
\begin{equation*}
\begin{split}
Var\left( \frac{1}{T}\sum_{i = 1}^n (\Delta _i X)^2 \right)
& =
\frac{1}{T^2} \left( \sum_{i=1}^n \mathbb{E}[(\Delta _i X)^4] + \sum_{i \neq j} \mathbb{E}[(\Delta _i X)^2(\Delta _j X)^2] - \sigma_{1/2}^4 \right) \\&
=
\frac{2}{T^2}
\sum_{i=1}^n \mathbb{E} \left( \int_{t_{i - 1}}^{t_i} \sigma_t^2 dt \right)^2
=
O(n^{-1}).
\end{split}
\end{equation*}
Note that we use the uniform kernel but we do not use a bandwidth that vanishes. The convergence rate here, $O(n^{-1})$, is better than the one stated in Theorem \ref{Approximated_MSE_Theorem} when we consider the kernel estimation with a vanished bandwidth. It is even better than the convergence rate if we use any ``higher order" kernel. Therefore, for this model, a kernel estimator with {vanishing} bandwidth does not have good performance.
\end{example}

With the additional assumption of model complexity, we are now able to show that the only possibility for the MSE of the kernel estimator to converge to zero is that both $\frac{\Delta}{h}$ and $h$ converge to zero.
\begin{proposition}\label{Finite_Sample_Property}
Define $\{ (n_k, h_k) : k \in \mathbb{N} \}$ such that $n_k \in \mathbb{N}$ and $h_k \in \mathbb{R}_+$ and suppose that the model (\ref{Asset_Dynamic}) satisfies Assumptions \ref{IndependentCondition}, \ref{Boundedness_Condition}, \ref{VolatilityConditionGeneral} and \ref{Complexity_of_Parameters}, and that the kernel function $K$ satisfies Assumption \ref{AdmissibleKernel}. Then, $\lim_{k \rightarrow \infty} {\mbox{MSE}(n_k,h_{n_k})} = 0$ if and only if $\lim_{k \rightarrow \infty} \frac{T}{n_k h_{n_k}} = 0$ and $\lim_{k \rightarrow \infty} h_{n_k} = 0$.
\end{proposition}

We {defer} the proof of Proposition \ref{Finite_Sample_Property} {to} Appendix \ref{More_Technical_Details}. As we can see from Proposition \ref{Finite_Sample_Property}, the kernel estimator only converges when the sample size $n \rightarrow \infty$. The following lemma enables us to consider the relationship between $h^*_n$ and $h^{a,opt}_n$, whose proof is again {given} in Appendix \ref{More_Technical_Details}.
\begin{lemma}\label{Approximation_Bandwidth_Lemma}
Assume $F:\mathbb{R}_+ \times \mathbb{R}_+ \rightarrow \mathbb{R}_+$ and $f(x,y) = Ax + By^\gamma$ for $A, B, \gamma > 0$, such that $F(x,y) - f(x,y) = o(x) + o(y^\gamma)$ as $x,y\rightarrow (0^+, 0^+)$. Also assume that for all $\delta > 0$, there exists $m > 0$, such that for all $x, y> \delta$, we have $F(x,y) > m$. Suppose $z_n \searrow 0$ and, for each $n\in \mathbb{N}_+$, $y_n$ and $y^*_n$ minimize $y \mapsto f(z_n/y,y)$ and $y \mapsto F(z_n/y,y)$, respectively. Then we have
$$
\inf_{y \in \mathbb{R}_+} F(z_n / y, y) \rightarrow 0, 
\quad 
y_n = y^*_n + o(y^*_n), 
\quad 
F(z_n/y_n, y_n) = \inf_{y \in \mathbb{R}_+}F(z_n / y, y) + o(\inf_{y \in \mathbb{R}_+} F(z_n / y, y)), 
$$
as $n \rightarrow \infty$. Note that $F$ might not be continuous, so we say that $y^*_n$ minimize $y \mapsto F(z_n/y,y)$ in the sense that 
there exists $\{y^*_{np}: p \in \mathbb{N}_+\}$, such that $\lim_{p \rightarrow \infty} y^*_{np} = y_n^*$ and $\lim_{p \rightarrow \infty} F(z_n / y^*_{np}, y^*_{np}) = \inf_{y \in \mathbb{R}_+} F(z_n / y, y)$. Note that by assumptions on $F$ and $f$, $y_n^*$ is finite for $n$ large enough.
\end{lemma}

\begin{remark}
The result $F(z_n/y_n, y_n) = F(z_n / y^*_n, y^*_n) + o(F(z_n / y^*_n, y^*_n))$ is quite important for our purpose. When connected to the kernel estimator, it means that the departure of approximated bandwidth from the true optimal bandwidth will not significantly affect the true MSE of the kernel estimator.
\end{remark}

With these in hand, we are ready for the result of the relationship between the approximated optimal bandwidth and the true optimal bandwidth.
\begin{theorem}\label{Optimal_Bandwidth_vs_True_Optimal}
For model (\ref{Asset_Dynamic}) with $\mu$ and $\sigma$ satisfying Assumptions \ref{IndependentCondition}, \ref{Boundedness_Condition}, \ref{VolatilityConditionGeneral} and \ref{Complexity_of_Parameters} and a kernel function $K(x)$ satisfying assumption \ref{AdmissibleKernel}, we have
\begin{equation}
\begin{split}
h^{a,opt}_n & = h^*_n + o(h^*_n),\\
{MSE_n}(h^{a,opt}_n) & = \inf_h {MSE_n(h)} + o(\inf_h {MSE_n(h)}),
\end{split}
\end{equation}
where the superscript ``$*$" denotes the true optimal bandwidth and MSE, while ``$a$" denotes the approximated ones.
\end{theorem}

\begin{proof}
Now we write $MSE^*(n, h) = F(\frac{\Delta}{h},h)$ and $MSE^{a}(n, h) = f(\frac{\Delta}{h},h^\gamma) = A\frac{\Delta}{h} + Bh^\gamma$ where the value of $A$ and $B$ can be found in (\ref{Approximated_MSE}). From Theorem \ref{Approximated_MSE_Theorem} and Proposition \ref{Finite_Sample_Property}, we know that $F(x,y)$ and $f(x,y)$ satisfy {the} requirements by Lemma \ref{Approximation_Bandwidth_Lemma}, where $z_n = \Delta = \frac{T}{n}$. Then, it is immediate to obtain the desired result.
\end{proof}

\begin{remark}
The theorem above also tells us a fact, that under our model setting, the kernel estimator generally perform better when we observe more data, i.e. the frequency of observation is higher. This seems to be an obvious fact, but is not always true. In the case of using realized variance to estimate the Integrated Volatility with market micro-structure noise, as proved in \cite{Two_Time_Scale}, there is an optimal frequency of the data. In such case, increasing the frequency does not yield better performance in general.
\end{remark}

\section{{Proofs}}\label{Sect:TecProof}
\subsubsection*{Proof of {Proposition} \ref{BMProcesses}}

We consider $h > 0$ in what follows, while the case of $h < 0$ is similar. Using the boundedness of $\mathbb{E}[f^2(t,\omega)]$ and continuity of $\mathbb{E}[g^2(t,\omega)]$, we have
\begin{align*}
\mathbb{E}[(V_{t+h}-V_t)^2]
  &=
\mathbb{E} \left[ \int_{t}^{t+h} f(s,\omega)ds \right] ^2
+ 2 \mathbb{E} \left[ \int_{t}^{t+h} f(s,\omega)ds \int_{t}^{t+h} g(s,\omega) dB_s \right] 
+ \mathbb{E} \left[ \int_{t}^{t+h} g^2(s,\omega) ds \right] \\
& = h \mathbb{E} [ g^2(t,\omega) ] + o(h), \quad h\rightarrow 0 ,
\end{align*}
where in the last equality we used {that}
\begin{align*}
\mathbb{E} \left[ \int_{t}^{t+h} f(s,\omega)ds \right] ^2
& \leq 
h \mathbb{E} \left[ \int_{t}^{t+h} f^2(s,\omega)ds \right] 
= 
h \int_{t}^{t+h} \mathbb{E} [ f^2(s,\omega) ] ds 
=
O(h^2), \\
\mathbb{E} \left[ \int_{t}^{t+h} f(s,\omega)ds \int_{t}^{t+h} g(s,\omega) dB_s \right]
& \leq
	\sqrt{ 
	\mathbb{E} \left[ \int_{t}^{t+h} f(s,\omega)ds \right] ^2
	\mathbb{E} \left[ \int_{t}^{t+h} g(s,\omega) dB_s \right] ^2
	}
=
O(h^{3/2}), \\
{h^{-1} \mathbb{E} \left[ \int_{t}^{t+h} g^2(s,\omega) ds \right] - \mathbb{E}\left[g^{2}(t,\omega)\right]}& = o(1),
\end{align*}
for $h \rightarrow 0_+$.
Now, for {$r> 0$ and $t>s>0$}, we have
\begin{align*}
& | \mathbb{E}[(V_{t+r}-V_t)(V_t-V_{t-s})] | \\
& = 
\left| \mathbb{E}
	\left[ \int_{t}^{t+r} f(s,\omega)ds + \int_{t}^{t+r} g(s,\omega) dB_s \right] 
	\left[ \int_{t-s}^{t} f(s,\omega)ds + \int_{t-s}^{t} g(s,\omega) dB_s \right]
 \right| \\
& \leq
\left| \mathbb{E} 
	\int_{t}^{t+r} f(s,\omega)ds 
	\int_{t-s}^{t} f(s,\omega)ds
\right|
+
 \left| \mathbb{E}
	\int_{t}^{t+r} g(s,\omega)dB_s
	\int_{t-s}^{t} f(s,\omega)ds
\right| \\&\quad 
+
 \left| \mathbb{E}
	\int_{t}^{t+r} f(s,\omega)ds
	\int_{t-s}^{t} g(s,\omega)dB_s
\right|
+
 \left| \mathbb{E}
	\int_{t}^{t+r} g(s,\omega)dB_s
	\int_{t-s}^{t} g(s,\omega)dB_s
\right|
\\
& \leq
\sqrt{ 
	\mathbb{E} \left[ \int_{t}^{t+r} f(s,\omega)ds \right]^2
	\mathbb{E} \left[ \int_{t-s}^{t} f(s,\omega)ds \right]^2
 }
+
\sqrt{
	\mathbb{E} \left[ \int_{t}^{t+r} f(s,\omega)ds \right]^2
	\mathbb{E} \left[ \int_{t-s}^{t} g(s,\omega)dB_s \right]^2
}
\\
& \leq
A_1 rs + A_2 r\sqrt{s}
\leq 
A r\sqrt{s} ,
\end{align*}
for some constant $A_1, A_2$ and $A$. Note that $A$ can be made uniform over $t \in (0,T)$ due to boundedness of $\mathbb{E}[f^2(t, \omega)]$ and continuity of $\mathbb{E}[g^2(t,\omega)]$. {Finally,} for $r>s>0$, we have
\begin{equation*}
\begin{split}
& \mathbb{E}[(V_{t+r}-V_t)(V_{t + s}-V_{t})]
=
\mathbb{E}[(V_{t + s}-V_{t})^2]
+
\mathbb{E}[(V_{t+r}-V_{t + s})(V_{t + s}-V_{t})] \\
& =
s \mathbb{E} [ g^2(t,\omega) ] + o(s)
+
O((r - s)\sqrt{s})
=
s \mathbb{E} [ g^2(t,\omega) ] + o(s)
+
O(r\sqrt{s})
=
s \mathbb{E} [ g^2(t,\omega) ] + o((r^2 + s^2)^{1/2}) .
\end{split}
\end{equation*}
Similar {arguments can be applied} for $r < s < 0$, {while the} case of $r < 0 < s$ can be proved by noticing that $r\sqrt{s} = o((r^2 + s^2)^{1/2})$. Therefore, in summary, we have proved {that 
Assumption \ref{VolatilityConditionGeneral} hold true} with $\gamma = 1$ and $C_1(r,s) = \min \{|r|, |s|\} 1_{\{rs \geq 0\}}$ and $L(t) = \mathbb{E} [ g^2(t,\omega) ]$. 

\noindent {It remains to prove that $C_1$ is positive definite. To that end, note that}
\begin{equation*}
\begin{split}
&\quad \iint K(r) K(s) \min \{|r|, |s|\} 1_{\{rs \geq 0\}} dr ds\\
& =
\int_{0}^{\infty} \int_{0}^{\infty} K(r) K(s) \min \{r, s\} drds
+
\int_{0}^{\infty} \int_{0}^{\infty} K(-r) K(-s) \min \{r, s\} drds \\
& =
\int_{0}^{\infty} \int_{0}^{\infty} \left[ K(r) K(s) + K(-r) K(-s) \right] \int_{0}^{\infty} 1_{\{t \leq r\}}1_{\{t \leq s\}}dt drds \\
& =
\int_0^{\infty} \left[ \int_0^{\infty} K(r) 1_{\{t \leq r\}} {dr}\right]^2 dt + \int_0^{\infty} \left[\int_0^{\infty} K(-r) 1_{\{t \leq r\}} {dr}\right]^2 dt,
\end{split}
\end{equation*}
{which is positive as long as $\int |K(x)| dx > 0$}.

\subsubsection*{Proof of {Proposition} \ref{stationaryfBM}}

{For easiness of notation we write $B^{H}$ and $Y^{H}$ instead of $B^{(H)}$ and $Y^{(H)}$.} Pipiras and Taqqu (2000) gave the following result:
\begin{equation}\label{KRel1}
\begin{split}
\mathbb{E}\left[ \int_{-\infty}^{\infty} g_1(u) dB_u^H \int_{-\infty}^{\infty} g_2(u) dB_u^H \right]
=
H(2H-1)\int_{-\infty}^{\infty} \int_{-\infty}^{\infty} g_1(u)g_2(v)|u-v|^{2H-2} dudv ,
\end{split}
\end{equation}
where $g_1$ and $g_2$ are assumed to be real valued function {satisfying the integrability condition (\ref{NCFEI}).}
We first use this results with $g_1(u) = 1_{[t,t+r]}(u),g_2(u) = 1_{[t,t+s]}(u)$, where $r,s \in \mathbb{R}$. We consider $r = s > 0$ first and we have
\begin{equation*}
\begin{split}
{\mathbb{E}\left[ \left( \int_{t}^{t+r} dB_u^H \right)^2 \right]
=
H(2H-1)\int_{t}^{t+r} \int_{t}^{t+r} |u-v|^{2H-2} dudv  =
2H(2H-1)\int_{t}^{t+r} \left( \int_{t}^{u} (u-v)^{2H-2} dv \right) du 
=
r^{2H}}.
\end{split}
\end{equation*}
For the case of $r>0>s$, we have
\begin{equation*}
\begin{split}
&\quad \mathbb{E}\left[ 
\left( \int_{t}^{t+s} dB_u^H \right) 
\left( \int_{t}^{t+r} dB_u^H \right)
\right]
=
- H(2H-1)\int_{t+s}^{t} \int_{t}^{t+r} |u-v|^{2H-2} dudv \\
& =
- H\int_{t+s}^{t} |t+r-v|^{2H-1} - |t-v|^{2H-1} dv
= 
\frac{1}{2} ( |r|^{2H} + |s|^{2H} - |r-s|^{2H}) .
\end{split}
\end{equation*}	
These two results can be combined to be
\begin{equation*} 
\begin{split}
\mathbb{E}\left[ 
\left( \int_{t}^{t+s} dB_u^H \right) 
\left( \int_{t}^{t+r} dB_u^H \right)
\right]
= 
\frac{1}{2} ( |r|^{2H} + |s|^{2H} - |r-s|^{2H})
=: C_{2H}(r,s;t),
\end{split}
\end{equation*}
for all {$r,s \in \mathbb{R}$. 
Next, we first assume that $f(t)>0$ and prove} the case of $r,s>0$. Other cases can be proved similarly. {Since $f$ is assumed to be continues, for any $\epsilon \in (0,f(t))$, let $\delta=\delta_{\varepsilon,t} > 0$, such that $\forall h \in (0,\delta)$, $|f(t+h) - f(t)|<\epsilon$}. Then, we have the following upper bound, {for any $0<r,s<\delta$}:
\begin{equation*}
\begin{split}
& \mathbb{E}\left[ \left( \int_{t}^{t+r} f(u) dB_u^H \right) \left( \int_{t}^{t+s} f(u) dB_u^H \right) \right]
=
\mathbb{E} \left[ \left( \int_{-\infty}^{\infty} f(u)1_{[t,t+r]} dB_u^H \right) \left( \int_{-\infty}^{\infty} f(u)1_{[t,t+s]} dB_u^H \right) \right] \\
 &\quad=
H(2H-1) \int_{t}^{t+r} \int_{t}^{t+s} f(u)f(v)|u-v|^{2H-2} dudv 
\leq
(f(t)+\epsilon)^2 H(2H-1)  \int_{t}^{t+r}\int_{t}^{t+s} |u-v|^{2H-2} dudv \\
&\quad {={}}
(f(t)+\epsilon)^2 C(r,s;t).
\end{split}
\end{equation*}
{A similar lower bound holds for {$0<r,s<\delta$} as the follows:}
\begin{equation*}
\begin{split}
\mathbb{E}\left[ \left( \int_{t}^{t+r} f(u) dB_u^H \right) \left( \int_{t}^{t+s} f(u) dB_u^H \right) \right]
\geq 
(f(t)-\epsilon)^2 C(r,s;t).
\end{split}
\end{equation*}
These two equations lead to the following {result:}
\begin{equation}\label{ELFBM}
\begin{split}
\lim_{r,s\rightarrow 0_+} C^{-1}(r,s;t) \mathbb{E}\left[ \left( \int_{t}^{t+r} f(u) dB_u^H \right) \left( \int_{t}^{t+s} f(u) dB_u^H \right) \right]
=
f^2(t) .
\end{split}
\end{equation}
{The case of $r,s\rightarrow 0_-$ and $f(t)\leq 0$ can be deduced similarly.}
This proves that the Assmption \ref{VolatilityConditionGeneral} is satisfied with $\gamma = 2H$ and $C_\gamma$ given by (\ref{fBM_Vol_Cov_Structure}).
The case of $\exp(Y_t^H)$ follows from Proposition \ref{proposition:C_gamma_of_function_of_gaussian_process}.

\subsubsection*{Proof of {Proposition} \ref{stationaryfOU}}

We {first consider $Y^H$ first and let $C_{2H}(r,s) = \frac{1}{2} (|r|^{2H} + |s|^{2H} - |r-s|^{2H})$}. The proof is quite similar to {Proposition} \ref{stationaryfBM}. Indeed, we have
\begin{equation*}
\begin{split}
& \mathbb{E}[(Y_{t+r}^H - Y_t^H)(Y_{t+s}^H - Y_t^H)] \\
= &
\sigma^2 \mathbb{E} \left[ \left(
\int_{-\infty}^{t+r} e^{-\lambda(t+r-u)}dB_u^H - \int_{-\infty}^t e^{-\lambda(t-u)}dB_u^H
\right) \left(
\int_{-\infty}^{t+s} e^{-\lambda(t+s-u)}dB_u^H - \int_{-\infty}^t e^{-\lambda(t-u)}dB_u^H
\right) \right] \\
= &
\sigma^2 \mathbb{E} \left[ 
\left(
(e^{-\lambda(t+r)} - e^{-\lambda t})\int_{-\infty}^{t+r} e^{\lambda u}dB_u^H
+
e^{-\lambda t}\int_{t}^{t+r} e^{\lambda u}dB_u^H
\right)
\left(
(e^{-\lambda(t+s)} - e^{-\lambda t})\int_{-\infty}^{t+s} e^{\lambda u}dB_u^H
+
e^{-\lambda t}\int_{t}^{t+{s}} e^{\lambda u}dB_u^H
\right) 
\right] \\
= & 
\sigma^2 C_{2H}(r,s) + o((r^2 + s^2)^{H}) ,
\end{split}
\end{equation*}
{where the last equality is a consequence of the following relationships, which in turn use (\ref{KRel1}):}
\begin{equation*}
\begin{split}
& \mathbb{E}\left[
(e^{-\lambda(t+r)} - e^{-\lambda t}) 
\left( \int_{-\infty}^{t+r} e^{\lambda u}dB_u^H \right)
(e^{-\lambda(t+s)} - e^{-\lambda t}) 
\left( \int_{-\infty}^{t+s} e^{\lambda u}dB_u^H \right)
\right] \\
&\quad  =
(e^{-\lambda r} - 1)(e^{-\lambda s} - 1)
\mathbb{E}\left[
e^{-\lambda t}
\left( \int_{-\infty}^{t+r} e^{\lambda u}dB_u^H \right)
e^{-\lambda t} 
\left( \int_{-\infty}^{t+s} e^{\lambda u}dB_u^H \right)
\right]
=  
O(rs) ,\\
& \mathbb{E} \left[
(e^{-\lambda r} - 1)e^{-\lambda t} \left( \int_{-\infty}^{t+r} e^{\lambda u}dB_u^H \right)
e^{-\lambda t} \left( \int_{t}^{t+s} e^{\lambda u}dB_u^H \right) \right] \\
&\quad  \leq
|e^{-\lambda r} - 1| \sqrt{\mathbb{E} \left[ \left(
e^{-\lambda t} \int_{-\infty}^{t+r} e^{\lambda u}dB_u^H
\right) ^2 \right]
\mathbb{E} \left[
\left(
e^{-\lambda t} \int_{t}^{t+s} e^{\lambda u}dB_u^H
\right)^2 \right]
}
=
O(s^Hr) ,\\
& \mathbb{E} \left[
e^{-\lambda t} \left( \int_{t}^{t+r} e^{\lambda u}dB_u^H \right) 
e^{-\lambda t} \left( \int_{t}^{t+s} e^{\lambda u}dB_u^H \right) \right]
=
C_{2H} (r,s) + o((r^2 + s^2)^{H}) .
\end{split}
\end{equation*}
{The last equality follows along the lines of the proof of (\ref{ELFBM}). This completes the proof of the first assertion. Once we notice $Y^H$ is also a Gaussian process, the proof of $\exp(Y^H)$} is similar as previous lemma.

\subsubsection*{{Proof of Proposition 
\ref{proposition:C_gamma_of_function_of_gaussian_process}}}
To begin with, since we assume that Assumption \ref{VolatilityConditionGeneral} is satisfied uniformly over $(0, T)$ and $\sup_{t\in (0, T)} |L(t)| < \infty$, we can use Kolmogorov-\v{C}entsov  continuity theorem to conclude that there is a continuous modification of $Z$ and, {thus, hereafter, we assume} that $\{Z_t\}_{t \in [0, T]}$ is a continuous process\footnote{ Indeed, for any $0<s<t<T$, we have $\mathbb{E}[(Z_{t} - Z_{s})^{2k}]= 
(2k - 1)!! (\mathbb{E} [(Z_{t} - Z_{s})^{2}] )^k\leq C |t - s|^{k \gamma}$, for some constant $C$, independent of $s$ and $t$. Then, we can conclude that there exists a modification of $Z$ that is H\"older continuous of order $(k\gamma - 1) / 2k$ and, thus, of any order less than $\gamma / 2$.}.
{Next}, by Taylor's expansion, there exists $\theta(\tau, r) \in ( \min(\tau, \tau + r), \max(\tau, \tau + r) )$ such that, a.s.,
$$
f(Z_{\tau + r}) - f(Z_{\tau}) = f^\prime(Z_\tau) (Z_{\tau + r} - Z_{\tau}) + f^{\prime\prime}(Z_{\theta(\tau, r)}) (Z_{\tau + r} - Z_{\tau})^2.
$$
Thus, we have the following decomposition
\begin{equation}\label{NEHP}
\begin{split}
(f(Z_{\tau + r}) - f(Z_{\tau}))(f(Z_{\tau + s}) - f(Z_{\tau}))
& =
(f^\prime(Z_\tau))^2  (Z_{\tau + r} - Z_{\tau}) (Z_{\tau + s} - Z_{\tau}) \\
&\quad +
f^\prime(Z_\tau) f^{\prime\prime} (Z_{\theta(\tau, s)}) (Z_{\tau + r} - Z_{\tau}) (Z_{\tau + s} - Z_{\tau})^2 \\
&\quad +
f^\prime(Z_\tau) f^{\prime\prime} (Z_{\theta(\tau, r)}) (Z_{\tau + s} - Z_{\tau}) (Z_{\tau + r} - Z_{\tau})^2 \\
&\quad +
f^{\prime\prime}(Z_{\theta(\tau, r)}) f^{\prime\prime} (Z_{\theta(\tau, s)}) (Z_{\tau + r} - Z_{\tau})^2 (Z_{\tau + s} - Z_{\tau})^2.
\end{split}
\end{equation}
Except for the first term, all other terms are of higher order. As an example, take the second term {and note that}
\begin{equation*}
\begin{split}
&\quad
\mathbb{E}|f^\prime(Z_\tau) f^{\prime\prime} ({Z_{\theta(\tau,r)}}) (Z_{\tau + r} - Z_{\tau}) (Z_{\tau + s} - Z_{\tau})^2| \\
& \quad\leq 
\left( 
\mathbb{E} [(f^\prime(Z_\tau))^4] 
\mathbb{E} [(f^{\prime\prime} ( Z_{\theta(\tau,s)}) )^4] 
\mathbb{E} [(Z_{\tau + r} - Z_{\tau})^4] 
\mathbb{E} [(Z_{\tau + s} - Z_{\tau})^8] 
\right)^{1/4} 
=
O( (r^2 + s^2)^{{3\gamma /4}}),
\end{split}
\end{equation*}
where the last equality uses (a) and the normality of $Z$.
Indeed, if we define $m_t = \mathbb{E}[Z_t]$ and $z_t = Z_t - m_t$, we have 
$$
\mathbb{E} [(Z_{\tau + r} - Z_{\tau})^4] 
=
\mathbb{E} [(z_{\tau + r} - z_{\tau})^4] 
+ 
6 \mathbb{E} [(z_{\tau + r} - z_{\tau})^2](m_{\tau + r} - m_{\tau})^2 
+
(m_{\tau + r} - m_{\tau})^4
= \mathbb{E} [(z_{\tau + r} - z_{\tau})^4] + o((r^2 + s^2)^{\gamma}).
$$
We proceed to consider the first term {of (\ref{NEHP})}. With similar argument as the above, we can assume, without loss of generality, that $Z$ has zero mean. Next, since $(Z_\tau, Z_{\tau + r}, Z_{\tau + s})$ are jointly Gaussian, we can define two independent standard normal variables $X(\tau, r, s)$ and $Y(\tau, r, s)$ that are also independent of $Z_{\tau}$ such that
$$
Z_{\tau + r} - Z_{\tau} = a_1 Z_\tau + a_2 X(\tau, r, s) + a_3 Y(\tau, r, s), \quad
Z_{\tau + s} - Z_{\tau} = b_1 Z_\tau + b_2 X(\tau, r, s) + b_3 Y(\tau, r, s),
$$
for some constants $a_{i}$ and $b_{i}$, $i=1,2,3$, depending on $\tau$, $r$, and $s$. Furthermore, $a_{1}$ and $b_{1}$ are such that
$$
a_1 = \frac{\mathbb{E}[(Z_{\tau + r} - Z_{\tau})Z_\tau]}{\mathbb{E}[Z_\tau^2]} = O(|r|), \quad
b_1 = \frac{\mathbb{E}[(Z_{\tau + s} - Z_{\tau})Z_\tau]}{\mathbb{E}[Z_\tau^2]} = O(|s|).
$$
Now, since $Z$ satisfies Assumption \ref{VolatilityConditionGeneral} and $O(|rs|) = o((r^2 + s^2)^{\gamma / 2})$, {we have}
$$
a_2b_2 \mathbb{E}[X(\tau, r, s)^2] + a_3b_3 \mathbb{E}[Y(\tau, r, s)^2]
=
\mathbb{E}[(Z_{\tau + r} - Z_{\tau})(Z_{\tau + {s}} - Z_{\tau})]
-
a_1b_1 \mathbb{E}[Z_\tau^2] 
=
C_\gamma (r, s) + o((r^2 + s^2)^{\gamma / 2}).
$$
{Finally,}
\begin{equation*}
\begin{split}
\mathbb{E}[
(f^\prime(Z_\tau))^2  (Z_{\tau + r} - Z_{\tau}) (Z_{\tau + s} - Z_{\tau})] 
& =
a_1b_1 \mathbb{E}[ (f^\prime(Z_\tau))^2 Z_{\tau}^2 ]
+
\mathbb{E}[ (f^\prime(Z_\tau))^2 ](
a_2b_2 \mathbb{E}[X(\tau, r, s)^2] 
+ 
a_3b_3 \mathbb{E}[Y(\tau, r, s)^2]
) \\
& =
\mathbb{E}[ (f^\prime(Z_\tau))^2 ] C_\gamma(r, s) + o((r^2 + s^2)^{\gamma / 2}),
\end{split}
\end{equation*}
and we conclude.

\subsubsection*{Proof of Lemma \ref{D1Lemma}}

We consider the case that $m = 1$, $K \in C((A,B))$, {and} $K$ is piecewise $C^1$ in $(A, B)$, where $A < 0 < B$. In the whole proof, all the summations are taken {under the additional constrain that} $(\frac{t_{i-1} - \tau}{h}, \frac{t_{i} - \tau}{h}) \in (A,B)$. {Note that this constraint introduces an additional term of order $o(\frac{\Delta}{h})$. (Indeed, this some times exclude a term at the right boundary.) We first assume that $K \in C^1((A, B))$, even though the same arguments apply for piecewise $C^1$ functions. First, note that}
\begin{equation*}
\begin{split}
D_1 & = \sum _{i=1}^n \left[ K_h(t_{i-1} - \tau) \int _{t_{i-1}}^{t_i} f(t) dt - \int _{t_{i-1}}^{t_i} K_h(t - \tau) f(t) dt \right] \\ 
	& = \frac{1}{h} \sum _{i=1}^n \int _{t_{i-1}}^{t_i} [ K( \frac{t_{i-1} - \tau}{h} ) - K( \frac{t - \tau}{h} ) ] f(t) dt
	= \frac{1}{h} \sum _{i=1}^n \int _{t_{i-1}}^{t_i} K^{\prime} ( \frac{s_t - \tau}{h} ) \frac{t_{i-1} - t}{h} f(t) dt \\
	& = \frac{1}{h} \sum _{i=1}^n \int _{t_{i-1}}^{t_i} K^{\prime} ( \frac{t_{i-1} - \tau}{h} ) \frac{t_{i-1} - t}{h} f(t) dt 
	 +
	\frac{1}{h} \sum _{i=1}^n \int _{t_{i-1}}^{t_i} [ K^{\prime} ( \frac{s_t - \tau}{h} ) - K^{\prime} ( \frac{t_{i-1} - \tau}{h} ) ] \frac{t_{i-1} - t}{h} f(t) dt \\
	& =: D_{11} + D_{12},
\end{split}
\end{equation*}
{for some $s_t \in (t_{i-1}, t)$}. In above, $D_{12}$ can be controlled as the following:
\begin{equation*}
\begin{split}
|D_{12}| \leq 
\frac{M_f \Delta^2}{2h^2} \sum _{i=1}^n \max_{t \in [t_{i - 1}, t_i]}| K^{\prime} ( \frac{t_{i-1} - \tau}{h} ) - K^{\prime} ( \frac{{t} - \tau}{h} ) |
\leq 
\frac{M_f \Delta^2}{2h^2} V_{{-\infty}}^{\infty}(K^{\prime})
= O(\frac{\Delta^2}{h^2}).
\end{split}
\end{equation*}
Note that if $K^\prime$ is not continuous at $t_{i}$, we use right derivative. For interval $(t_{i - 1}, t_i)$ that $K^\prime$ is not continuous, there still exists $K^*(t)$ such that $K( \frac{t_{i-1} - \tau}{h} ) - K( \frac{t - \tau}{h} ) = K^*(t) \frac{t_{i-1} - t}{h}$, and by Lipschitz's condition, the above control of $D_{12}$ is still valid, except that $ V_{0}^{\infty}(K^{\prime})$ is replaced by $ V_{0}^{\infty}(K^{\prime}) + 2pL$, where $p$ and $L$ are the number of non-continuous points and Lipschitz constant, respectively.

We then consider $D_{11}$. Indeed, for any $\delta \in (0,\min (T-\tau,\tau))$, we have
\begin{equation*}
\begin{split}
D_{11} = \frac{1}{h} \left( \sum _{ |t_i - \tau| < \delta } + \sum _{ \delta \leq |t_i - \tau| \leq T } \right) K^{\prime} ( \frac{t_{i-1} - \tau}{h} ) \int _{t_{i-1}}^{t_i} \frac{t_{i-1} - t}{h} f(t) dt 
	\triangleq D_{111} + D_{112} .
\end{split}
\end{equation*}
By assumptions, for $A > \sup_{\tau \in (0, T)} |l(\tau)|$, there exists $\delta _0 \in (0,\min (T-\tau,\tau))$, such that for all $t$ with $|t - \tau| < \delta_0$, we have
$$\frac{|f(t)-f(\tau)|}{|t-\tau|^\gamma} < A .$$
Now define $\epsilon (\delta) = A \delta ^\gamma$. Then, for all $\delta \in (0, \delta_0)$ and $|t - \tau| < \delta$, we have $|f(t) - f(\tau)| < \epsilon (\delta)$.

The term $D_{112}$ above can be controlled as the following:
\begin{equation*}
\begin{split}
| D_{112} |
	\leq 
	\frac{1}{h} \sum _{ \delta \leq |t_i - \tau| \leq T } | K^{\prime} ( \frac{t_{i-1} - \tau}{h} ) | \int _{t_{i-1}}^{t_i} \frac{ | t_{i-1} - t | }{h} f(t) dt 
	\leq 
	\frac{M_{f} \Delta }{2 h} \sum _{ |t_i - \tau| \geq \delta } \frac{\Delta}{h} | K^{\prime} ( \frac{t_{i-1} - \tau}{h} ) | .
\end{split}
\end{equation*}
The term $D_{111}$ can by controlled by the following inequality, when $\Delta < \delta_0$:
\begin{equation*}
\begin{split}
D_{111} 
	& = \frac{1}{h} \left( \sum _{ |t_i - \tau| < \delta , K^{\prime} \leq 0 } + \sum _{ |t_i - \tau| < \delta , K^{\prime} > 0 } \right) K^{\prime} ( \frac{t_{i-1}  - \tau}{h} ) \int _{t_{i-1}}^{t_i} \frac{t_{i-1} - t}{h} f(t) dt \\
	& \leq \frac{1}{h} \sum _{ |t_i - \tau| < \delta , K^{\prime} \leq 0 } K^{\prime} ( \frac{t_{i-1} - \tau}{h} ) \int _{t_{i-1}}^{t_i} \frac{t_{i-1} - t}{h} ( f(\tau) + \epsilon (\delta) ) dt \\
	&\quad +
	\frac{1}{h} \sum _{ |t_i - \tau| < \delta , K^{\prime} > 0 } K^{\prime} ( \frac{t_{i-1} - \tau}{h} ) \int _{t_{i-1}}^{t_i} \frac{t_{i-1} - t}{h} ( f(\tau) - \epsilon (\delta) ) dt \\
	& =
	- \frac{ f(\tau) \Delta }{ 2h } \sum _{ |t_i - \tau| < \delta } \frac{\Delta}{h} K^{\prime} ( \frac{t_{i-1} - \tau}{h} ) 
	+ \frac{  \epsilon (\delta) \Delta }{ 2h } \sum _{ |t_i - \tau| < \delta } \frac{\Delta}{h} | K^{\prime} ( \frac{t_{i-1} - \tau}{h} ) | 
	\triangleq \overline{D}_{111}.
\end{split}
\end{equation*}
Similarly, the lower bound can be written as
$$
D_{111} \geq
	- \frac{ f(\tau) \Delta }{ 2h } \sum _{ |t_i - \tau| < \delta } \frac{\Delta}{h} K^{\prime} ( \frac{t_{i-1} - \tau}{h} ) 
	- \frac{  \epsilon (\delta) \Delta }{ 2h } \sum _{ |t_i - \tau| < \delta } \frac{\Delta}{h} | K^{\prime} ( \frac{t_{i-1} - \tau}{h} ) | 
 	\triangleq \underline{D}_{111}.
$$
Now we can set $\delta = \sqrt{h}$ and we assume that $\delta < \delta_0$. In the following, all limits are taken when $h \rightarrow 0$, $\frac{\Delta}{h} \rightarrow 0$ and $\frac{\delta}{h}\rightarrow \infty$.

Firstly we consider $\sum _{ |t_i - \tau| < \delta } \frac{\Delta}{h} K^{\prime} ( \frac{t_{i-1} - \tau}{h} )$. Indeed, there exists $s_{i-1} \in ( t_{i-1},t_i )$, such that $\int _{(t_{i-1} - \tau)/h}^{(t_i - \tau)/h}K^{\prime} (x)dx = \frac{\Delta}{h} K^{\prime} ( \frac{s_{i-1} - \tau}{h} )$. Then, we have
\begin{equation*}
\begin{split}
\sum _{ |t_i - \tau| < \delta } \frac{\Delta}{h} K^{\prime} ( \frac{t_{i-1} - \tau}{h} )
	& = \sum _{ |t_i - \tau| < \delta } \int _{(t_{i-1} - \tau)/h}^{(t_i - \tau)/h}K^{\prime} (x)dx 
	+ \frac{\Delta}{h} \sum _{ |t_i - \tau| < \delta } \left[ K^{\prime} ( \frac{t_{i-1} - \tau}{h} ) - K^{\prime} ( \frac{s_{i-1} - \tau}{h} ) \right] \\
	& = \int _{ ( \delta^- / h , \delta^+ / h ) \cap (A,B)} K^{\prime} (x) dx + \frac{\Delta}{h} \sum _{ |t_i - \tau| < \delta } \left[ K^{\prime} ( \frac{t_{i-1} - \tau}{h} ) - K^{\prime} ( \frac{s_{i-1} - \tau}{h} ) \right] \\
	& = (K(B-)- K(A+)) + o (1).
\end{split}
\end{equation*}
since we have
$$
\frac{\Delta}{h} \sum _{ |t_i - \tau | < \delta } \left[ K^{\prime} ( \frac{t_{i-1} - \tau}{h} ) - K^{\prime} ( \frac{s_{i-1} - \tau}{h} ) \right] \leq \frac{\Delta}{h} V_{-\infty}^{\infty} (K^{\prime}) = O(\frac{\Delta}{h}).
$$
Here we define $\delta^+ = \max \{t_i - \tau : t_i<\tau + \delta\}, \delta^- = \min \{t_{i-1} - \tau : t_i>\tau - \delta\}$. Note that since $\frac{\Delta}{h} \rightarrow 0$, we have $\frac{\delta^+}{h}\rightarrow +\infty, \frac{\delta^-}{h}\rightarrow -\infty$, so we have $\int _{ ( \delta^- / h , \delta^+ / h ) \cap (A,B)} K^{\prime} (x) dx \rightarrow K(B-)- K(A+)$. Here we notice that in the case that $K^\prime$ is not continuous in some intervals, the constant $V_{-\infty}^{\infty} (K^{\prime})$ is replaced by$ V_{-\infty}^{\infty} (K^{\prime}) + 2pL$.

Then we consider $\sum _{ |t_i - \tau | \geq \delta } \frac{\Delta}{h} | K^{\prime} ( \frac{t_{i-1} - \tau}{h} ) | $. We note here the absolute integrability of $K^{\prime} (\cdot)$ and thus have the following
\begin{equation*}
\begin{split}
\sum _{ |t_i - \tau | \geq \delta } \frac{\Delta}{h} | K^{\prime} ( \frac{t_{i-1} - \tau}{h} ) |
	\leq & \sum _{ |t_i - \tau | \geq \delta } \int _{(t_{i-1} - \tau)/h}^{(t_i - \tau)/h} | K^{\prime} ( x ) | dx 
	+ \frac{\Delta}{h} \sum _{ |t_i - \tau | \geq \delta } | | K^{\prime} ( \frac{t_{i-1} - \tau}{h} ) | - | K^{\prime} ( \frac{s_{i-1} - \tau}{h} ) | | \\
	\leq & \left( \int _{\delta^+ / h}^{+\infty} + \int _{-\infty}^{\delta^- / h}   \right) | K^{\prime} ( x ) | dx + \frac{\Delta}{h} V_{-\infty}^{\infty} (K^{\prime})
	= o(1).
\end{split}
\end{equation*}

Combining previous equations together, we get
$$
\overline{D}_{111} , \underline{D}_{111}, D_{111} 
= 
\frac{ (K(A+) - K(B-)) f(\tau) }{ 2 } \frac{\Delta}{h} + o ( \frac{\Delta}{h} ) ,\quad 
|D_{112}| 
= 
o (\frac{\Delta}{h}),
$$
and thus we have the first order approximation of $D_1$ as the following
$$
D_{1} = \frac{ (K(A+) - K(B-)) f(\tau) }{ 2 } \frac{\Delta}{h} + o(\frac{\Delta}{h}).
$$
From the previous proof, we observe that such a first order approximation is uniform for $\tau \in (0, T)$.

We then proceed to consider the general case. For simplicity of notation, we only prove the case of $m = 2$ and $m_1 = m_2 = 1$, i.e., $K_i$ is $C((A_i,B_i))$ and piecewise $C^1((A_i, B_i))$, $i = 1,2$. The proof of general case is direct generalization of this proof. Since we have already prove the case of $m=1$, we will only briefly outline some calculations.

\begin{equation*}
\begin{split}
D_1 
	= & 
	\sum _{i=1}^n \sum _{j=1}^n \left[ K_{1h}(t_{i-1} - \tau) K_{2h}(s_{j-1} - \tau) \int _{t_{i-1}}^{t_i} \int _{s_{j-1}}^{s_j} f(t,s) dtds - \int _{t_{i-1}}^{t_i} \int _{s_{j-1}}^{s_j} K_{1h}(t - \tau) K_{2h}(s - \tau) f(t,s) dtds \right]  \\
	= & 
	\sum _{i=1}^n \sum _{j=1}^n \int _{t_{i-1}}^{t_i} \int _{s_{j-1}}^{s_j} [ K_{1h}(t_{i-1} - \tau) K_{2h}(s_{j-1} - \tau) - K_{1h}(t - \tau) K_{2h}(s - \tau) ] f(t,s) dtds \\
	= & 
	\frac{1}{h^2} \sum _{i=1}^n \sum _{j=1}^n \int _{t_{i-1}}^{t_i} \int _{s_{j-1}}^{s_j} [ K_1( \frac{t_{i-1} - \tau}{h} ) K_2( \frac{s_{j-1} - \tau}{h} ) - K_1( \frac{t - \tau}{h} ) K_2( \frac{s - \tau}{h} ) ] f(t,s) dtds \\
	= & 
	\frac{1}{h^2} \sum _{i=1}^n \sum _{j=1}^n \int _{t_{i-1}}^{t_i} \int _{s_{j-1}}^{s_j} [ ( K_1( \frac{t_{i-1} - \tau}{h} ) - K_1( \frac{t - \tau}{h} ) ) K_2( \frac{s_{j-1} - \tau}{h} ) + K_1( \frac{t - \tau}{h} ) ( K_2( \frac{s_{j-1} - \tau}{h} ) - K_2( \frac{s - \tau}{h} ) ) ] f(t,s) dtds \\
	= & 
	\frac{1}{h^2} \sum _{i=1}^n \sum _{j=1}^n \int _{t_{i-1}}^{t_i} \int _{s_{j-1}}^{s_j} [ K_1^\prime ( \frac{t_{i-1} - \tau}{h} ) K_2( \frac{s_{j-1} - \tau}{h} ) \frac{t_{i-1} - t}{h} + K_1 ( \frac{t_{i-1} - \tau}{h} ) K_2^\prime( \frac{s_{j-1} - \tau}{h} ) \frac{t_{i-1} - t}{h} ] f(t,s) dtds + o(\frac{\Delta}{h})\\
	= & 
	\frac{ f(\tau,\tau) [ (K_1(A_1+) - K_1(B_1+)) \int_{A_2}^{B_2} K_2(x) dx + (K_2(A_2+) - K_2(B_2+)) \int_{A_1}^{B_1} K_1(x) ] }{2} \frac{\Delta}{h}
	 + o(\frac{\Delta}{h}).
\end{split}
\end{equation*}
Finally, if we notice the additivity, we get our final conclusion.

\subsubsection*{Proof of Lemma \ref{D2Lemma}}

We will only prove the case of $m=2$ for simplicity of notation. The proof of general case is quite clear from the proof below. 

For any $\epsilon \in (0,\min (\tau , T - \tau))$, define $S_\epsilon = \{|t - \tau|, |s - \tau| \leq \epsilon \}$ and we divide the integration into two parts:
\begin{equation*}
\begin{split}
\int_0^T \int_0^T K_h(t - \tau) K_h(s - \tau) f(t,s)dtds
= &
\left( \int_{S_\epsilon} + \int_{ [0,T]^2 - S_\epsilon } \right) K(\frac{t - \tau}{h}) K(\frac{s - \tau}{h}) \frac{1}{h^2} f(t,s)dtds 
\triangleq 
A_{\epsilon} + B_{\epsilon}
\end{split}
\end{equation*}
Then, we have
$$
| B_{\epsilon} | 
\leq 
2 M_f \int |K(t)| dt \int_{ |s| > \epsilon / h} |K(s)| ds.
$$
As to $A_\epsilon$, for all $\delta > 0$, there exists $\epsilon>0$ small enough, s.t. $\forall |t - \tau|, |s - \tau| < \epsilon$
$$
C(t,s;\tau) - \delta (t^2 + s^2)^{\gamma / 2}
\leq f(\tau + t, \tau + s) - f(\tau,\tau) 
\leq
C(t,s;\tau) + \delta (t^2 + s^2)^{\gamma / 2}.
$$
Note that if we assume that $f(\tau + t, \tau + s) - f(\tau,\tau) = C(t,s;\tau) + o((t^2 + s^2)^{\gamma / 2})$ uniformly over $\tau \in (0, T)$, then for $\delta > 0$, the $\epsilon$ can be picked such that the above holds for all $\tau \in (0, T)$.
With this set up, we can get upper bound of $A_{\epsilon} - f(t)$ as the following:
\begin{equation*}
\begin{split}
     &\quad A_{\epsilon} - f(\tau,\tau) \left( \int K(s)ds \right) ^2 \\
& =    
	\int_{S_\epsilon} K(\frac{t - \tau}{h})K(\frac{s - \tau}{h})\frac{1}{h^2} (f(t,s) - f(\tau,\tau)) dtds 
	 - f(\tau,\tau) \int_{\mathbb{R}^2 - S_\epsilon} K(\frac{t - \tau}{h}) K(\frac{s - \tau}{h})\frac{1}{h^2} dtds \\
& = 
	\left( \int_{ S_\epsilon \cap \{ K(\frac{t - \tau}{h}) K(\frac{s - \tau}{h}) \geq 0 \} } + \int_{ S_\epsilon \cap \{ K(\frac{t - \tau}{h}) K(\frac{s - \tau}{h}) < 0 \} } \right) 
	K(\frac{t - \tau}{h}) K(\frac{s - \tau}{h}) \frac{1}{h^2} (f(t,s) - f(\tau,\tau)) dtds \\
	 &\quad - f(\tau,\tau) \int_{\mathbb{R}^2 - S_\epsilon} K(\frac{t - \tau}{h}) K(\frac{s - \tau}{h})\frac{1}{h^2} dtds \\
& \leq
	\int_{ S_\epsilon \cap \{ K(\frac{t - \tau}{h}) K(\frac{s - \tau}{h}) \geq 0 \} } 
	K(\frac{t - \tau}{h}) K(\frac{s - \tau}{h}) \frac{1}{h^2} [ C(t-\tau,s-\tau ;\tau) + \delta (t^2 + s^2)^{\gamma / 2}] dtds \\
	&\quad + \int_{ S_\epsilon \cap \{ K(\frac{t - \tau}{h}) K(\frac{s - \tau}{h}) < 0 \} } 
     K(\frac{t - \tau}{h}) K(\frac{s - \tau}{h}) \frac{1}{h^2} [ C(t-\tau,s-\tau ;\tau) - \delta (t^2 + s^2)^{\gamma / 2}] dtds \\
	&\quad + 2f(\tau,\tau) \int |K(t)|dt \int_{ |s| > \epsilon / h} |K(s)| ds \\
& =
	h^{\gamma} \int_{- \epsilon / h}^{ \epsilon /h} \int_{- \epsilon / h}^{ \epsilon /h} K(t)K(s) C(t,s;\tau) dtds
+ 
	h^{\gamma} \delta \int_{- \epsilon / h}^{ \epsilon /h} \int_{- \epsilon / h}^{ \epsilon /h} |K(t)K(s)| (t^2 + s^2)^{\gamma / 2} dtds \\
&\quad + 
	2f(\tau,\tau) \int |K(t)|dt \int_{ |s| > \epsilon / h} |K(s)| ds . \\
\end{split}
\end{equation*}
Similarly, the lower bound is the following:
\begin{equation*}
\begin{split}
&\quad A_{\epsilon} - f(\tau,\tau) \left( \int K(s)ds \right) ^2 \\
& \geq h^{\gamma} \int_{- \epsilon / h}^{ \epsilon /h} \int_{- \epsilon / h}^{ \epsilon /h} K(t)K(s) C(t,s;\tau) dtds
- 
h^{\gamma} \delta \int_{- \epsilon / h}^{ \epsilon /h} \int_{- \epsilon / h}^{ \epsilon /h} |K(t)K(s)| (t^2 + s^2)^{\gamma / 2} dtds \\
&\quad - 
2f(\tau,\tau) \int |K(t)|dt \int_{ |s| > \epsilon / h} |K(s)| ds. \\
\end{split}
\end{equation*}
For any $\delta$, we can find satisfactory $\epsilon$ and we fix this two numbers. Then we let $h \rightarrow 0$. By l'Hopital rule, we have:
$$
\lim _{h\rightarrow 0} \frac{\int_{\epsilon / h }^{\infty} K(s) ds}{h^{\gamma}}
= 
\lim _{h\rightarrow 0} \frac{ \epsilon h^{-2} K(\frac{\epsilon}{h}) } {\gamma h^{\gamma - 1}} 
= 
\lim _{h\rightarrow 0} \frac{ \epsilon K(\frac{\epsilon}{h}) } {\gamma h^{\gamma + 1}} 
= 
\lim _{x\rightarrow \infty} C K(x)x^{\gamma + 1} \rightarrow 0.
$$
Therefore, we have
\begin{equation*}
\begin{split}
&\quad h^{ - \gamma} \left|  A_{\epsilon} + B_{\epsilon} - f(\tau,\tau) \left( \int K(s)ds \right)^2
-
h^{\gamma} \int_{-\epsilon/h}^{\epsilon /h} \int_{-\epsilon/h}^{\epsilon /h} K(t)K(s) C(t,s;\tau) dtds \right| \\
& \leq
\delta \iint |K(t)K(s)| (t^2 + s^2)^{\gamma / 2} dtds
+
h^{-\gamma} |M_f + 2f(\tau,\tau)| \int |K(t)|dt \int_{ |s| > \epsilon / h} |K(s)| ds. \\
\end{split}
\end{equation*}
Now let $h \rightarrow 0$, and notice that $\delta$ is arbitrary, we have $D_2(f) = h^\gamma \iint K(t)K(s) C(t,s;\tau) dtds + o(h^\gamma) $.

\subsubsection*{Proof of Theorem \ref{CstTSVV}}
The $t_b$ here is basically to rule our boundary effects and for brevity of notation, we will write $t_b = 0$ and assume we have a left side estimator near $t = 0$ and a right side estimator near $T = t$, with the same convergence rate. Define the error terms from left side estimation and right side estimation as the following:
$$
l_i = \hat{\sigma}_{l, t_i}^2 - \sigma_{t_i}^2, \quad r_i = \hat{\sigma}_{r, t_i}^2 - \sigma_{t_i}^2.
$$
We will consider a slightly different estimator as the following:
\begin{equation}\label{eq:TSRVV2}
   \widehat{IVV}_T^{\text{(tsrvv)}}= \frac{1}{k} \sum_{i = 0}^{n - k} (\Delta_i^{(k)} \hat{\sigma}^2)^2  - \frac{1}{k} \sum_{i = 0}^{n - 1} (\Delta_i \hat{\sigma}^2)^2.
\end{equation}
The two summations can be written as
\begin{equation*}
\begin{split}
\sum_{i = 0}^{n - k} (\Delta_i^{(k)} \hat{\sigma}^2)^2
& =
\sum_{i = 0}^{n - k} (\Delta_i^{(k)} \sigma^2)^2 + 2 \sum_{i = 0}^{n - k} (\sigma_{t_{i + k}}^2 - \sigma_{t_i}^2) (r_{i + k} - l_i) + \sum_{i = 0}^{n - k} l_i^2 + \sum_{i = k}^{n} r_i^2 - 2 \sum_{i = 0}^{n - k} l_i r_{i + k}, \\
\sum_{i = 0}^{n - 1} (\Delta_i \hat{\sigma}^2)^2
& =
\sum_{i = 0}^{n - 1} (\Delta_i \sigma^2)^2 + 2 \sum_{i = 0}^{n - 1} (\sigma_{t_{i + 1}}^2 - \sigma_{t_i}^2) (r_{i + 1} - l_i) + \sum_{i = 0}^{n - 1} l_i^2 + \sum_{i = 1}^{n} r_i^2 - 2 \sum_{i = 0}^{n - 1} l_i r_{i + 1}. \\
\end{split}
\end{equation*}
Putting them together, we get
\begin{equation*}
\begin{split}
\widehat{IVV}_T^{\text{(tsrvv)}}
& = 
\frac{1}{k} \left[ 
\sum_{i = 0}^{n - k} (\Delta_i^{(k)} \sigma^2)^2
-
\sum_{i = 0}^{n - 1} (\Delta_i \sigma^2)^2
+
2 \sum_{i = n - k + 1}^{n - 1} \sigma_{t_i}^2 l_i
-
2 \sum_{i = 1}^{k - 1} \sigma_{t_i}^2 r_i \right. \\
&\quad +
2 \sum_{i = k}^{n - 1} (\sigma_{t_{i}}^2 - \sigma_{t_{i - k + 1}}^2) r_{i + 1}
-
2 \sum_{i = 0}^{n - k} (\sigma_{t_{i + k}}^2 - \sigma_{t_{i + 1}}^2) l_i
+
2 \sum_{i = 0}^{k - 1} \sigma_{t_{i}}^2 r_{i + 1}
-
2 \sum_{i = n - k + 1}^{n - 1} \sigma_{t_{i + 1}}^2 l_i \\
&\quad \left. -
\sum_{i = n - k + 1}^{n - 1} l^2_{i}
-
\sum_{i = 1}^{k - 1} r^2_{i}
-
2 \sum_{i = 0}^{n - k} l_i r_{i + k}
+
2 \sum_{i = 0}^{n - 1} l_i r_{i + 1} \right].
\end{split}
\end{equation*}
Now we consider the convergence rate separately, and for each pair of similar terms, we consider only one of them and the other one has the same convergence rate. Indeed, with some additional assumptions, we have
\begin{equation*}
\begin{split}
& \frac{1}{k} \mathbb{E} \left| \sum_{i = k}^{n - 1} (\sigma_{t_{i}}^2 - \sigma_{t_{i - k + 1}}^2) r_{i + 1} \right|
\leq
\frac{1}{k} \sqrt{ \sum_{i = 0}^{n - k} \mathbb{E} [(\sigma_{t_{i + k}}^2 - \sigma_{t_{i + 1}}^2)^2] \sum_{i = 0}^{n - k} \mathbb{E}(r^2_{i + 1}) }
=
O(\frac{n^{1/4}}{k^{1/2}}), \\
& \frac{1}{k} \mathbb{E} \left| \sum_{i = n - k + 1}^{n - 1} \sigma_{t_i}^2 l_i \right|
\leq
\frac{1}{k} \sqrt{\sum_{i = n - k + 1}^{n - 1} \mathbb{E} (\sigma_{t_i}^4) \sum_{i = n - k + 1}^{n - 1} \mathbb{E} (l_i^2)}
 = O(\frac{1}{n^{1/4}}), \\
& \frac{1}{k} \mathbb{E} \left| \sum_{i = n - k + 1}^{n - 1} l^2_{i} \right| 
= 
O (\frac{1}{\sqrt{n}}), \quad
\frac{1}{k} \mathbb{E} \left| \sum_{i = 0}^{n - k} l_i r_{i + k} \right| 
\leq \frac{1}{k} \sqrt{ \sum_{i = 0}^{n - k} \mathbb{E} (l^2_i) \sum_{i = 0}^{n - k} \mathbb{E} (r^2_{i + k}) } = O (\frac{\sqrt{n}}{k}).
\end{split}
\end{equation*}
Similarly, we can see that the difference between \eqref{eq:TSRVV} and \eqref{eq:TSRVV2} is $O_p(\Delta)$. Putting all these together, we get
\begin{equation}\label{eq:TSRVV_first_part_error}
\mbox{TSRVV} 
- 
\sum_{i = 0}^{n - k} (\Delta_i^{(k)} \sigma^2)^2
-
\sum_{i = 0}^{n - 1} (\Delta_i \sigma^2)^2 = O_p(\frac{n^{1/4}}{k^{1/2}}).
\end{equation}

With similar assumptions and proofs as Theorem 2 and 3 of \cite{Two_Time_Scale}, we have the following:
\begin{equation}\label{eq:TSRVV_second_part_error}
\frac{1}{k} \left[ 
\sum_{i = 0}^{n - k} (\Delta_i^{(k)} \sigma^2)^2
-
\sum_{i = 0}^{n - 1} (\Delta_i \sigma^2)^2
\right]
-
\int_0^T \Lambda_t^2 dt
=
O_p(\sqrt{\frac{k}{n}}).
\end{equation}
Therefore, we have
$$
\mbox{TSRVV} - \int_0^T \Lambda_t^2 dt = O_p(\frac{n^{1/4}}{k^{1/2}}) + O_p(\sqrt{\frac{k}{n}})
$$
which yields that the optimal $k$ is given by $Cn^{3/4}$, in which case the convergence rate is $n^{-1/8}$.

\section{{Other Technical Proofs}}\label{More_Technical_Details}

\subsection*{Proof of (\ref{MSE_Deduction1})}
First, we have the following:
\begin{equation*}
\begin{split}
(\Delta_i X)^2 - \Delta \sigma^2_{\tau}
=
\left( \int_{t_{i-1}}^{t_i} \mu_t dt \right) ^2 
+ 
2 \int_{t_{i-1}}^{t_i} \mu_t dt \int_{t_{i-1}}^{t_i} \sigma_t dB_t 
+ 
\left( \left( \int_{t_{i-1}}^{t_i} \sigma_t dB_t \right) ^2
- 
\Delta \sigma^2_{\tau} \right).
\end{split}
\end{equation*}
Now, we demonstrate that all the term involving $\mu$ will be of $o(\frac{\Delta}{h})$ and thus the third equality is true. Here we take two terms for example. Under Assumption \ref{IndependentCondition}, we can condition on $\mathscr{F}^{\mu,\sigma} = \sigma(\mu_t,\sigma_t: t \geq 0)$ to get:
\begin{equation*}
\begin{split}
& \sum_{i=1}^n \sum_{j=1}^n K_h(t_{i-1} - \tau) K_h(t_{j-1} - \tau) \mathbb{E} 
\left[ 
\int_{t_{i-1}}^{t_i} \mu_t dt \int_{t_{i-1}}^{t_i} \sigma_t dB_t  
\left( \left( \int_{t_{j-1}}^{t_j} \sigma_t dB_t \right) ^2
- 
\Delta \sigma^2_{\tau} \right)
\right] \\
= &
\sum_{i=1}^n \sum_{j=1}^n K_h(t_{i-1} - \tau) K_h(t_{j-1} - \tau) \mathbb{E} 
\left[ 
\int_{t_{i-1}}^{t_i} \mu_t dt 
\int_{t_{i-1}}^{t_i} \sigma_t dB_t
\left( \int_{t_{j-1}}^{t_j} \sigma_t dB_t \right) ^2
-
\int_{t_{i-1}}^{t_i} \mu_t dt \int_{t_{i-1}}^{t_i} \sigma_t dB_t \Delta \sigma_\tau^2
\right]
= 0 .
\end{split}
\end{equation*}
For another term, we have
\begin{equation*}
\begin{split}
\sum_{i \neq j} K_h(t_{i-1} - \tau) K_h(t_{j-1} - \tau) 
\mathbb{E} 
\left[ 
\int_{t_{i-1}}^{t_i} \mu_t dt \int_{t_{i-1}}^{t_i} \sigma_t dB_t  
\int_{t_{j-1}}^{t_j} \mu_t dt \int_{t_{j-1}}^{t_j} \sigma_t dB_t  
\right]
= 0 ,
\end{split}
\end{equation*}
and
\begin{equation*}
\begin{split}
& \left|
\sum_{i=1}^n K^2_h(t_{i-1} - \tau) 
\mathbb{E} 
\left[ 
\left(
\int_{t_{i-1}}^{t_i} \mu_t dt \int_{t_{i-1}}^{t_i} \sigma_t dB_t  
\right)^2
\right] \right| 
\leq
\sum_{i=1}^n K^2_h(t_{i-1} - \tau) \Delta
\mathbb{E} 
\left[ 
\int_{t_{i-1}}^{t_i} \mu^2_t dt \int_{t_{i-1}}^{t_i} \sigma^2_t dt 
\right] \\
& 
\leq
\sum_{i=1}^n K^2_h(t_{i-1} - \tau) \Delta
\sqrt{ \mathbb{E} \left[ \int_{t_{i-1}}^{t_i} \mu^2_t dt \right]^2 }
\sqrt{ \mathbb{E} \left[ \int_{t_{i-1}}^{t_i} \sigma^2_t dt \right]^2 }
\leq
\sum_{i=1}^n K^2_h(t_{i-1} - \tau) \Delta^2
\sqrt{ \mathbb{E} \left[ \int_{t_{i-1}}^{t_i} \mu^4_t dt \right] }
\sqrt{ \mathbb{E} \left[ \int_{t_{i-1}}^{t_i} \sigma^4_t dt \right] } \\ 
& \leq
M_T \sum_{i=1}^n K^2_h(t_{i-1} - \tau)
\Delta^3
= O(\Delta) = o(\frac{\Delta}{h})
\end{split}
\end{equation*}
where $M_T$ is given in Assumption \ref{Boundedness_Condition}. Combining the two above together, we get
\begin{equation*}
\begin{split}
\sum_{i = 1}^n \sum_{j = 1}^n K_h(t_{i-1} - \tau) K_h(t_{j-1} - \tau) 
\mathbb{E} 
\left[ 
\int_{t_{i-1}}^{t_i} \mu_t dt \int_{t_{i-1}}^{t_i} \sigma_t dB_t  
\int_{t_{j-1}}^{t_j} \mu_t dt \int_{t_{j-1}}^{t_j} \sigma_t dB_t  
\right]
= o(\frac{\Delta}{h}).
\end{split}
\end{equation*}
Using similar technique, we can prove the following:
\begin{equation*}
\begin{split}
\sum_{i = 1}^n \sum_{j = 1}^n K_h(t_{i-1} - \tau) K_h(t_{j-1} - \tau) 
\mathbb{E} 
\left[ 
\left( \int_{t_{i-1}}^{t_i} \mu_t dt \right)^2
\left( \int_{t_{j-1}}^{t_j} \mu_t dt \right)^2
\right]
& = o(\frac{\Delta}{h}), \\
\sum_{i = 1}^n \sum_{j = 1}^n K_h(t_{i-1} - \tau) K_h(t_{j-1} - \tau) 
\mathbb{E} 
\left[ 
\left( \int_{t_{i-1}}^{t_i} \mu_t dt \right)^2
\int_{t_{j-1}}^{t_j} \mu_t dt \int_{t_{j-1}}^{t_j} \sigma_t dB_t
\right]
& = 0, \\
\sum_{i = 1}^n \sum_{j = 1}^n K_h(t_{i-1} - \tau) K_h(t_{j-1} - \tau) 
\mathbb{E} 
\left[ 
\left( \int_{t_{i-1}}^{t_i} \mu_t dt \right)^2
\left( \left( \int_{t_{j-1}}^{t_j} \sigma_t dB_t \right) ^2
- 
\Delta \sigma^2_{\tau} \right) 
\right]
& = o(\frac{\Delta}{h}).
\end{split}
\end{equation*}
With all these above, we finish the proof.

\subsection*{Proof of Proposition \ref{Finite_Sample_Property}}
For notation simplicity, we shall write $\Delta = T/n_k$ and $h = h_{n_k}$. Since we already know $\lim_{k \rightarrow \infty} \frac{\Delta}{h} \left( = \lim_{k \rightarrow \infty} \frac{T}{n_k h_{n_k}} \right) = 0$ and $\lim_{k \rightarrow \infty} h \left( = \lim_{k \rightarrow \infty} h_{n_k} \right) = 0$ are sufficient for the convergence of the $\mbox{MSE}^*(n,h)$ from Corollary \ref{Corollary_Convergence_of_Estimator}, we only need to prove that the convergence fails in other situations. Note that it is enough to consider the case when both the limit of $h$ and $\frac{\Delta}{h}$ exists (including convergence to infinity), since otherwise we can always choose a subsequence with the same limit of the true MSE. In what follows, we will prove that the true MSE cannot converge to zero in the following cases: (1) $h \rightarrow \infty$, (2) $\Delta \nrightarrow 0$, (3) $\Delta \rightarrow 0$ and $h \rightarrow h_0 > 0$, (4) $\Delta \rightarrow 0$, $h \rightarrow 0$ and $\frac{\Delta}{h} \rightarrow \alpha_0 > 0$.

Firstly, we prove that $h$ cannot converge to infinity. To this end, we observe the following inequality:
\begin{equation}\label{Lower_Bound_By_Expectation}
\begin{split}
\mathbb{E}[(\hat{\sigma}^2_{\tau} - \sigma^2_{\tau} )^2] 
& \geq
(\mathbb{E}[\hat{\sigma}^2_{\tau}] - \mathbb{E}[\sigma^2_{\tau}] )^2 \\&
=
\left[
\sum _{i=1}^n K_h(t_{i-1} - \tau) \left( \mathbb{E} \left( \int_{t_{i-1}}^{t_i} \mu_t dt \right)^2 + \int_{t_{i-1}}^{t_i} \mathbb{E}[\sigma_t^2] dt \right) - \mathbb{E}[\sigma_{\tau}^2]
\right]^2 .
\end{split}
\end{equation}
by Assumption \ref{IndependentCondition}. If $h \rightarrow \infty$, then
\begin{equation*}
\begin{split}
& \left|
\sum _{i=1}^n K_h(t_{i-1} - \tau) \left( \mathbb{E} \left( \int_{t_{i-1}}^{t_i} \mu_t dt \right)^2 + \int_{t_{i-1}}^{t_i} \mathbb{E}[\sigma_t^2] dt \right)
\right| \\&\quad
\leq
\sum _{i=1}^n \left| K_h(t_{i-1} - \tau) \right| \left( \Delta \int_{t_{i-1}}^{t_i} \mathbb{E} [\mu^2_t] dt + \int_{t_{i-1}}^{t_i} \mathbb{E}[\sigma_t^2] dt \right) \\&\quad 
\leq
\frac{1}{h} nM_K (\Delta^2 + \Delta) M_T
\leq
\frac{1}{h} M_K T(T + 1) M_T \rightarrow 0, \quad h \rightarrow 0 ,
\end{split}
\end{equation*}
where $M_T$ is defined in Assumption \ref{Boundedness_Condition} and $M_K$ is such that $|K(x)| < M_K$, for all $x \in \mathbb{R}$, whose existence is guaranteed by Assumption \ref{AdmissibleKernel}. Therefore, the R.H.S of (\ref{Lower_Bound_By_Expectation}) converges to $(\mathbb{E}[\sigma_\tau^2])^2 > 0$ if $h \rightarrow \infty$. We are now able to conclude that we only need to consider that $h$ converges to a finite limit.

Next, we prove that $\Delta \rightarrow 0$ must hold. First, assume that $h \rightarrow h_0 > 0$. If we do not have $\Delta \rightarrow 0$, since $n_k$ can only take integer values, it is enough to consider the case that $n_k$ and, thus, $\Delta$, are fixed. In such a case, we have the following for $k$ large enough:
\begin{equation*}
\begin{split}
\left(\sum_{i=1}^n K_h(t_{i-1} - \tau) (\Delta_i X)^2 - \sigma^2_{\tau} \right)^2
& \leq
2\left( \sum_{i=1}^n K_h(t_{i-1} - \tau) (\Delta_i X)^2 \right)^2 + 2 \sigma^4_{\tau} \\
& \leq
\frac{2M_K^2}{h_0^2}n\sum_{i = 0}^n (\Delta _i X)^4 + 2\sigma_\tau^4 ,
\end{split}
\end{equation*}
Note that $h_0$ and $n$ are fixed and $(\Delta_i X)^4$ and $\sigma_\tau ^4$ have finite expectations by Assumption \ref{Boundedness_Condition}. Therefore, we can implement Dominate Convergence Theorem to conclude that
$$
\liminf_{h \rightarrow h_0} \mbox{MSE}^*_n (h)
=
\mathbb{E} \left[  \left(\sum_{i=1}^n \liminf_{h \rightarrow h_0} K_h(t_{i-1} - \tau) (\Delta_i X)^2 - \sigma^2_{\tau} \right)^2 \right] .
$$
This equals to zero if and only if
$$
\sum_{i=1}^n \alpha_i (\Delta_i X)^2 - \sigma^2_{\tau} = 0, \mbox{ a.s.},
$$
for some $\alpha_i \in \mathbb{R}$, which is not possible by Lemma \ref{Lemma_of_NonPredictable_And_Positive_Definiteness}.

We now analyze the case of $h \rightarrow 0$ and $\Delta \nrightarrow 0$, where we may still assume a fixed $\Delta$. Consider (\ref{Lower_Bound_By_Expectation}) again. From Assumption \ref{AdmissibleKernel}, we know that if $t_{i - 1} \neq \tau$, we have 
$$
\lim_{h \rightarrow 0} K_h(t_{i - 1} - \tau) 
=
\frac{1}{t_{i - 1} - \tau} \lim_{x \rightarrow \infty} x K(x) 
= 0
.$$
Therefore, if there exists $i_0$ such that $t_{i_0} = \tau$, then $\mathbb{E}[\hat{\sigma}^2_{\tau}]$ converges to either infinity or zero, depending on if $K(0) \neq 0$ or $K(0) = 0$. Otherwise, $\mathbb{E}[\hat{\sigma}^2_{\tau}]$ always converges to zero. In both cases, we have that the true MSE of the kernel estimator does not converge to zero and, therefore, it must be true that $\Delta \rightarrow 0$.

Next, we prove that it is not possible that $\Delta \rightarrow 0$ but $h \rightarrow h_0 > 0$. Using similar arguments as the proof of Theorem \ref{Approximated_MSE_Theorem}, we have
\begin{equation*}
\begin{split}
\mathbb{E} \left(\sum_{i=1}^n K_h(t_{i-1} - \tau) (\Delta_i X)^2 - \int_0^T K_h(t - \tau) \sigma^2_{t} dt \right)^2 = o(1), \\
\mathbb{E} \left(\int_0^T K_h(t - \tau) \sigma^2_{t} dt - \int_0^T K_{h_0}(t - \tau) \sigma^2_{t} dt \right)^2 = o(1).
\end{split}
\end{equation*}
In the first equality, we use Lemma \ref{D1Lemma} and in the second equality, we notice that $K_{h_0}(t - \tau) \neq \lim K_{h}(t - \tau)$ for only finite many $t$.
By Assumption \ref{Complexity_of_Parameters}, $\mathbb{E} \left(\int_0^T K_{h_0}(t - \tau) \sigma^2_{t} dt - \sigma_\tau^2 \right)^2 \neq 0$. As a result, we have proved that the third case is not possible.

Finally, we need to consider the case that $\Delta \rightarrow 0$, $h \rightarrow 0$ and $\frac{\Delta}{h} \rightarrow \alpha_0 > 0$. We notice that
\begin{equation*}
\begin{split}
Cov((\Delta_i X)^2, (\Delta_j X)^2 | \sigma (\mu, \sigma))
 =
0, \quad
Cov((\Delta_i X)^2, (\Delta_i X)^2 | \sigma (\mu, \sigma))
 =
2\left[\int_{t_{i-1}}^{t_i} \sigma_t^2 dt\right]^2,
\end{split}
\end{equation*}
Thus, we have
\begin{equation*}
\begin{split}
 Var\left[ \sum _{i=1}^n K_h(t_{i-1} - \tau) (\Delta_i X)^2 | \sigma (\mu, \sigma) \right]
& =
\sum _{i=1}^n \sum _{j=1}^n K_h(t_{i-1} - \tau) K_h(t_{j-1} - \tau) Cov((\Delta_i X)^2, (\Delta_j X)^2 | \sigma (\mu, \sigma)) \\
& =
2 \sum _{i=1}^n K^2_h(t_{i-1} - \tau) \left[ \int_{t_{i-1}}^{t_i} \sigma_t^2 dt \right]^2 
.
\end{split}
\end{equation*}
Above, there are two possibilities. The first one is that $
\lim_{n \rightarrow \infty} \sum _{i=1}^n K^2 \left( \frac{t_{i-1} - \tau}{h} \right) = 0$, which implies $K \frac{t_{i-1} - \tau}{h} \rightarrow 0$ and thus, by Dominate Convergence Theorem, the estimator $\hat{\sigma}_\tau^2$ converges to zero in probability. 
The second case is that $
\lim_{n \rightarrow \infty} \sum _{i=1}^n K^2 \left( \frac{t_{i-1} - \tau}{h} \right) > 0$, in which case $\sum _{i=1}^n K^2_h(t_{i-1} - \tau) \left[ \int_{t_{i-1}}^{t_i} \sigma_t^2 dt \right]^2$ is bounded away from zero and thus the conditional variance is not zero. In both cases, the estimator does not converge to the true spot volatility.

\subsection*{Proof of Lemma \ref{Approximation_Bandwidth_Lemma}}
Fix an arbitrary $\epsilon > 0$. Then, there exists $\delta > 0$, such that 
$$(1 - \epsilon) f(x,y) < F(x,y) < (1 + \epsilon) f(x,y), \mbox{ for all } (x,y) \in (0, \delta) \times (0, \delta).$$
Let $m > 0$ be such that
$$F(x,y) > m, \mbox{ for all } (x,y) \in \mathbb{R}_+ \times \mathbb{R}_+ - (0,\delta) \times (0,\delta).$$
and let $I_n(\delta) := \left\lbrace y \in \mathbb{R}_+ : \left( \frac{z_n}{y}, y \right) \in (0, \delta) \times (0, \delta) \right\rbrace$.
Notice that $z_n \searrow 0$, so there exists $N \in \mathbb{N}_+$, such that for all $n > N$,
\begin{equation*}
\begin{split}
\max \left( 
\left(\frac{B\gamma z_n^\gamma}{A}\right) ^{1/(\gamma + 1)}, 
\left(\frac{A z_n}{B\gamma}\right)^{1/(\gamma + 1)} 
\right) < \delta,
\\
(1 + \epsilon) \left[
\left(\gamma A^\gamma B z_n^\gamma \right)^{1/(\gamma + 1)}
+  
\left(A^\gamma B z_n^\gamma / \gamma^\gamma \right)^{1/(\gamma + 1)}
\right]
< m .
\end{split}
\end{equation*}
These implies that for all $n > N$, we have
\begin{equation*}
\begin{split}
& \inf_{y \in I_n(\delta)} F(z_n / y, y) 
\leq 
\min_{y \in I_n(\delta)} f(z_n / y, y)(1+\epsilon) 
=
f\left( 
\left(\frac{B\gamma z_n^\gamma}{A}\right) ^{1/(\gamma + 1)}, 
\left(\frac{A z_n}{B\gamma}\right)^{1/(\gamma + 1)}  
\right) (1+\epsilon) \\&
= 
(1 + \epsilon) \left[
\left(\gamma A^\gamma B z_n^\gamma \right)^{1/(\gamma + 1)}
+  
\left(A^\gamma B z_n^\gamma / \gamma^\gamma \right)^{1/(\gamma + 1)}
\right]
< m .
\end{split}
\end{equation*}
Combining these three inequalities, we have that for $n > N$,
\begin{equation*}
\begin{split}
\inf_{y \in \mathbb{R}_+} F(z_n / y, y) 
= 
\min \left( \inf_{y \in I_n(\delta)} F(z_n / y, y) , \inf_{y \in I_n(\delta)^C} F(z_n / y, y) \right)
= \inf_{y \in I_n(\delta)} F(z_n / y, y)
< m .
\end{split}
\end{equation*}
Without loss of generality, we may assume that the above holds for all $n$.
Now we define 
$y_n = \left(\frac{A z_n}{B\gamma}\right)^{1/(\gamma + 1)}$ and we have
\begin{equation*}
\begin{split}
(1 - 2\epsilon)f(z_n / y_n, y_n) 
\leq \inf_{y \in \mathbb{R}_+} F(z_n/y, y)
\leq F(z_n / y_n, y_n)
< (1 + 2\epsilon)f(z_n / y_n, y_n) .
\end{split}
\end{equation*} 
Therefore, we have 
$$\inf_{y \in \mathbb{R}_+} F(z_n/y, y) \rightarrow 0 
, \quad
F(z_n / y_n, y_n) = \inf_{y \in \mathbb{R}_+} F(z_n/y, y) + o(\inf_{y \in \mathbb{R}_+} F(z_n/y, y)), \quad n \rightarrow \infty . $$

Then, by definition of $y_n^*$, there exists $y^{**}_n$ such that $y^*_n/y^{**}_n \rightarrow 1$ and the following holds:
\begin{equation*}
\begin{split}
(1 - 2\epsilon)f(z_n / y_n, y_n) 
\leq \inf_{y \in \mathbb{R}_+} F(z_n/y, y)
\leq F(z_n / y^{**}_n, y^{**}_n)
< (1 + 2\epsilon)f(z_n / y_n, y_n).
\end{split}
\end{equation*} 
The existence of such $y^{**}_n$ is guaranteed by $\inf_{y \in \mathbb{R}_+} F(z_n/y, y) \leq (1 + \epsilon)f(z_n / y_n, y_n) < (1 + 2\epsilon)f(z_n / y_n, y_n)$ and the fact that $\{y_{np}^*:p \in \mathbb{N}\} \cap \{y : F(z_n/y, y) < (1 + 2\epsilon)f(z_n / y_n, y_n)\} \cap \left( y_{n}^* \left( 1 - \frac{1}{n} \right), y_{n}^* \left( 1 + \frac{1}{n} \right) \right)$ is not empty.

We claim that the inequalities above imply
$$
f(z_n / y^{**}_n, y^{**}_n) 
<
\alpha f(z_n / y_n, y_n),
$$
where $\alpha = \frac{1 + 2\epsilon}{1 - 2\epsilon}$. Otherwise, we will have
$$
F(z_n / y^{**}_n, y^{**}_n) 
> 
(1 - 2\epsilon)f(z_n / y^{**}_n, y^{**}_n) 
> 
(1 + 2\epsilon) f(z_n / y_n, y_n) 
>
F(z_n / y^{**}_n, y^{**}_n) ,
$$
which is a contradiction. Since this is true for all $\epsilon > 0$, we then have $\lim _{n \rightarrow \infty} y_n / y^{**}_n = 1$, which implies $\lim _{n \rightarrow \infty} y_n / y^{*}_n = 1$. This completes the proof.

\bibliographystyle{plain}

%

\nocite{fan2008spot}
\nocite{kristensen2010nonparametric}

\nocite{durrett2010probability}
\nocite{karatzas2012brownian}
\nocite{oksendal2003stochastic}

\nocite{positive_definite_kernel}

\nocite{Two_Time_Scale}
\nocite{Multi_Time_Scale}
\nocite{barndorff2008designing}
\nocite{mancini2001disentangling}
\nocite{mancini2004estimation}
\nocite{figueroa2013optimally}

\nocite{mancini2015estimation}

\nocite{van1985mean_optKernel}
\nocite{foster1994continuous_optKernel}
\nocite{epanechnikov1969non}

\nocite{rosenblatt1956remarks}
\nocite{nadaraya1964estimating}
\nocite{watson1964smooth}

\nocite{park1992practical}
\nocite{cao1994comparative}
\nocite{jones1996brief}

\nocite{hall1983large}
\nocite{park1990comparison}

\end{document}